\documentclass[final,3p,times]{elsarticle}
\usepackage{amssymb}
\usepackage{amsmath}
\usepackage{amsthm}
\usepackage{units}
\usepackage{epstopdf}
\usepackage{xcolor,multicol,booktabs,calligra}
\usepackage{tabularx}
\usepackage{threeparttable}
\usepackage{ragged2e}
\usepackage{tabulary}
\usepackage{longtable}
\usepackage{graphics}
\usepackage{algorithm}
\usepackage{algpseudocode}
\usepackage{subcaption}
\usepackage{enumitem}

\allowdisplaybreaks

\numberwithin{equation}{section}
\newtheorem{proposition}{Proposition}[section]
\newtheorem{lemma}{Lemma}[section]

\newtheorem{theorem}{Theorem}[section]

\usepackage[colorlinks,
            linkcolor=green!88!black,        %%
            anchorcolor=green!55!black,     %%
            citecolor=green!66!black,       %%
            ]{hyperref}

\begin{document}

\begin{frontmatter}


\title{iSMART: An Iterative Sampling-and-Regression Technique for Solving Martingale-Based PDEs}
%\tnotetext[label1]{This work was supported by National Key R\&D Program of China under grant 2021YFA1003301, the National Science Foundation of China under grant 12288101, and the Peking University Boya Postdoctoral Fellowship.}

\author{Tiejun Li$^{1,2,3}$, Xiaoguang Li$^{4}$, Fugui Ma$^{1~*}$\\[10pt]
\it{\small$^1$ LMAM and School of Mathematical Sciences, Peking University, Beijing 100871, China.}\\
\it{\small$^2$ Center for Machine Learning Research, Peking University, Beijing 100871, China.}\\
\it{\small$^3$ National Engineering Laboratory for Big Data Analysis and Applications, Peking University, Beijing 100871, China.}\\
\it{\small$^4$ MOE-LCSM, School of Mathematics and Statistics, Hunan Normal University, Changsha 410081, China.}\\
}
\date{}
\cortext[cor1]{Email addresses: tieli@pku.edu.cn (T. Li), lixiaoguang@hunnu.edu.cn (X. Li), and  mafugui@math.pku.edu.cn (F. Ma).}
%\maketitle\S

%% Abstract
\begin{abstract}
We propose the {\bf i}terative {\bf S}a{\bf M}pling-{\bf A}nd-{\bf R}egression {\bf T}echnique (iSMART) for high-dimensional martingale-based  partial differential equations (PDEs) in this paper. By leveraging the $L^2$-projection property of conditional expectation and adopting the stop-gradient technique, iSMART reformulates the continuous martingale condition derived from PDEs into a sequence of tractable sampling-regression problems within an iterative framework. This approach relies solely on standard SDE path simulation and plain squared-error loss minimization, completely bypassing the need for adversarial optimization or nested expectation estimation in previous methods.  iSMART accommodates linear, semi-linear, and fully nonlinear martingale-based PDEs within a unified iterative procedure. In particular, for fully nonlinear Hamilton-Jacobi-Bellman (HJB) equations, a freezing-and-compensating technique is introduced to strategically shift a portion of the nonlinearity into the SDE drift, thereby improving the convergence behavior of the iterations. Numerous numerical experiments on linear reaction-diffusion equations with sharp gradients, semilinear Burgers-type equations, and fully nonlinear HJB equations demonstrate the accuracy, efficiency, and robustness of the proposed approach in various high dimensions.
\end{abstract}


%%Research highlights Highlights should consist of 3 to 5 bullet points, each a maximum of 85 characters, including spaces.

%% Keywords
\begin{keyword}
%You are required to provide 1 to 7 keywords for indexing purposes. Keywords should be written in English. Please try to avoid keywords consisting of multiple words (using "and" or "of").

Martingale-based PDEs \sep Deep neural networks \sep Least-squares regression problems \sep Freezing-and-compensating sampling  \sep Hamilton-Jacobi-Bellman


\end{keyword}

\end{frontmatter}

% \linenumbers


\section{Introduction}
\label{sec1}
The curse of dimensionality  \cite{Bellman57,Bellman61} renders conventional grid-based methods impractical for high-dimensional PDEs. The deep neural networks (DNN), by virtue of their universal approximation property, provide a natural foundation for constructing promising alternative solvers. In this work, we propose a new DNN-based solver tailored to martingale-based PDEs, which can be demonstrated as the following terminal-value problem for $d\in\mathbb{N}_{+}$,
\begin{equation}\label{eq:PDEs}
\left\{\begin{aligned}
&\partial_tu(t,x)+\mathcal{A}u(t,x)=f\Big(t,x,u(t,x),\nabla_x u(t,x)\Big),
&& (t,x)\in[0,T)\times\mathbb{R}^{d},\\
& u(T,x)=g(x),
&& x\in\mathbb{R}^{d},
\end{aligned}\right.
\end{equation}
where the source term $f:[0,T]\times\mathbb{R}^{d}\times\mathbb{R}\times\mathbb{R}^{d}\to\mathbb{R}$ and terminal condition $g:\mathbb{R}^d\to\mathbb{R}$ are smooth functions. The second-order differential operator $\mathcal{A}$ is defined by
\begin{equation}
\mathcal{A}\,u:=\mu^{\top}\left(t,x,u,\nabla_x u\right)\nabla_x u
+\frac{1}{2}\mathrm{Tr}\left\{\sigma\sigma^{\top}\big(t,x,u\big)\nabla_x^2u\right\}.
\end{equation}
Here, the drift coefficient $\mu$ is a mapping from $[0,T]\times\mathbb{R}^{d}\times\mathbb{R}\times\mathbb{R}^{d}\to\mathbb{R}^{d}$, while the diffusion coefficient $\sigma:[0,T]\times\mathbb{R}^{d}\times\mathbb{R}\to\mathbb{R}^{d\times q}$ such that $\sigma\sigma^\top$ satisfies uniform ellipticity condition. $\mathrm{Tr}$ denotes trace operator. Under certain regularity conditions, Equation \eqref{eq:PDEs} admits unique smooth solution $u:[0,T]\times\mathbb{R}^{d}\to\mathbb{R}$ (see \cite{Friedman64}, while we will discuss it in detail in section \ref{Sec:convAnal}), so that whose spatial gradient $\nabla_x\,u(t,x)\in\mathbb{R}^{d}$ and Hessian $\nabla_x^{2}u(t,x)\in\mathbb{R}^{d\times d}$. The main equation in \eqref{eq:PDEs} involves quite a wide range of models and applications, like Black-Scholes equation \cite{Chen2026} in financial engineering; Fokker-Planck equation \cite{zhao2024epr} and McKean-Vlasov equation \cite{Han24} in physics and bio-chemistry; the coupled forward-backward PDE systems in mean-field games and Hamilton-Jacobi-Bellman (HJB) equation in stochastic optimal control \cite{Han18,Capponi23} and so forth.

Despite the different application backgrounds and different forms for $\mu,\sigma$ and $f$, the sufficiently regular solutions of Equation \eqref{eq:PDEs} share a common martingale structure. For any $(t,x)\in[0,T]\times\mathbb{R}^d$, let $\mathbf{X}_s$ ($s\in [t,T]$), satisfy the associated It\^{o} dynamics
\begin{equation}\label{eq:Ito}
\mathrm{d}\mathbf{X}_s = \mu\Bigl(s,\mathbf{X}_s,u(s,\mathbf{X}_s),\nabla_x\, u(s,\mathbf{X}_s)\Bigr)\,\mathrm{d}s+\sigma\Bigl(s,\mathbf{X}_s,u(s,\mathbf{X}_s)\Bigr)\,\mathrm{d}\mathbf{B}_s,
\qquad \mathbf{X}_t = x.
\end{equation}
A direct application of It\^{o}'s formula gives
\begin{equation}\label{eq:Itof}
u(s,\mathbf{X}_s)-\int_{t}^{s}f\Bigl(\tau,\mathbf{X}_\tau,u(\tau,\mathbf{X}_\tau),\nabla_x u(\tau,\mathbf{X}_\tau)\Bigr)\,\mathrm{d}\tau=u(t,x)+\int_{t}^{s} \Bigl(\nabla_x u(\tau,\mathbf{X}_\tau)\Bigr)^{\top}
\sigma\Bigl(\tau,\mathbf{X}_\tau,u(\tau,\mathbf{X}_\tau)\Bigr)\,\mathrm{d}\mathbf{B}_\tau.
\end{equation}
Consequently, under standard regularity and integrability conditions, the compensated process
\begin{equation}
\mathfrak{M}^{t,x}_s:=u(s,\mathbf{X}_s)-u(t,x)-\int_{t}^{s}f\Big(\tau,\mathbf{X}_{\tau},u(\tau,\mathbf{X}_\tau),\nabla_x u(\tau,\mathbf{X}_\tau)\Big)\mathrm{d}\tau,\qquad s\in[t,T],
\end{equation}
is a zero-mean local martingale (see, e.g., \cite{Pardoux92,Ma94, ELV2019}). This martingale characterization provides a unified probabilistic structure for the PDEs in \eqref{eq:PDEs} and motivates our numerical method, which approximates $u$ by a neural network and enforces the martingale property along simulated trajectories of $\mathbf{X}_{\tau}$.

Taking advantage of the universal approximation nature of DNN, over the past several years, a variety of DNN-based methods have been developed. The Deep Ritz Method \cite{E18} makes use of the variational formulation of PDE and surrogates the solution by a DNN. The Deep Galerkin Method (DGM, \cite{Sirignano18}) and Physics-Informed Neural Networks (PINNs, \cite{Raissi19}) replace the solution by DNN and reformulate the PDEs into residual minimization problems. The Weak Adversarial Networks (WANs, \cite{Zang20}) leverages the weak formulation of the PDE.  It casts the solution and test functions as a generator-discriminator pair that optimises a min-max loss function adversarially.  A parallel line of research, operator learning, aims to learn the mapping between function spaces rather than a single instance; the Fourier Neural Operator (FNO, \cite{li2021fourier}) and its variant Multi-scale variant (MscaleFNO, \cite{YouC26}) learn kernel functions in the frequency domain, with applications to highly oscillatory wave scattering problems, among others. The mesh-free nature of these DNN-based methods makes them convenient to apply to high dimensional problems with irregular domain.

A parallel route, rooted not in the PDE operator but in its probabilistic representation, has also flourished. The Feynman-Kac formula \cite{oksendal2013stochastic, ELV2019} and Pardoux-Peng theory \cite{Pardoux90, Pardoux92} establish a bridge between PDEs and Forward Backward Stochastic Differential Equations (FBSDEs), giving rise to a class of SDE-based deep learning methods. The pioneering Deep BSDE method \cite{Han18} employs neural networks to approximate the solution and its gradient along simulated FBSDE trajectories, an idea subsequently extended in a number of directions \cite{Hure20,Beck21,zhang2022fbsde,Jaemin26}, with theoretical underpinnings developed in parallel \cite{Han20,Jentzen21,Grohs23}. More recently, grounded in the probabilistic interpretation of \eqref{eq:PDEs}, DeepMartNet \cite{Cai2026a} adopts Varadhan's martingale problem formulation and trains the network by enforcing a conditional-expectation constraint along Itô diffusion paths, demonstrating success on high-dimensional Dirichlet and elliptic eigenvalue problems. Following this line, the Deep Random Difference Method (DRDM, \cite{Cai26b}) extends the martingale framework to time-dependent quasilinear parabolic PDEs via a Galerkin variational formulation, offering the notable benefit of a derivative-free, random-difference implementation.

To motivate the design of iSMART, let us briefly introduce the DeepMartNet and DRDM, as they are the most closely related predecessors to our approach. DeepMartNet \cite{Cai2023,Cai2026a} constructs its loss (see its continuous counterpart in \eqref{eq:LosDeepMartNet}) directly from the conditional-expectation condition,
\begin{equation}
\label{eq:deepmartnet_loss}
\mathrm{Loss}_{\mathrm{DeepMartNet}}(\theta):=\frac{1}{N}\sum_{i=0}^{N-1}\frac{1}{|A_i|^2}
\left[\sum_{m\in A_i}\left( u_\theta\bigl(\mathbf{X}_{i+k}^{(m)}\bigr)-u_\theta\bigl(\mathbf{X}_i^{(m)}\bigr)-\Delta t\sum_{l=0}^k w_l \left[f\left(\mathbf{X}_{i+l}^{(m)},u_\theta(\mathbf{X}_{i+l}^{(m)})\right)-v_\theta\left(\mathbf{X}_{i+l}^{(m)}\right)\right]\right)\right]^2,
\end{equation}
where $N$ is the number of time steps, $A_i$ the minibatch at time $t_i$, $|A_i|$ its size, $k$ the martingale increment step length, $w_l$ the trapezoidal-rule weights, and $v_\theta$ the network approximating the nonlinearity. This formulation is conceptually simple and avoids explicit spatial differentiation of $u_\theta$. The loss essentially computes two nested empirical expectations: an inner average over the minibatch to approximate the conditional expectation, and an outer average over time steps to accumulate the squared deviation. Its accuracy therefore hinges on the quality of the minibatch approximation to the underlying conditional expectation. Larger minibatches reduce sampling variance and may improve training stability, but at the cost of increased trajectory and network evaluations per iteration. This sampling–cost trade-off can become significant in high-dimensional problems.
On the other hand, DRDM \cite{Cai26b} considers the following minimax optimization problem:
\begin{equation}
\label{eq:drdm_minimax}
\min_{u \in \mathcal{V}}\max_{\rho \in \mathcal{T}}\,\bigl| \mathrm{Loss}_{\mathrm{DRDM}}(u,\rho) \bigr|^2, \qquad \mathrm{Loss}_{\mathrm{DRDM}}(u,\rho) := \int_0^{T-h} \mathbb{E}\Bigl[ \rho(t,\mathbf{X}_t) R(t,\mathbf{X}_t,\xi; u) \Bigr] \, \mathrm{d}t,
\end{equation}
where $\mathcal{V}$ and $\mathcal{T}$ denote candidate and test-function spaces, respectively, $\xi$ is a centered isotropic random vector, and $R(t,x,\xi;u):=(u(t+h, x+\mu h+\sigma\sqrt{h}\xi)-u(t,x))/h-f(t,x,u(t,x))$ is the random-difference residual. While DRDM avoids automatic differentiation of spatial derivatives (e.g., Hessians), its minimax nature introduces adversarial training, which requires fine tuning of learning rates for $u$ and $\rho$ to achieve training stability. Additionally, evaluating the squared objective requires two disjoint minibatches, doubling residual evaluations and sampling overhead.

To circumvent both the nested-expectation sampling bottleneck of DeepMartNet and the adversarial optimization complexities of DRDM, we propose an {\bf i}terative {\bf S}a{\bf m}pling-{\bf a}nd-{\bf R}egression {\bf T}echnique (iSMART) for solving martingale-based PDEs. The core innovation of iSMART lies in finding the solution by an iterative scheme and recasting the evaluation of conditional expectation as a sequence of local least-squares regression problems that can be trained iteratively and efficiently by employing the stop gradient manipulation. This sampling, regression and iteration paradigm ensures efficiency and stability of the proposed iSMART approach, which also avoids adversarial training and nested expectation estimation. We remark that similar idea has been adopted in machine learning community \cite{Song2023ConsistencyModel, He2025MFM, Deng2026Drifting}, and is utilized in free energy sampling via flow matching \cite{liu2026}. One purpose of this paper is to realize this paradigm in the AI for scientific computing community.

The key contributions of this work can be summarized from the following three aspects.
\begin{enumerate}
    \item \textbf{An iteration and regression framework.} We propose iSMART, a martingale-based iterative framework that replaces the explicit evaluation of the conditional expectations by the least-squares regression against pathwise samples. This formulation avoids the min-max optimization and nested-expectation structures encountered in previous SDE-based approaches, which also shares the higher order derivative free advantage. We also conduct a  partial convergence analysis to provide theoretical support for this iterative regression framework.

    \item \textbf{Flexible path-generation strategies for different PDE classes.} iSMART accommodates linear, semi-linear and fully non-linear equation. We design three kinds of path-generation strategies tailored to different classes of PDEs (referred as Method ~($I$), ~($II$) and ~($III$) in Section \ref{sec:Sampling}). For fully non-linear equations, we propose a freezing-and-compensating strategy to improve the efficiency and stability of the iteration. The flexible path generation strategies make iSMART both accurate and efficient, as shown in numerical experiments.

    \item \textbf{Systematic numerical validation across PDE classes.} We assess iSMART on three representative classes of problems: linear reaction--diffusion equations with sharp gradients, semilinear Burgers-type equations with quadratic convective nonlinearities, and fully nonlinear HJB equations arising from stochastic optimal control. We further compare iSMART with DeepMartNet on selected benchmark problems. The numerical results demonstrate favourable accuracy and computational efficiency, confirming the robustness and scalability of iSMART across a broad range of PDE types and dimensions.
\end{enumerate}

The remainder of this paper is organised as follows. Section \ref{Sec:MR} introduces the mathematical formulation of the martingale-based PDEs and the underlying probabilistic representation. Section \ref{Sec:ismart} presents the iSMART, detailing the sampling strategy, the construction of the regression problem, and its connection to existing methods. Section \ref{Sec:convAnal} provides a theoretical analysis of the convergence analysis. Section \ref{Sec:NR} presents numerical experiments on benchmark problems. Section \ref{Sec:Conclusion} concludes the paper with a summary.

%++++++++++++++++++++++++++++++++++++++++++++++++++++++++
\section{Revisiting the Martingale Representation of Solution}
\label{Sec:MR}
In this section, we present the mathematical formulation of our method, beginning with the martingale representation of the solution to \eqref{eq:PDEs} and clarifying its connection with existing martingale-based method DeepMartNet.

We recall that for each $(t,x)\in[0,T]\times\mathbb{R}^{d}$, the process $s\mapsto\mathbf{X}_s$ ($s\in[t,T]$) denotes the forward stochastic process associated with $\mathcal{A}$ starting from $\mathbf{X}_t=x$.  With the standard assumptions ensuring the applicability of Itô’s formula, the solution of Problem \eqref{eq:PDEs} fulfills
\begin{displaymath}
u\big(s,\mathbf{X}_{s}\big)
=u(t,x)
+\int_{t}^{s}f\Big(\tau,\mathbf{X}_\tau,u(\tau,\mathbf{X}_\tau),
\nabla_x\,u(\tau,\mathbf{X}_\tau)\Big)\,\mathrm{d}\tau
+\int_{t}^{s}\Big(\nabla_xu(\tau,\mathbf{X}_\tau)\Big)^{\top}
\sigma\Big(\tau,\mathbf{X}_\tau,u(\tau,\mathbf{X}_\tau)\Big)\,\mathrm{d}\mathbf{B}_\tau.
\end{displaymath}
Since the It\^{o} integral is a local martingale (see, e.g., \cite{Pardoux92,Ma94,Cai2026a,Cai26b}), the process defined by
\begin{displaymath}
\mathfrak{M}^{t,x}_s:=u\big(s,\mathbf{X}_{s}\big)-u(t,x)
-\int_{t}^{s}f\Big(\tau,\mathbf{X}_\tau,u(\tau,\mathbf{X}_\tau),
\nabla_x\,u(\tau,\mathbf{X}_\tau)\Big)\,\mathrm{d}\tau,\quad s\in[t,T],
\end{displaymath}
is a local martingale with respect to the natural filtration $\{\mathcal{F}_s\}_{s\geq t}$. Under the standard integrability condition
\begin{displaymath}
\mathbb{E}^{t,x}\left[\int_t^T\left|
\sigma^{\top}\Big(\tau,\mathbf{X}_\tau,u(\tau,\mathbf{X}_\tau)\Big)
\nabla_xu(\tau,\mathbf{X}_\tau)\right|^2\mathrm{d}\tau\right]<\infty,
\end{displaymath}
the stochastic integral is a square-integrable martingale. Hence, for any step size $h$ satisfying $0<h\leq T-t$, the increment
\begin{displaymath}
\mathfrak{M}^{t,x}_{t+h}-\mathfrak{M}^{t,x}_{t}
=u\big(t+h,\mathbf{X}_{t+h}\big)-u\big(t,\mathbf{X}_{t}\big)
  -\int_{t}^{t+h}f\Big(\tau,\mathbf{X}_{\tau},u(\tau,\mathbf{X}_{\tau}),\nabla_x u(\tau,\mathbf{X}_\tau)\Big)\,\mathrm{d}\tau
\end{displaymath}
has zero expectation. Using the fact that $\mathbf{X}_t=x$, the martingale property gives
\begin{equation}\label{eq:martingal}
\mathbb{E}^{t,x}\big[\mathfrak{M}^{t,x}_{t+h}-\mathfrak{M}^{t,x}_{t}\big]
=\mathbb{E}^{t,x}\left[u\big(t+h,\mathbf{X}_{t+h}\big)-u\big(t,x\big)
 -\int_{t}^{t+h}f\Big(\tau,\mathbf{X}_{\tau},u(\tau,\mathbf{X}_{\tau}),\nabla_x u(\tau,\mathbf{X}_\tau)\Big)\,\mathrm{d}\tau\right]
=0.
\end{equation}
Then we get the martingale based representation of solution $u(t,x)$ that for any $0<h\leq T-t$,
\begin{equation}\label{eq:martingalSolv}
u(t,x)=\mathbb{E}^{t,x}\left[u\big(t+h,\mathbf{X}_{t+h}\big)
-\int_{t}^{t+h}f\Big(\tau,\mathbf{X}_{\tau},u(\tau,\mathbf{X}_{\tau}),\nabla_x u(\tau,\mathbf{X}_\tau)\Big)\,\mathrm{d}\tau\right].
\end{equation}
In the special case where $T=t+h$ and Equation \eqref{eq:PDEs} is linear, Equation \eqref{eq:martingalSolv} reduces to the classical Feynman-Kac formula.

As a consequence, for given $(t,x)$, $u(t,x)$ is the minimizer of an optimization problem
\begin{displaymath}
\min_{u} \left|u(t,x)-\mathbb{E}^{t,x}\left[u\big(t+h,\mathbf{X}_{t+h}\big)
+\int_{t}^{t+h}f\Big(\tau,\mathbf{X}_{\tau},u(\tau,\mathbf{X}_{\tau}),\nabla_x u(\tau,\mathbf{X}_\tau)\Big)\,\mathrm{d}\tau\right]\right|^2.
\end{displaymath}
Hence, the solution $u(t,x)$ of Equation \eqref{eq:PDEs} can be found by solving the following minimization problem
\begin{equation}\label{eq:LosDeepMartNet}
\begin{aligned}
\min_{u}\
& \mathbb{E}_{(t,x)\sim P}\left|u(t,x)-\mathbb{E}^{t,x}\left[u\big(t+h,\mathbf{X}_{t+h}\big)
-\int_{t}^{t+h}f\Big(\tau,\mathbf{X}_{\tau},u\big(\tau,\mathbf{X}_{\tau}\big),\nabla_x\, u(\tau,\mathbf{X}_\tau)\Big)\,\mathrm{d}\tau\right]\right|^{2}\\
&=\int_{0}^{T}\int_{\mathbb{R}^d}\left|u(t,x)-\mathbb{E}^{t,x}\left[u\big(t+h,\mathbf{X}_{t+h}\big)
-\int_{t}^{t+h}f\Big(\tau,\mathbf{X}_{\tau},u\big(\tau,\mathbf{X}_{\tau}\big),\nabla_x\, u(\tau,\mathbf{X}_\tau)\Big)\,\mathrm{d}\tau\right]\right|^{2}
P(t,x)\mathrm{d}t\,\mathrm{d}x,
\end{aligned}
\end{equation}
where $P(t,x)$ denotes the probability used for sampling $(t,x)$ and will be discussed in Sec.~\ref{sec:Sampling}. The minimization problem \eqref{eq:LosDeepMartNet} serves as the starting point of DeepMartNet and the proposed iSMART method.

It is worth noting that the problem \eqref{eq:LosDeepMartNet} exhibits a nested (or double) expectation structure. An outer expectation over the spatio-temporal sampling w.r.t. $P(t,x)$, and an inner conditional expectation $\mathbb{E}^{t,x}[\cdot]$ for a given $(t,x)$. Evaluating such double expectations is inherently computationally expensive and hence limit the algorithm's scalability in high dimensions. To address this issue, we will turn the evaluation of conditional expectation into a least square problem and raise the iSMART method.

\section{The iSMART: An Iterative Sampling-and-Regression Technique}
\label{Sec:ismart}
In this section, we present the theoretical basis of the proposed iSMART framework and outline its implementation framework.

\subsection{Theoretical Basis: An Equivalence Proposition}
The following elementary proposition establishes the equivalence between the conditional expectation of a random variable and its $L^2$-projection onto the space of functions measurable with respect to the current state, which is also well-known in probability theory \cite{ELV2019}. This result justifies replacing the conditional expectation in the martingale representation by a regression objective, and forms the theoretical basis of the iSMART procedure that we now introduce.

\begin{proposition}\label{lem:eqLem}
Let $(\Omega, \mathcal{F}, \mathbb{P})$ be a probability space, let $\mathcal{F}_t\subset\mathcal{F}$ be a sub-$\sigma$-algebra,
$\xi\in L^2(\Omega,\mathcal{F},\mathbb{P})$ be a real-valued square-integrable random variable. Then for any $\mathcal{F}_t$-measurable  and square-integrable function $\varphi$, define the following two loss functionals
\begin{displaymath}
\mathcal{J}_1(\varphi):=\mathbb{E}\Big[\bigl|\varphi-\mathbb{E}\big[\xi\mid\mathcal{F}_t\big]\bigr|^2\Big]
\quad{\rm and}\quad
\mathcal{J}_2(\varphi):=\mathbb{E}\Big[\bigl|\varphi-\xi\bigr|^2\Big].
\end{displaymath}
Then the sets of minimisers of $\mathcal{J}_1(\varphi)$ and $\mathcal{J}_2(\varphi)$ over $\varphi\in L^2(\mathcal{F}_t)$ coincide, i.e.,
\begin{equation}
\operatorname*{arg\,min}_{\varphi\in L^2(\mathcal{F}_t)}\mathcal{J}_1(\varphi)
= \operatorname*{arg\,min}_{\varphi\in L^2(\mathcal{F}_t)}\mathcal{J}_2(\varphi),
\end{equation}
with equality understood up to almost sure equivalence. Moreover, the two functionals differ only by a constant that does not depend on $\varphi$, namely, the expected conditional variance of $\xi$:
\begin{equation}
\mathcal{J}_2(\varphi)-\mathcal{J}_1(\varphi)=\mathbb{E}\Big[\operatorname{Var}\big(\xi\mid\mathcal{F}_t\big)\,\Big].
\end{equation}
\end{proposition}

\begin{proof}
We adopt the standard orthogonal projection argument. Set $m_t:=\mathbb{E}[\xi\mid\mathcal{F}_t]$. Since $\xi\in L^2(\Omega,\mathcal{F},\mathbb{P})$, its conditional expectation $m_t\in L^2(\mathcal{F}_t)$ is well-defined. For any $\varphi\in L^2(\mathcal{F}_t)$, we decompose the residual as $\varphi-\xi=(\varphi-m_t)+(m_t-\xi)$. Expanding the squared norm gives $|\varphi-\xi|^2=|\varphi-m_t|^2+|m_t-\xi|^2+2(\varphi-m_t)(m_t-\xi)$. Taking expectations and using the definition of $\mathcal{J}_1$ and $\mathcal{J}_2$, we obtain
\begin{displaymath}
\mathcal{J}_2(\varphi)=\mathcal{J}_1(\varphi)+\mathbb{E}\left[|m_t-\xi|^2\right]+2\,\mathbb{E}\left[\big(\varphi-m_t\big)\,\big(m_t-\xi\big)\right].
\end{displaymath}
The cross term vanishes. Indeed, since $\varphi-m_t$ is $\mathcal{F}_t$-measurable and belongs to $L^2(\mathcal{F}_t)$, while $m_t-\xi\in L^{2}$, the product $(\varphi-m_t)(m_t-\xi)$ is integrable, and $\mathbb{E}\left[(\varphi-m_t)(m_t-\xi)\right]=
\mathbb{E}\left[(\varphi-m_t)\mathbb{E}[m_t-\xi\mid\mathcal{F}_t]\right]=0$, due to $\mathbb{E}[m_t-\xi\mid\mathcal{F}_t]=m_t-\mathbb{E}[\xi\mid\mathcal{F}_t]=0$. Consequently,
\begin{displaymath}
\mathcal{J}_2(\varphi)=\mathcal{J}_1(\varphi)+\mathbb{E}\left[|\xi-m_t|^2\right].
\end{displaymath}
Finally, $\mathbb{E}\left[|\xi-m_t|^2\right]
=\mathbb{E}\left[\mathbb{E}[|\xi-\mathbb{E}\left[\xi\mid\mathcal{F}_t\right]|^2\mid\mathcal{F}_t]\right]=\mathbb{E}\left[\operatorname{Var}(\xi\mid\mathcal{F}_t)\right]$,
which is independent of $\varphi$. Thus, $\mathcal{J}_1$ and $\mathcal{J}_2$ differ only by an additive constant over $L^2(\mathcal{F}_t)$, so they share the same minimisers.
\end{proof}

The key algorithmic implication of Proposition~\ref{lem:eqLem} is that explicit conditional expectation $\mathbb{E}[\xi\mid\mathcal{F}_t]$ can be avoided. In high dimensions, direct computation of $\mathbb{E}[\xi\mid\mathcal{F}_t]$ at every state requires multi-path or branching simulations, which quickly becomes prohibitive. Since $\mathcal{J}_1$ and $\mathcal{J}_2$ share identical minimizers, the proposition guarantees that regression against raw single-path samples asymptotically recovers the conditional expectation. This equivalence, which replaces expensive multi-path evaluations with single-path simulations, serves as the theoretical cornerstone of iSMART.

%++++++++++++++++++++++++++++++++++++++++++++++++++++++++
\subsection{The iSMART: A Novel Iterative Sampling-and-Regression Approach Incorporating DNN}
We now present the iSMART approach, which builds upon the local martingale representation in \eqref{eq:martingalSolv}. For any $(t,x)\in[0,T)\times\mathbb{R}^d$, let $\mathbf{X}_{s}$ ($s\in[t,T]$) be the forward process generated by $\mathcal{A}$ with $\mathbf{X}_{t}=x$. Then, for any $s\in[0, T-t]$
\begin{equation}\label{eq:iSMART_mr}
u(t,x)=\mathbb{E}^{t,x}\left[u\big(t+s,\mathbf{X}_{t+s}\big)
-\int_t^{t+s}f\Big(\tau,\mathbf{X}_\tau,u(\tau,\mathbf{X}_\tau),\nabla_x u(\tau,\mathbf{X}_\tau)\Big)\,\mathrm{d}\tau\right].
\end{equation}
Equation \eqref{eq:iSMART_mr} formally characterises $u$ as a fixed point of the conditional-expectation operator on its right-hand side. This suggests a iterative scheme where, given an approximation $u_n$, we define an updated target by substituting $u_n$ into the right-hand side of \eqref{eq:iSMART_mr}. Specifically, let
\begin{equation}\label{eq:iSMART_xi}
\mathbf{\xi}^{t,x}_n:=u_n\big(t+s,\mathbf{X}_{t+s}\big)
-\int_{t}^{t+s}f\Big(\tau,\mathbf{X}_\tau,u_n(\tau,\mathbf{X}_\tau),\nabla_x u_n(\tau,\mathbf{X}_\tau)\Big)\,\mathrm{d}\tau.
\end{equation}
The corresponding deterministic target is
\begin{equation}\label{eq:iSMART-lt}
v_n(t,x):=\mathbb{E}^{t,x}\Big[\mathbf{\xi}^{t,x}_n\Big].
\end{equation}
If $u_n=u$, then \eqref{eq:iSMART_mr} gives $v_n(t,x)=u(t,x)$, confirming the exact solution as a fixed point.

The conditional expectation in \eqref{eq:iSMART-lt} is a well-defined function of $(t,x)$, but its direct evaluation is infeasible in high dimensions. The integral over the SDE transition density might be intractable, and a naive Monte Carlo approximation may suffer from nested expectation evaluation. To circumvent these issues, we avoid evaluating the conditional expectation explicitly. Instead, we invoke the $L^2$-projection property of conditional expectation from Proposition~\ref{lem:eqLem}, which states that for a fixed random variable $\xi$, the minimizer of the regression loss
$$\min_{\varphi\in L^2(\mathcal{F}_t)} \mathbb{E}\left[|\varphi-\xi|^2\right]$$
is exactly $\mathbb{E}[\xi\mid\mathcal{F}_t]$. Applying this with $\varphi=v(t,x)$ and $\xi=\xi_n^{t,x}$ yields
\begin{equation}
\operatorname*{arg\,min}_{v(t,x)}\mathbb{E}^{t,x}\left[\left|v(t,x)-\xi_n^{t,x}\right|^2\right]
=\mathbb{E}^{t,x}\left[\xi_n^{t,x}\right]
=v_n(t,x).
\end{equation}
Thus, although $\mathbb{E}^{t,x}[\xi_n^{t,x}]$ cannot be evaluated directly, it can be recovered as the solution of a regression problem involving only sample trajectories. This equivalence forms the theoretical core of iSMART, i.e., the local target is obtained not by computing an expectation, but by solving a least-squares problem whose minimiser is precisely that expectation.

%This observation motivates the iSMART update. Let $P(\mathrm{d}t,\mathrm{d}x)$ be a sampling distribution over the admissible spatio-temporal domain $\mathcal{D}_s:=\left\{(t,x):0\leq t\leq T-s,\ x\in\mathbb{R}^d\right\}$, and let $\mathcal{N}_{\Theta}:=\{u_{\theta}:\theta\in\Theta\}$ be a parametric family of neural networks. The next iterate is obtained by projecting the local target $v_n$ onto $\mathcal{N}_{\Theta}$ in the weighted $L^2(P)$ sense
%\begin{equation}\label{eq:iSMART_pro}
%\theta_{n+1}\in\operatorname*{argmin}_{\theta\in\Theta}
%\int_{\mathcal{D}_s}\big|u_{\theta}(t,x)-v_n(t,x)\big|^2P(\mathrm{d}t,\mathrm{d}x),\qquad
%u_{n+1}:=u_{\theta_{n+1}}.
%\end{equation}
%Since $v_n(t,x)=\mathbb{E}^{t,x}[\xi_n^{t,x}]$ is not explicitly computable, Proposition~\ref{lem:eqLem} further implies that \eqref{eq:iSMART_pro} is equivalent, up to a constant independent of $\theta$, to the sample based regression objective
%\begin{equation}\label{eq:iSMART_sl}
%\theta_{n+1}
%\in\operatorname*{argmin}_{\theta\in\Theta}\mathbb{E}_{(t,x)\sim P}\mathbb{E}^{t,x}
%\Bigl[\bigl|u_{\theta}(t,x)-\xi_n^{t,x}\bigr|^2\Bigr].
%\end{equation}
%Consequently, the neural network can be trained directly using the simulated pathwise targets $\xi_n^{t,x}$ without explicitly evaluating the conditional expectation $v_n$.

This observation motivates the iSMART update.  The complete iSMART procedure is summarised as follows.

\begin{description}[labelindent=2.5em]
  \item[Step 1: Iterative scheme via SDE sampling.]
      For any $(t,x)\in [0,T)\times\mathbb{R}^d$ and a fixed $0<s< T-t$, define a iteration sequence by
      \begin{equation}\label{eq:iSMARTLo}
       \left\{\begin{aligned}
        &u_{n+1}(t,x)=\mathbb{E}^{t,x}\left[u_n\left(t+s,\mathbf{X}_{t+s}\right)-\int_{t}^{t+s}f\Big(\tau,\mathbf{X}_{\tau},u_n(\tau,\mathbf{X}_{\tau}),\nabla_x u_n(\tau,\mathbf{X}_\tau)\Big)\mathrm{d}\tau\right], \\
        & s.t.\quad u_{n+1}(T,x) = g(x),
       \end{aligned}\right.
      \end{equation}
      where $u_0(t,x)$ is a properly chosen initial state, and the stochastic trajectory $\mathbf{X}_\tau$ ($\tau\geq t$) is generated via the SDE in Eq.~\eqref{eq:Xn}, with the sampling procedure described in Sec.~\ref{sec:Sampling}:
      \begin{equation}\label{eq:Xn}
        \mathrm{d}\mathbf{X}_\tau
          =\mu\Big(\tau,\mathbf{X}_\tau,u_n(\tau,\mathbf{X}_\tau),\nabla_xu_n(\tau,\mathbf{X}_\tau)\Big)\,\mathrm{d}\tau
          +\sigma\Big(\tau,\mathbf{X}_\tau,u_n(\tau,\mathbf{X}_\tau)\Big)\,\mathrm{d}\mathbf{B}_\tau,\qquad \mathbf{X}_t=x.
      \end{equation}
      For $s=T-t$, applying the terminal condition simplifies iteration \eqref{eq:iSMARTLo} to
      \begin{equation}\label{eq:picard1}
        \left\{\begin{aligned}
         &u_{n+1}(t,x)=\mathbb{E}^{t,x}\left[g(\mathbf{X}_T)-\int_{t}^{T}f\Big(\tau,\mathbf{X}_{\tau},u_n(\tau,\mathbf{X}_{\tau}),\nabla_x u_n(\tau,\mathbf{X}_\tau)\Big)\mathrm{d}\tau\right], \\
         & s.t.\quad u_{n+1}(T,x) = g(x),
        \end{aligned}\right.
      \end{equation}
     %The solution to \eqref{eq:PDEs}, characterized by \eqref{eq:martingalSolv}, can be viewed as the fixed point of iterations \eqref{eq:iSMARTLo} and \eqref{eq:picard1}.
     The convergence of this iterative scheme is discussed in Section~\ref{Sec:convAnal}; hereafter, convergence is assumed.
  \item[Step 2: Global Regression Update.]
      For given $u_n(t,x)$, we update the global approximation $u_{n+1}$ by solving the weighted $L^2$ regression problem
      \begin{equation}\label{eq:iSMART_opt}
        u_{n+1}(t,x)=\operatorname*{arg\,min}_{v(t,x)}\mathbb{E}_{(t,x)\sim P}\mathbb{E}^{t,x}\left[\Big|v(t,x)-u_n(t+s,\mathbf{X}_{t+s})+\int_{t}^{t+s}f(\tau,\mathbf{X}_\tau,u_n(\tau,\mathbf{X}_\tau),\nabla_x u_n(\tau,\mathbf{X}_\tau))\mathrm{d}\tau\Big|^2\right].
      \end{equation}
      Since \eqref{eq:iSMART_opt} holds for any $0<s\leq T-t$, we can replace the first expectation in \eqref{eq:iSMART_opt} by $\mathbb{E}_{(t,x,s)\sim \tilde{P}}$ and update $u_{n+1}$ by
      \begin{equation}\label{eq:updateun1}
       u_{n+1}(t,x)=\operatorname*{arg\,min}_{v(t,x)} \mathbb{E}_{(t,x,s)\sim \tilde{P}}\mathbb{E}^{t,x}\left[\Big|v(t,x)-u_n(t+s,\mathbf{X}_{t+s})+\int_{t}^{t+s}f(\tau,\mathbf{X}_\tau, u_n(\tau,\mathbf{X}_\tau),\nabla_x u_n(\tau,\mathbf{X}_\tau))\mathrm{d}\tau\Big|^2\right],
      \end{equation}
      where $\tilde{P}(t,x,s)$ is the probability to sample $(t,x,s)$. Typically, we can set $\tilde{P}(t,x,s) = P_t(t)P_x(x)P_s(s|t)$. The choice of each probability will be specified in the following subsection.
  \item[Step 3: Iterative Termination.]
      Repeat Steps 1 and 2 sequentially until the convergence of the sequence $\{u_n\}_{n\in\mathbb{N}}$ is achieved.
\end{description}

The above procedure can be implemented efficiently using a DNN with the stop-gradient technique. Let $u_\theta(t,x)$ be a DNN parameterized by $\theta$. Given parameters $\theta_n$ and fixed $(t,x,s)$, we update the parameters by solving the mean least-squares problem:
\begin{equation}\label{eq:semiloss}
  \begin{aligned}
    \theta_{n+1} &= \operatorname*{arg\,min}_{\theta}\mathbb{E}_{(t,x,s)\sim \widetilde P}\mathbb{E}^{t,x}\left[\Big|u_{\theta}(t,x) - u_{\theta_n}\big(t+s,\mathbf{X}_{t+s}\big)+\int_{t}^{t+s}f\Big(\tau,\mathbf{X}_\tau,u_{\theta_n}\big(\tau,\mathbf{X}_\tau\big),\nabla_x u_{\theta_n}\big(\tau,\mathbf{X}_\tau\big)\Big)\,\mathrm{d}\tau\Big|^2\right] \\
    &\approx \operatorname*{arg\,min}_{\theta}\frac{1}{N M}\sum_{i=1}^{N}\sum_{m=1}^{M}\left|u_{\theta}(t_i,x_i)-u_{\theta_n}\big(t_i+s_i,\mathbf{X}^{(m,i)}_{t_i+s_i}\big)+\int_{t_i}^{t_i+s_i}f\Big(\tau,\mathbf{X}^{(m,i)}_\tau,u_{\theta_n}\big(\tau,\mathbf{X}^{(m,i)}_\tau\big),\nabla_x u_{\theta_n}\big(\tau,\mathbf{X}^{(m,i)}_\tau\big)\Big)\,\mathrm{d}\tau\right|^2,
  \end{aligned}
\end{equation}
where $\{(t_i,x_i,s_i)\}_{i=1}^N$ are $N$ i.i.d. samples drawn from $\widetilde{P}$. For each $(t_i,x_i)$, $\mathbf{X}^{m,i}$ ($m=1,\dots,M$) are $M$ i.i.d. trajectories generated by
\begin{equation}\label{eq:Xmn}
\mathrm{d}\mathbf{X}^{(m,i)}_\tau = \mu\Big(\tau,\mathbf{X}^{(m,i)}_\tau, u_{\theta_n}\big(\tau,\mathbf{X}^{(m,i)}_\tau\big), \nabla_x u_{\theta_n}\big(\tau,\mathbf{X}^{(m,i)}_\tau\big)\Big)\,\mathrm{d}\tau+\sigma\Big(\tau,\mathbf{X}^{(m,i)}_\tau, u_{\theta_n}\big(\tau,\mathbf{X}^{(m,i)}_\tau\big)\Big)\,\mathrm{d}\mathbf{B}^{(m)}_\tau, \qquad \mathbf{X}^{(m,i)}_{t_i}=x_i,
\end{equation}
with $\mathbf{B}_\tau^{(m)}$ being $M$ independent standard Brownian motions. The stop-gradient technique allows efficient, unrolled implementation of this procedure. Using the operator $\mathfrak{Sg}[\cdot]$ to denote the stop-gradient operation, the update \eqref{eq:semiloss} can be written compactly expressed as
\begin{equation}\label{eq:semiloss1}
    \theta\gets \operatorname*{arg\,min}_{\theta}\frac{1}{NM}\sum_{i=1}^N\sum_{m=1}^{M}\left|u_{\theta}(t_i,x_i)-\mathfrak{Sg}\left[u_{\theta}(t_i+s_i,\mathbf{X}^{(m,i)}_{t_i+s_i})-\int_{t_i}^{t_i+s_i}f(\tau,\mathbf{X}^{(m,i)}_\tau, u_{\theta}(\tau,\mathbf{X}^{(m,i)}_\tau),\nabla_x u_{\theta}(\tau,\mathbf{X}^{(m,i)}_\tau))\,\mathrm{d}\tau\right]\right|^2.
\end{equation}

The overall procedure consists of two decoupled stages per iteration. First, generate stochastic trajectories according to \eqref{eq:Xmn} and compute the pathwise targets, namely
\begin{equation}\label{eq:iSMART_xi}
\xi_{m,i} := u_{\theta}\Big(t_i+s_i,\mathbf{X}^{(m,i)}_{t_i+s_i}\Big)-\int_{t_i}^{t_i+s_i}f\Big(\tau,\mathbf{X}^{(m,i)}_\tau,u_{\theta}(\tau,\mathbf{X}^{(m,i)}_\tau),\nabla_x u_{\theta}(\tau,\mathbf{X}^{(m,i)}_\tau)\Big)\,\mathrm{d}\tau.
\end{equation}
Second, fit the generated targets by minimizing the following universal loss function
\begin{equation}\label{eq:loss}
\mathcal{L}_n(\theta):=\frac{1}{NM}\sum_{i=1}^N\sum_{m=1}^{M} \left|u_{\theta}(t_i,x_i)-\mathfrak{Sg}[\xi_{m,i}]\right|^2,
\end{equation}
so that the parameter update is given by
\begin{equation}\label{eq:update}
\theta_{n+1} \gets \operatorname*{arg\,min}_{\theta}\, \mathcal{L}_n(\theta).
\end{equation}
These two stages are executed iteratively until a specified stopping criterion is met.

From a computational perspective, the first stage involves only forward evaluations of the SDE and the current DNN, requiring no computational graph construction for automatic differentiation, which renders its computational overhead negligible compared to network training. In the second stage, the stop-gradient operation cuts off backward propagation through the target $\xi_{m,i}$, reducing the sub-problem \eqref{eq:loss} to a standard data-fitting task compatible with modern deep learning optimizers (e.g., Adam). In practice, we do not solve \eqref{eq:update} to exact convergence; instead, we optimize $\mathcal{L}_n(\theta)$ for a fixed number of gradient descent steps before regenerating new sampling trajectories via \eqref{eq:Xmn} and \eqref{eq:iSMART_xi}.

The stochastic trajectories \eqref{eq:Xmn} are discretized using the Euler-Maruyama scheme, and the time integral in \eqref{eq:iSMART_xi} is approximated via a quadrature rule (e.g., the trapezoidal rule). Additional sampling strategies and the selection of the proposal distribution $\widetilde P(t,x,s)$ are detailed in the following subsection.

The complete iSMART iteration is formally summarised in Algorithm~\ref{alg:ismart}.

\begin{algorithm}[htbp]
\caption{iSMART: A Martingale-Based PDE Solver}
\label{alg:ismart}
\begin{algorithmic}[1]
\Require Proposal distribution $\widetilde{P}(t,x,s)$ over domain $\mathcal{D}_{\widetilde{P}} = \{(t,x,s) \mid t \in [0,T), x \in \Omega, s \in (0, T-t]\}$
\Require Parameterized neural network $u_\theta(t,x)$ with initial parameters $\theta_0$
\Require Mini-batch size $N$, number of Monte Carlo paths $M$, max outer iterations $\mathrm{ITER_{max}}$, inner steps $K$
\For{$n = 0$ \textbf{to} $\mathrm{ITER_{max}}-1$}
    \State Sample a mini-batch of spatiotemporal points $\{(t_i, x_i, s_i)\}_{i=1}^N \sim \widetilde{P}$
    \For{$m = 1$ \textbf{to} $M$}
        \State Simulate SDE trajectory $\mathbf{X}^{(m,i)}_\tau$ over $\tau \in [t_i, t_i+s_i]$ initialized at $\mathbf{X}^{(m,i)}_{t_i} = x_i$ via \eqref{eq:Xmn}
        \State Compute pathwise target $\xi_{m,i}$ via \eqref{eq:iSMART_xi} (substituting $u(T,\mathbf{X}_T)=g(\mathbf{X}_T)$ if $t_i+s_i=T$)
        \State Detach target from computational graph: $\xi_{m,i} \gets \mathfrak{Sg}[\xi_{m,i}]$
    \EndFor
    \For{$k = 1$ \textbf{to} $K$}
        \State Compute mean squared loss: $\mathcal{L}(\theta) = \frac{1}{NM}\sum_{i=1}^{N}\sum_{m=1}^{M}\left|u_{\theta}(t_i,x_i) - \xi_{m,i}\right|^2$
        \State Update parameters $\theta$ by optimizing $\mathcal{L}(\theta)$ for one gradient step (e.g., via Adam)
    \EndFor
\EndFor
\Ensure Converged parameters $\theta^*$ and corresponding approximate solution $u_{\theta^*}(t,x)$
\end{algorithmic}
\end{algorithm}

\subsection{Sampling and Path Generation Strategies}
\label{sec:Sampling}

We now discuss the selection of the proposal sampling distribution $\widetilde{P}(t,x,s)$ and the trajectory generation strategies for various PDE settings.

Typically, it is computationally advantageous to sample the temporal variable $t$ and spatial variable $x$ independently, while the lookahead time interval $s$ is sampled conditional on $t$. For the time variable $t \in [0,T]$, although uniform sampling over $[0,T]$ is a standard choice, we design a dynamic Beta distribution sampling strategy. During the training process, $t$ is drawn according to
\begin{equation}\label{eq:temporalsample}
t\sim T\cdot\mathrm{Beta}(\alpha,1),\qquad\alpha=\max\big\{1,\,10\,(1-\mathrm{epoch}/N)\big\},
\end{equation}
where $N$ denotes the total number of training iterations. This strategy follows the way how the information propagating. Initially $\alpha=10$, which strongly biases samples toward the terminal time and reinforces the numerical solution $u_\theta$ to be accurate near the terminal time. As training proceeds, $\alpha\to1$ and the distribution becomes uniform, allowing the learned part of solution propagates gradually to the entire time interval.

For spatial sampling, uniform sampling in high dimensions inherently concentrates samples near the boundary of the domain due to the concentration of measure. To ensure effective spatial coverage, we introduce a multi-layer Gaussian sampling strategy supplemented by uniform boundary regularization. The spatial coordinate $x$ is sampled from the mixture density
\begin{equation}\label{eq:spatialsample}
p_x(x) = \sum_{k=1}^K w_k\,\mathcal{N}\big(x; 0, \sigma_k^2 I_d\big) + w_0\,\mathcal{U}\big(x; [-R, R]^d\big),
\end{equation}
where $w_k\ge0$ be given. In high spatial dimensions ($d \gg 1$), a Gaussian vector $x \sim \mathcal{N}(0, \sigma_k^2 I_d)$ concentrates heavily near a hypersphere of radius approximately $\sqrt{d}\,\sigma_k$. By appropriately selecting the bandwidth parameters $\{\sigma_k\}_{k=1}^K$, the generated samples cover nested spherical shells across the domain of interest. The uniform component $\mathcal{U}([-R,R]^d)$ serves as a global spatial regularizer. The time discretization is implemented using an accurate numerical scheme, while the spatial sampling is illustrated schematically in Figure~\ref{fig:sampling}.

\begin{figure}[!htbp]
 \centering
 \includegraphics[width=0.86\textwidth]{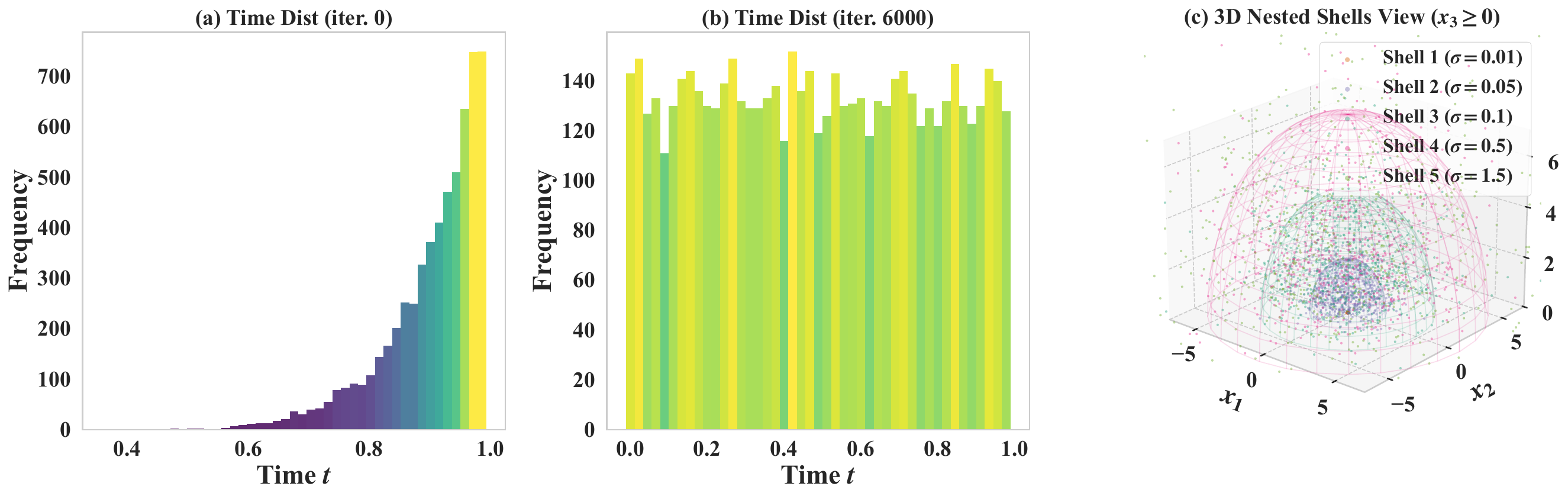}
\caption{Visualization of the spatiotemporal sampling scheme. Subfigures (a) and (b) illustrate the temporal sampling profile at epoch~$0$ and the final epoch, respectively. Subfigure (c) presents a 3D schematic illustration of spatial samples generated by the multi-scale Gaussian strategy \eqref{eq:spatialsample}.}
\label{fig:sampling}
\end{figure}

We next elaborate on path generation strategies and choices for the conditional lookahead distribution $P_s(s \mid t)$. To adapt to different structural characteristics of equation \eqref{eq:PDEs}, we categorize three tailored strategies.

\smallskip\vspace{0.15cm}
\noindent\textbf{Method ($I$): Brownian trajectory sampling for constant diffusion.}
When the diffusion coefficient $\sigma(t,x,u) \equiv \sigma$ is constant, Equation \eqref{eq:PDEs} can be rewritten as
\begin{equation}
\partial_t u + \frac{\sigma^2}{2}\Delta_x u = \tilde{f}(t,x,u,\nabla_x u) := f(t,x,u,\nabla_x u) - \mu(t,x,u,\nabla_x u)^\top \nabla_x u,
\end{equation}
where $\mathcal{A} = \frac{\sigma^2}{2}\Delta$ acts as the infinitesimal generator, and $\tilde{f}$ is treated as a pseudo-source term. The corresponding stochastic process satisfies $\mathrm{d}\mathbf{X}_\tau = \sigma \mathrm{d}\mathbf{B}_\tau$, implying that $\mathbf{X}_s \sim \mathcal{N}(x, \sigma^2 (s-t) I_d)$ for any $s > t$. Since exact transition densities are available, we choose a fixed lookahead time step $P_s(s \mid t) = \delta(s - s^*)$ with $s^* = \min\{s_0, T-t\}$ for a constant $s_0 > 0$. To achieve second-order temporal accuracy $\mathcal{O}(\Delta\tau^2)$ for the pathwise integral in \eqref{eq:iSMART_xi}, we evaluate the integral via the composite trapezoidal quadrature rule over $N_\mathrm{Int}$ sub-intervals with step size $\Delta\tau = s/N_\mathrm{Int}$:
\begin{equation}\label{eq:integral_trapezoidal}
\int_{t}^{t+s} \tilde{f}\Big(\tau, \mathbf{X}_\tau, u(\tau, \mathbf{X}_\tau), \nabla_x u(\tau, \mathbf{X}_\tau)\Big) \,\mathrm{d}\tau
\approx \Delta\tau \left[ \frac{1}{2}\tilde{f}_0 + \sum_{j=1}^{N_\mathrm{Int}-1} \tilde{f}_j + \frac{1}{2}\tilde{f}_{N_\mathrm{Int}} \right],
\end{equation}
where the evaluation node $\tilde{f}_j$ ($j = 0, 1, \dots, N_\mathrm{Int}$) is defined as $\tilde{f}_j := \tilde{f}(t + j\Delta\tau, \, \mathbf{X}_{t + j\Delta\tau}, \, u(t + j\Delta\tau, \mathbf{X}_{t + j\Delta\tau}), \, \nabla_x u(t + j\Delta\tau, \mathbf{X}_{t + j\Delta\tau}))$. Because $\mathbf{X}_s$ has an explicit sampling way, any quadrature rule can be applied if needed. This method is the simplest way to generate trajectories, and also be applied in \cite{Cai2026a} and \cite{Cai26b}. Note that this method is independent of the solution of PDE, the generated trajectories need not to be updated every epoch. In practise, we can re-generate trajectories every 10-20 iterations. However, it is also worth noting that this method has to compute the gradient of $u$ providing $\mu\neq 0$. Although it can be achieved by backward propagation efficiently, the cost can be expensive for extremely high dimensional problem. The DRDM in \cite{Cai26b} can be used to reduce the auto differentiation cost.

\smallskip\vspace{0.15cm}
\noindent\textbf{Method ($II$): Full SDE discretization via Euler-Maruyama.} Given a set of parameters $\theta$, the straight forward way of generating paths is to solve \eqref{eq:Xmn} by Euler-Maruyama Scheme with a fixed step size $\Delta \tau>0$. For a given $(t_i,x_i)$, with a little abuse of notation,
\begin{equation}\label{eq:trajectory}
\mathbf{X}_{j+1}^{(m,i)} = \mathbf{X}_{j}^{(m,i)} + \mu\Big(\tau_{i,j}, \mathbf{X}_{j}^{(m,i)}, u_\theta(\tau_{i,j}, \mathbf{X}_{j}^{(m,i)}), \nabla_x u_\theta(\tau_{i,j}, \mathbf{X}_{j}^{(m,i)})\Big)\Delta\tau
+ \sigma\Big(\tau_{i,j}, \mathbf{X}_{j}^{(m,i)}, u_\theta(\tau_{i,j}, \mathbf{X}_{j}^{(m,i)})\Big)\sqrt{\Delta\tau}\, Z_{j+1}^{(m)},
\end{equation}
where $\mathbf{X}_0^{(m,i)} = x_i$, $\tau_{i,j} = t_i + j\Delta\tau$ ($j = 0, \dots, \lfloor(T-t_i)/\Delta\tau\rfloor$), and $\{Z_{j+1}^{(m)}\}_{j,m}$ are i.i.d. standard Gaussian vectors. Collecting intermediate discrete nodes under $P_s(s) = \sum_j \delta(s-j\Delta\tau)$, the empirical loss \eqref{eq:loss} reduces to
\begin{equation}\label{eq:loss_method2}
\mathcal{L}(\theta)=\frac{1}{NM}\sum_{(t_i,x_i)}\frac{1}{N_{\mathrm{Int}}^{(i)}}\sum_{j=1}^{N_{\mathrm{Int}}^{(i)}} \sum_{m=1}^{M} \left| u_\theta(t_i, x_i)-\mathfrak{Sg}\Bigg[ u_\theta(\tau_{i,j}, \mathbf{X}_j^{(m,i)})-\Delta\tau \sum_{l=1}^{j} f\Big(\tau_{i,l}, \mathbf{X}_l^{(m,i)}, u_\theta(\tau_{i,l}, \mathbf{X}_l^{(m,i)}), \nabla_x u_\theta(\tau_{i,l},\mathbf{X}_l^{(m,i)})\Big) \Bigg] \right|^2,
\end{equation}
with $N_{\mathrm{Int}}^{(i)}:=\lfloor(T-t_i)/\Delta\tau\rfloor$. This method makes full use of all the points needed for generating every single path. Since \eqref{eq:trajectory} relies on the current $u_{\theta}$, once the parameter $\theta$ is updated, $u_{\theta}$ is changed and one has to re-compute \eqref{eq:trajectory} and \eqref{eq:loss_method2}. The path generation has to be carried out in every epoch. So method ($II$) is more expansive than Method ($I$) when generating trajectories. However, the key advantage of this method is that when $\mu$ and $f$ are independent of $\nabla_x u$, this method is derivative free. The data generation stage only need the evaluation of DNN, which is efficient from programming aspect. Thus, Method ($II$) can be more efficient than Method ($I$) for this case. Numerical experiments verify it in Sec.~\ref{Sec:NR}.

\smallskip\vspace{0.15cm}
\noindent\textbf{Method ($III$): Freezing-and-Compensating strategy.} When the drift $\mu$ and the source $f$ depend on $\nabla_x u$, while the diffusion $\sigma$ is state-independent, we introduce a freezing-and-compensating technique to stabilize the iteration and avoid frequent trajectory re-simulations. Given the solution estimate $u_n$ from the $n$-th iteration, we freeze the drift term $\mu(t,x,u_n,\nabla_x u_n)$ and write
\[\mu(t,x,u,\nabla_x u) = \mu(t,x,u_n,\nabla_x u_n) + \left(\mu(t,x,u,\nabla_x u) - \mu(t,x,u_n,\nabla_x u_n)\right).\]
Moreover, to control the non-linearity introduced by the gradient dependence in $f(t,x,u,\nabla_x u)$,  we can decompose it as:
\begin{equation}
f(t,x,u,y)=f(t,x,u,0)+\varphi(t,x,u,\nabla_x u)^\top \nabla_x u,
\end{equation}
where $\varphi(t,x,u,\nabla_x u):=\int_{0}^{1}\nabla_{\!y}f(t,x,u,s\,\nabla_x u)\,\mathrm{d}s\in\mathbb{R}^d$. By freezing $\alpha_n \varphi$, where $\alpha_n\in (0,1)$ is constant, we can reformulate \eqref{eq:PDEs} as follow:
\begin{equation}
\begin{aligned}
\partial_t u &
+\Big(\mu(t,x,u_n,\nabla_x u_n)-\alpha_n \varphi(t,x,u_n,\nabla_x u_n)\Big)^\top\nabla_x u
+\frac{1}{2}\mathrm{Tr}\left\{\sigma\sigma^\top(t,x)\nabla_x^2 u\right\} =\tilde{f}(t,x,u,\nabla_x u)\\
& = f(t,x,u,\nabla_x u)-\alpha_n \varphi(t,x,u_n,\nabla_x u_n)^\top\nabla_x u -  \left(\mu(t,x,u,\nabla_x u) - \mu(t,x,u_n,\nabla_x u_n)\right)^\top\nabla_x u.
\end{aligned}
\end{equation}
For the $(n+1)$-th iteration, trajectories are generated via the decoupled SDE:
\begin{equation}
\mathrm{d}\mathbf{X}_\tau
=\Big(\mu\big(\tau,\mathbf{X}_\tau,u_n(\tau,\mathbf{X}_\tau),\nabla_x u_n(\tau,\mathbf{X}_\tau)\big)
-\alpha_n \varphi\big(\tau,\mathbf{X}_\tau,u_n(\tau,\mathbf{X}_\tau),\nabla_x u_n(\tau,\mathbf{X}_\tau)\big)\Big) \,\mathrm{d}\tau+\sigma(\tau,\mathbf{X}_\tau)\,\mathrm{d}\mathbf{B}_\tau.
\end{equation}
Because $u_n$ and $\nabla_x u_n$ are known from the previous iteration, all coefficients in the drift of $\mathbf{X}_\tau$ remain fixed during trajectory sampling. Comparing with method ($II$), the re-sampling procedure does not need to be applied every epoch. We can re-generate paths after 10-20 epoch of training. Concurrently, the compensation term together with original $f$ on the right-hand side serves as a pseudo-source, forming a consistent regression target in the loss function and ensuring strict mathematical equivalence to the original PDE. In particular, for the HJB consider in this work, source term $f(t,x,u,\nabla_x u)=\delta^2|\nabla_x u|^2$. Evaluating the integral identity yields the explicit vector-valued mapping $\varphi(t,x,u,\nabla_x u)=\int_{0}^{1}(s\nabla_x u)\mathrm{d}s=\delta^2\nabla_x u$. Method ($III$) freezes part of the gradient-dependent nonlinearity in the drift and compensates for it through the source term, thereby decoupling trajectory generation from the current solution update and avoiding the high computational overhead of frequent path re-simulations. The relaxation parameter $\alpha_n$ provides crucial flexibility in balancing the absorption of nonlinearities against iteration stability. While this strategy is particularly effective for managing quadratic gradient nonlinearities of the form $|\nabla_x u|^2$ frequently encountered in HJB equations, its effectiveness relies on the condition that the previous iterate $u_n$ is sufficiently accurate to render the compensation reliable.

%+++++++++++++++++++++++++++++++++++++++++++++++++
\section{Convergence analysis of the iSMART}
\label{Sec:convAnal}
In this section, we provide a rigorous convergence analysis for the iteration \eqref{eq:iSMARTLo} and \eqref{eq:picard1}.  We remark that this convergence analysis does not take into account the minibatch optimization step for \eqref{eq:updateun1} utilized in practical computations.

Denoting $\mathcal{D}=[0,T]\times\mathbb{R}^d$,  we work in the weighted Sobolev space
\begin{equation}
\mathbb{X}_{\beta,m}:=W^{1,\infty}_{\beta,m}(\mathcal{D})
= \Bigl\{ u\in L^{1}_{\mathrm{loc}}(\mathcal{D}):
\nabla_x u \text{ exists weakly and } \|u\|_{\beta,m}<\infty, m\in\mathbb{N}_{+}\Bigr\},
\end{equation}
endowed with the norm
\begin{equation}\label{eq:weightednorm}
\|\,u\,\|_{\beta,m}:= \sup_{(t,x)\in\mathcal{D}} e^{-\beta(T-t)}\left(\frac{|u(t,x)|}{1+|x|^m}\right)
 + \sup_{(t,x)\in\mathcal{D}} e^{-\beta(T-t)}\left(\frac{|\nabla_{x}u(t,x)|}{1+|x|^m}\right).
\end{equation}
The pair $(\mathbb{X}_{\beta,m},\|\cdot\|_{\beta,m})$ forms a Banach space. Because the norm controls the essential supremum of the weighted gradient, every $u\in\mathbb{X}_{\beta,m}$ is locally Lipschitz continuous and therefore differentiable almost everywhere.

For convenience, we also introduce the semi-norms
\begin{align*}
\big|\,u\,\big|_{\beta,m}:&=\sup_{(t,x)\in\mathcal{D}} e^{-\beta(T-t)}\left(\frac{|u(t,x)|}{1+|x|^m}\right),\\
\big|\,\nabla_{x}u\,\big|_{\beta,m}:&=\sup_{(t,x)\in\mathcal{D}} e^{-\beta(T-t)}\left(\frac{|\nabla_{x} u(t,x)|}{1+|x|^m}\right).
\end{align*}
These semi-norms will be used to obtain succinct estimates in the subsequent analysis.

Building upon the above functional framework, we now state the assumptions required for the convergence analysis.
\begin{itemize}
  \item[\textbf{A1.}] Drift coefficient.
      $\mu(t,x,y,z): [0,T]\times\mathbb{R}^d\times\mathbb{R}\times\mathbb{R}^d\to\mathbb{R}^d$
      is uniformly bounded and Lipschitz continuous:
      \begin{displaymath}
      \left\{\begin{aligned}
      &|\mu(t,x,y,z)|\le M, && M>0,\ \forall\,(t,x,y,z)\in[0,T]\times\mathbb{R}^{d}\times\mathbb{R}\times\mathbb{R}^{d},\\[2pt]
      &|\mu(t,x,y,z)-\mu(t,x',y',z')|
      \le L\left(|x-x'|+|y-y'|+|z-z'|\right), &&
          \forall\,t\in[0,T],\ x, x'\in\mathbb{R}^d,\ y, y'\in\mathbb{R},\ z, z'\in\mathbb{R}^d.
      \end{aligned}\right.
      \end{displaymath}
  \item[\textbf{A2.}] Diffusion coefficient.
      $\sigma(t,x,y): [0,T]\times\mathbb{R}^d\times\mathbb{R}\to\mathbb{R}^{d\times q}$
      is Lipschitz continuous, and uniformly elliptic:
      \begin{displaymath}
      \left\{\begin{aligned}
      &|\sigma(t,x,y)-\sigma(t,x',y')|
          \le L\left(|x-x'|+|y-y'|\right),
          &&\forall\,t\in[0,T],\ x, x'\in\mathbb{R}^d,\ y, y'\in\mathbb{R},\\[2pt]
      &\lambda^{-1}|\xi|^2 \le \xi^{\top}\sigma(t,x,y)\sigma^{\top}(t,x,y)\,\xi\le \lambda\,|\xi|^2, \qquad
          &&\forall\,t\in[0,T],\ x\in\mathbb{R}^d,\ y\in\mathbb{R},\ \mathbf{\xi}\neq0,\ \xi\in\mathbb{R}^d.
      \end{aligned}\right.
      \end{displaymath}
  \item[\textbf{A3.}] Source term.
      $f(t,x,y,z): [0,T]\times\mathbb{R}^d\times\mathbb{R}\times\mathbb{R}^d\to\mathbb{R}$
      satisfies the Lipschitz condition
      \begin{displaymath}
      |f(t,x,y,z)-f(t,x',y',z')|
      \le L\left|x-x'|+|y-y'||z-z'|\right), \quad
      \forall\,t\in[0,T],\ x, x'\in\mathbb{R}^d,\ y,y'\in\mathbb{R},\ z, z'\in\mathbb{R}^d.
      \end{displaymath}

  \item[\textbf{A4.}] Terminal condition.
      $g(x)\in\mathbb{X}_{\beta,1}$ and is Lipschitz continuous:
      \begin{displaymath}
      |g(x)-g(x')| \le L|x-x'|, \qquad \forall\, x, x'\in\mathbb{R}^d.
      \end{displaymath}
\end{itemize}
In the above assumptions, $|\cdot|$ denotes the Euclidean norm for vectors and the Frobenius norm for matrices. $L>0$ and $\lambda>0$ are some constants. Since Lipschitz continuity implies linear growth, we work mainly with $m=1$ and write $\mathbb{X}_\beta:=\mathbb{X}_{\beta,1}$ for simplicity. Throughout the paper, $C$ stands for a generic constant depending only on $L$, $M$, $T$, and $\lambda$; its value may change from line to line.

Here we focus on presenting the main theorem, i.e. the convergence result; the corresponding auxiliary lemmas and their interconnections are referred to \ref{Lems:ALTF}.

%\begin{theorem}\label{thm:converge}
%Assume that $g$ is Lipschitz continuous and that for every $n\ge 1$ the terminal condition $u_n(T,x)=g(x)$ holds. Then for any $h>0$ the sequence $\{u_n\}_{n=1}^\infty$ obtained from the iteration \eqref{eq:iterat} converges in $\mathbb{X}_\beta$ to the unique solution $u^*$ of Problem \eqref{eq:PDEs}.
%\end{theorem}

\begin{theorem}\label{thm:converge}
Assume that the drift $\mu$ satisfies the boundedness and Lipschitz conditions in Assumption~\textbf{A1}, the diffusion $\sigma$ satisfies the Lipschitz and uniform ellipticity conditions in Assumption~\textbf{A2}, and the source term $f$ satisfies the Lipschitz condition in Assumption~\textbf{A3}. Let the terminal data $g\in\mathbb{X}_{\beta,1}$ be Lipschitz continuous as in Assumption~\textbf{A4}. Suppose further that the terminal condition $u_n(T,x)=g(x)$ holds for every $n\ge 1$. Then, for any $h>0$, the sequence $\{u_n\}_{n=1}^\infty$ generated by the iteration \eqref{eq:iSMARTLo} converges in $\mathbb{X}_\beta$ to the unique solution $u^*$ of Problem \eqref{eq:PDEs}.
\end{theorem}

\begin{proof}
We divide the proof into two stages: first we establish that the sequence $\{u_n\}$ is Cauchy in $\mathbb{X}_\beta$, and then we identify its limit as the unique solution of \eqref{eq:PDEs}.

\medskip\noindent\textbf{Stage 1: Cauchy property.}
Fix $h>0$ and $(t,\mathbf{x})\in[0,T)\times\mathbb{R}^d$. We distinguish two cases.

\medskip\noindent\textbf{Case 1: $t+h\ge T$.}
By the probabilistic representation of $u_{n+1}$ and $u_n$ combined with the Lipschitz continuity of $g$ and $f$, we obtain
\begin{equation}\label{eq:diff_T}
\begin{aligned}
|u_{n+1}(t,x)-u_n(t,x)|
&\le\left|\mathbb{E}^{t,x}\left[g(\mathbf{X}_{T}^n)-g(\mathbf{X}_{T}^{n-1})\right]\right| \\
&\quad + \left|\mathbb{E}^{t,x}\left[\int_t^T
        \left(f\Big(\tau,\mathbf{X}_\tau^n,u_n(\tau,\mathbf{X}_\tau^n),\nabla_x u_n(\tau,\mathbf{X}_\tau^n)\Big)-f\left(\tau,\mathbf{X}_\tau^{n-1},u_{n-1}(\tau,\mathbf{X}_\tau^{n-1}),\nabla_x u_{n-1}(\tau,\mathbf{X}_\tau^{n-1})\right)\right)\,\mathrm{d}\tau\right]\right| \\
&\le L\,\mathbb{E}^{t,x}\left[\big|\mathbf{X}_{T}^n-\mathbf{X}_{T}^{n-1}\big|\right]
      + L\,\mathbb{E}^{t,x}\left[\int_t^T\left(1+M+\frac{M}{\sqrt{T-\tau}}\right)
        \big|\mathbf{X}_\tau^n-\mathbf{X}_\tau^{n-1}\big|\,\mathrm{d}\tau \right]\\
&\quad + L\,\mathbb{E}^{t,x}\left[\int_t^T
        \left(\left|u_n(\tau,\mathbf{X}_\tau^{n-1})-u_{n-1}(\tau,\mathbf{X}_\tau^{n-1})\right|
        +\left|\nabla_x u_n(\tau,\mathbf{X}_\tau^{n-1})-\nabla_x u_{n-1}(\tau,\mathbf{X}_\tau^{n-1})\right|\right)\mathrm{d}\tau\right].
\end{aligned}
\end{equation}
Applying Lemma~\ref{lem:moment} and the definition of $\|\cdot\|_\beta$ bounds the last integral by
\begin{displaymath}
C\|u_n-u_{n-1}\|_\beta\int_t^T e^{\beta(T-\tau)}\left(1+|x|\right)\,\mathrm{d}\tau
\le \frac{C}{\beta}\,e^{\beta(T-t)}\left(1+|x|\right)\,\|u_n-u_{n-1}\|_\beta .
\end{displaymath}
Using Lemma~\ref{lem:X} to handle the terms involving $\mathbf{X}_T$ and $\mathbf{X}_\tau$, we deduce
\begin{equation}\label{eq:case1}
|u_{n+1}(t,x)-u_n(t,x)|
\le \frac{C}{\sqrt{\beta}}\,e^{\beta(T-t)}\left(1+|x|\right)\,\|u_n-u_{n-1}\|_\beta .
\end{equation}

\medskip\noindent\textbf{Case 2: $t+h<T$.}
We now split the interval at $t+h$ and employ the representations
\begin{align}
u_{n+1}(t,x)&=\mathbb{E}^{t,x}\left[u_n(t+h,\mathbf{X}_{t+h}^n)
      +\int_t^{t+h}f\Big(\tau,\mathbf{X}_\tau^n,u_n,\nabla_x u_n\Big)\,\mathrm{d}\tau\right], \label{eq:rep_n+1}\\
u_n(t,x)&=\mathbb{E}^{t,x}\left[u_{n-1}(t+h,\mathbf{X}_{t+h}^{n-1})
      +\int_t^{t+h}f\Big(\tau,\mathbf{X}_\tau^{n-1},u_{n-1},\nabla_x u_{n-1}\Big)\,\mathrm{d}\tau\right]. \label{eq:rep_n}
\end{align}
Subtracting and proceeding with the same Lipschitz estimates yields
\begin{equation}\label{eq:diff_h}
\begin{aligned}
\big|u_{n+1}(t,x)-u_n(t,x)\big|
&\le \mathbb{E}^{t,x}\left[\bigl|u_n(t+h,\mathbf{X}_{t+h}^n)-u_n(t+h,\mathbf{X}_{t+h}^{n-1})\bigr|\right]
   + \mathbb{E}^{t,x}\left[\bigl|u_n(t+h,\mathbf{X}_{t+h}^{n-1})-u_{n-1}(t+h,\mathbf{X}_{t+h}^{n-1})\bigr|\right] \\
&\quad + L\,\mathbb{E}^{t,x}\left[\int_t^{t+h}\left(1+M+\frac{M}{\sqrt{T-\tau}}\right)
        |\mathbf{X}_\tau^n-\mathbf{X}_\tau^{n-1}|\,\mathrm{d}\tau\right] \\
&\quad + L\,\mathbb{E}^{t,x}\left[\int_t^{t+h}
        \Bigl(\big|u_n(\tau,\mathbf{X}_\tau^{n-1})-u_{n-1}(\tau,\mathbf{X}_\tau^{n-1})\big|
        +\big|\nabla_x u_n-\nabla_x u_{n-1}\big|\Bigr)\mathrm{d}\tau\right].
\end{aligned}
\end{equation}
For the first term on the right, we invoke the gradient bound $|\nabla_x u_n|\le M$ (Lemma~\ref{lem:boundofgu}) together with Lemma~\ref{lem:X}:
\begin{displaymath}
\mathbb{E}^{t,x}\left[\big|u_n(t+h,\mathbf{X}_{t+h}^n)-u_n(t+h,\mathbf{X}_{t+h}^{n-1})\big|\right]
\le M\,\mathbb{E}^{t,x}\left[\big|\mathbf{X}_{t+h}^n-\mathbf{X}_{t+h}^{n-1}\big|\right]
\le \frac{C}{\sqrt{\beta}}\,e^{\beta(T-t)}\left(1+\big|x\big|\right)\,\big\|u_n-u_{n-1}\big\|_\beta .
\end{displaymath}
The second term is controlled directly via the $|\cdot|_\beta$-norm, i.e.,
\begin{displaymath}
\mathbb{E}^{t,x}\left[\big|u_n(t+h,\mathbf{X}_{t+h}^{n-1})-u_{n-1}(t+h,\mathbf{X}_{t+h}^{n-1})\big|\right]
\le e^{\beta(T-t-h)}\left(1+\big|x\big|\right)\,\big|u_n-u_{n-1}\big|_\beta .
\end{displaymath}
The remaining integrals are estimated exactly as in Case~1, but now over $[t,t+h]$; they contribute an additional
$\frac{C}{\sqrt{\beta}}e^{\beta(T-t)}(1+|x|)\|u_n-u_{n-1}\|_\beta$.
Gathering all bounds, we arrive at
\begin{equation}\label{eq:case2}
\big|u_{n+1}(t,x)-u_n(t,x)\big|
\le \left(\frac{C}{\sqrt{\beta}} + e^{-\beta h}\right) e^{\beta(T-t)}\left(1+\big|x\big|\right)\,\big\|u_n-u_{n-1}\big\|_\beta .
\end{equation}

From \eqref{eq:case1} and \eqref{eq:case2} we see that the estimate
\begin{equation}\label{eq:common}
\big|u_{n+1}(t,x)-u_n(t,x)\big|
\le \left(\frac{C}{\sqrt{\beta}}+e^{-\beta h}\right)e^{\beta(T-t)}\left(1+\big|x\big|\right)\,\big\|u_n-u_{n-1}\big\|_\beta
\end{equation}
holds in both cases. Dividing \eqref{eq:common} by $(1+|x|)e^{\beta(T-t)}$ and taking the supremum over $(t,x)$ gives
\begin{equation}\label{eq:norm}
\big|u_{n+1}-u_n\big|_\beta \le \left(\frac{C}{\sqrt{\beta}} + e^{-\beta h}\right) \big\|u_n-u_{n-1}\big\|_\beta .
\end{equation}
Lemma~\ref{lem:gradientE} provides the analogous gradient estimate:
\begin{equation}\label{eq:grad}
\big|\nabla_x u_{n+1}-\nabla_x u_n\big|_\beta\le\frac{C}{\sqrt{\beta}}\,\big\|u_n-u_{n-1}\big\|_\beta .
\end{equation}
Adding \eqref{eq:norm} and \eqref{eq:grad}, we choose $\beta$ sufficiently large and then, for that fixed $\beta$, choose $h>0$ large enough so that
\begin{displaymath}
\big\|u_{n+1}-u_n\big\|_\beta \le K\,\big\|u_n-u_{n-1}\big\|_\beta \qquad\text{with } 0<K<1.
\end{displaymath}
Consequently, for any $m,n\ge1$,
\begin{displaymath}
\big\|u_{n+m}-u_n\big\|_\beta \le \sum_{k=0}^{m-1} K^{n+k-1}\big\|u_1-u_0\big\|_\beta
\le C K^n \big\|u_1-u_0\big\|_\beta,
\end{displaymath}
which proves that $\{u_n\}$ is Cauchy in the complete space $\mathbb{X}_\beta$. Hence there exists $u^*\in\mathbb{X}_\beta$ such that $u_n\to u^*$ in $\mathbb{X}_\beta$; in particular, $|\nabla_x u_n-\nabla_x u^*|_\beta\to0$.

\medskip\noindent\textbf{Stage 2: Identification of the limit.}
The terminal condition is inherited immediately because each $u_n$ satisfies $u_n(T,\cdot)=g(\cdot)$, so $u^*(T,\cdot)=g(\cdot)$. Let $\mathbf{X}^*_t$ denote the solution of the SDE with coefficients $\mu(\cdot,u^*,\nabla_x u^*)$ and $\sigma(\cdot,u^*)$. Fix $(t,x)$ and choose $h>0$ small enough that $t+h<T$. From the iteration we have
\begin{displaymath}
u_{n+1}(t,x) = \mathbb{E}^{t,x}\left[u_n(t+h,\mathbf{X}_{t+h}^n)
      + \int_t^{t+h} f\Big(\tau,\mathbf{X}_\tau^n,u_n(\tau,\mathbf{X}_\tau^n),\nabla_x u_n(\tau,\mathbf{X}_\tau^n)\Big)\,\mathrm{d}\tau\right].
\end{displaymath}
Sending $n\to\infty$, the convergence established in Stage~1, together with Lemma~\ref{lem:moment} and Lemma~\ref{lem:X}, justifies passage to the limit (the arguments are identical to those in the estimates above, with $u_n$ replaced by $u^*$ and $\mathbf{X}^n$ by $\mathbf{X}^*$). We obtain
\begin{equation}\label{eq:dp}
u^*(t,\mathbf{x}) = \mathbb{E}^{t,x}\left[u^*(t+h,\mathbf{X}_{t+h}^*)
      + \int_t^{t+h} f\Big(\tau,\mathbf{X}_\tau^*,u^*(\tau,\mathbf{X}_\tau^*),\nabla_x u^*(\tau,\mathbf{X}_\tau^*)\Big)\,\mathrm{d}\tau\right].
\end{equation}
Since $h$ can be taken arbitrarily small, \eqref{eq:dp} is precisely the dynamic programming principle for the PDE \eqref{eq:PDEs}. A standard application of It\^o's formula then shows that $u^*$ satisfies the equation. Uniqueness follows from the contraction estimate already obtained; this completes the proof.
\end{proof}

\section{Numerical Results}
\label{Sec:NR}
To comprehensively assess the iSMART algorithm \ref{alg:ismart}, we consider three representative classes of martingale-based PDEs: a linear reaction-diffusion equation with sharp gradients, semi-linear problems of Burgers type, and a fully nonlinear HJB equation. These experiments demonstrate its high accuracy, robustness, and the broad applicability across different PDE settings.

As Theorem~\ref{thm:converge} and its proof suggest, the exact satisfaction of the terminal condition for $u_{\theta}(t,x)$ plays a crucial role in guaranteeing that the iterative scheme converges to the solution of \eqref{eq:PDEs}. Motivated by this observation, we design the neural network architecture for iSMART using a soft boundary-matching ansatz:
\begin{equation}\label{eq:NNA}
u_{\theta}(t,x) = N_{\theta}(t,x)\left(1 - e^{-\gamma(T-t)}\right) + g(x)\,e^{-\gamma(T-t)},
\end{equation}
where $N_{\theta}(t,x)$ denotes a fully connected feedforward neural network parameterized by $\theta$. The exponential weight $e^{-\gamma(T-t)}$, with $\gamma\in[3,6]$ (set to $5.0$ in our implementation), provides a smooth transition from the terminal boundary to the domain interior.
%We remark that the boundary conditions in Algorithm~\ref{alg:ismart} can be relaxed by adding an auxiliary terminal penalty loss $\mathcal{L}_{\mathrm{terminal}}$. This penalty-based approach avoids explicit architectural constraints, but typically requires a higher density of spatio-temporal sampling points and more SDE trajectories to achieve satisfactory boundary convergence and stability.
In our implementation of Algorithm~\ref{alg:ismart}, the common hyperparameters are set as $N=20$, $M=4096$, $K=1$, the spatial mixture weights $w_k$ in \eqref{eq:spatialsample} as $[0.01,0.05,0.1,0.5,1.5,2.5]$, and an initial learning rate as $10^{-3}$.

To quantify solution quality, we adopt the discrete relative $L^2$ error
\begin{equation}\label{eq:L2referrexact}
\mathrm{Relative}~L^2~\mathrm{Error}= \frac{\sqrt{\sum_{x \in D} |u_{\mathrm{pred}}(t,x) - u_{\mathrm{ref}}(t,x)|^2}}{\sqrt{\sum_{x\in D} |u_{\mathrm{ref}}(t,x)|^2}},
\end{equation}
where $D=\{x_i\}_{i=1}^{N}$ denotes the set of spatial points, $u_{\mathrm{pred}}$ and $u_{\mathrm{ref}}$ denote the numerical and reference solutions, respectively.

All computations were carried out on a workstation equipped with an Intel\textsuperscript{\textregistered} Core\texttrademark\ i9-14900K processor (3.20\,GHz) and a single NVIDIA GeForce RTX 4090 GPU.

%+++++++++++++++++++++++++++++++++++++++++++++++++++++++++++++++++
%        Example 1
%+++++++++++++++++++++++++++++++++++++++++++++++++++++++++++++++++
\subsection{Example 1: Linear reaction-diffusion equations with sharp gradients}
We begin by testing iSMART on a linear parabolic equation from \cite{Cai26b}. This problem combines a oscillatory terminal condition with a spatially varying drift, posing a clear challenge for numerical methods, especially in high dimensions:
\begin{equation}\label{prob:expline}
\left\{
\begin{aligned}
&\left(\partial_t+\mu^\top\partial_x+\frac{\vartheta^2}{2}\sum_{i=1}^d\partial_{x_i}^2\right)u(t,x)=0,
    &&\quad (t,x)\in[0,T)\times\mathbb{R}^d,\\
&u(T,x)=\frac{1}{d}\sum_{i=1}^d\Bigl(\tanh(x_i)+\cos(10x_i)\Bigr),
    &&\quad x\in\mathbb{R}^d.
\end{aligned}
\right.
\end{equation}
We choose $d=200$, and $T=2$, $\vartheta^2=0.1$. The drift coefficient is $\mu(t,x)=c~\left[\tanh(10x_1),\cdots,\tanh(10x_d)\right]^\top$, where $c>0$ is a constant. The terminal data contains a oscillatory component $\cos(10 x_i)$ while $\mu$ varies sharply near $x=0$. When $c\gg \vartheta$ the solution develops a steep gradient at the origin $(0,0)$. Simultaneously handling high dimensionality, oscillatory terminal data, and a locally sharp transition makes this example a demanding benchmark for both accuracy and robustness.

\begin{figure}[!h]
  \centering
  \begin{minipage}{0.95\textwidth}
    \centering
    \includegraphics[width=1\textwidth]{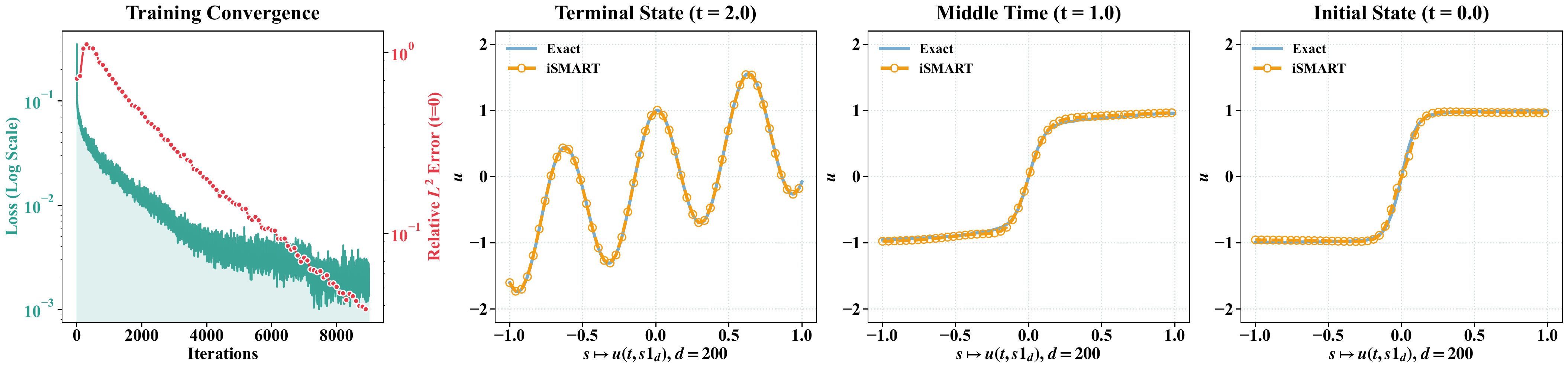}
  \end{minipage}
  \begin{minipage}{0.95\textwidth}
    \centering
    \includegraphics[width=1\textwidth]{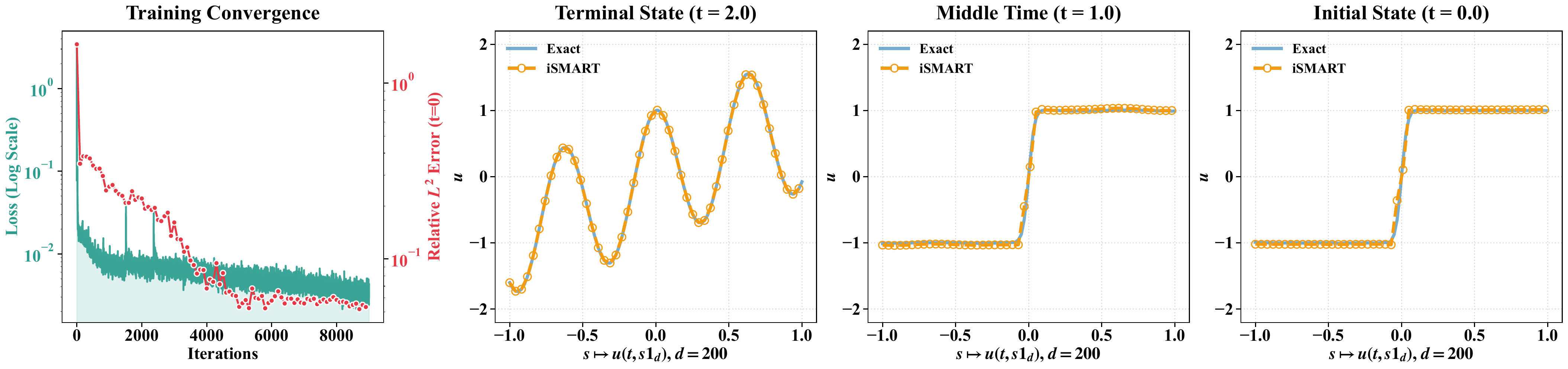}
  \end{minipage}
  \caption{Validation and quantitative assessment of the iSMART algorithm for the 200-dimensional Problem~\eqref{prob:expline}. The figure is split into two rows: the top row corresponds to $c=1$, while the bottom row corresponds to $c=5$. In the second to forth subfigures in each row, the numerical solution is compared with the reference solution along $s\mapsto s\mathbf{1}_d$ at representative snapshots $t=0.0$, $1.0$, $2.0$. In the first subfigure of each row, the loss histories and $L^2$ relative errors defined in~\eqref{eq:L2referrexact} is shown (in log scale) by green line and red dots, respectively. }
  \label{fig:LinearAD}
\end{figure}

We run iSMART with $6$ hidden layers and $\mathrm{SiLU}$ activation functions. Following the architecture in \cite{Cai26b}, the number of neurons in each hidden layer is 410. To better capture the underlying dynamics, we adopt path generation Method~$(II)$ for sampling trajectories. We consider both $c=1$ and $c=5$ in the drift coefficient.

The reference solution is constructed by Feynman-Kac formula and sufficient Monte Carlo samples. By the Feynman-Kac formula, the solution admits the probabilistic representation
\begin{equation}\label{eq:lexp}
u(t,x)=\mathbb{E}\left[u\big(T,\mathbf{X}_T\big)\right],
\end{equation}
where the process $\mathbf{X}_s$ follows
$\mathbf{X}_s=x+\int_t^s\mu\left(\tau,\mathbf{X}_\tau\right)\,\mathrm{d}\tau
+\sigma\left(\mathbf{B}_s-\mathbf{B}_t\right)$ for $s\in[t,T]$, and $\mathbf{B}$ is a standard $d$-dimensional Brownian motion. The reference solution is then generated by Monte Carlo with $10^6$ independent Euler-Maruyama paths with time step of $1/50$.

To visualize the high-dimensional results, we compare the computed solution and the reference solution along a one dimensional curve in $\mathbb{R}^d$. Let $\mathbf{1}_d=(1,1,\cdots,1)^{\top}\in\mathbb{R}^d$, we compare $u_{\mathrm{pred}}(t,s\mathbf{1}_d)$ and $u_{\mathrm{ref}}(t,s\mathbf{1}_d)$ for $s\in [-1,1]$ at $t=0$, $t=1$, and $t=2$.

Figure~\ref{fig:LinearAD} reports the numerical solutions against the reference for $c=1$ (first row) and $c=5$ (second row).  In both cases, iSMART accurately captures the sharp transition near $x=0$ and the predicted curves are almost indistinguishable from the reference ones. Quantitatively, the relative $L^2$ error remains below $5\%$ across all tests, confirming that the algorithm retains high accuracy for high dimensional problems.

%+++++++++++++++++++++++++++++++++++++++++++++++++++++++++++++++++
%        Example 2
%+++++++++++++++++++++++++++++++++++++++++++++++++++++++++++++++++
\subsection{Example 2: Semi-linear parabolic PDEs of Burgers type}
\label{exp:Ne2}
Next we are going to solve a semi-linear parabolic PDEs of Burgers type
\begin{equation}\label{eq:200DsemiLPDE}
\left\{\begin{aligned}
&\frac{\partial u(t,x)}{\partial t}+u(t,x)\sum_{i=1}^{d}\frac{\partial u(t,x)}{\partial x_i}+\frac{1}{2}\sum_{i=1}^{d} \frac{\partial^2 u(t,x)}{\partial x_i^2}+f(t,x)=0,
&&(t,x)\in[0,T)\times\mathbb{R}^{d},\\
&u(T,x)=\sin\left(\frac{1}{d}\sum_{i=1}^d x_i\right),                               &&x\in\mathbb{R}^{d},
\end{aligned}\right.
\end{equation}
with the exact solution $u(t,x)=\sin\left(\frac{1}{d}\sum_{i=1}^{d} x_i\right)e^{T-t}$. Writing  $\bar{x}=\frac{1}{d}\sum_{i=1}^{d} x_i$, the source term becomes $f(t,x)= u\cdot(1+\frac{1}{2d}-\cos(\bar{x})e^{T-t})$. The main numerical challenge comes from the non-linear convective term. We first solve the one-dimensional non-linear equation and then extend it to high dimensional case.

%+++++++++++++++++++++++++++++++++++++++++++++++++++++++++++++++++
%        Example 2.1
%+++++++++++++++++++++++++++++++++++++++++++++++++++++++++++++++++
\subsubsection{One-dimensional illustration and sampling strategies verification}
We first examine the one-dimensional case to clearly demonstrate the two path-generation methods employed within iSMART. When $d=1$, equation \eqref{eq:200DsemiLPDE} reduces to
\begin{equation}\label{examp:Semi-linear1D}
\left\{\begin{aligned}
&\frac{\partial u(t,x)}{\partial t}+\frac{1}{2}\frac{\partial^2 u(t,x)}{\partial x^2}
    +u(t,x)\frac{\partial u(t, x)}{\partial x}=f(t, x), && (t,x)\in[0,T)\times\mathbb{R},\\
&u(T,x)=\sin(x),                                        && x\in\mathbb{R},
\end{aligned}\right.
\end{equation}
with the exact solution $u(t,x)=\sin(x)e^{T-t}$ and the source term $f(t,x)=\sin(x)\cos(x)e^{2(T-t)}-1.5\sin(x)e^{T-t}$.

We choose $T=1$. For this example, we can treat the equation as
\[\frac{\partial u(t,x)}{\partial t}+\frac{1}{2}\frac{\partial^2 u(t,x)}{\partial x^2} = f(t,x) - u(t,x)\frac{\partial u(t, x)}{\partial x}\]
and apply Method~($I$) to generate trajectories. Meanwhile, we can also use Method~($II$) to generate trajectories based on the current $u_{\theta_n}(t,x)$ at each iteration. We run iSMART with the same network architecture for both sampling methods. The DNN has 4 hidden layers of $128$ neurons, the $\mathrm{SiLU}$ activation. Learning rate is $10^{-3}$, while batch size is set as $512$, and $N_\mathrm{Int}=20$ for Method ($I$). The training is carried out for 6001 iterations. Figures~\ref{fig:semilinearMI} and \ref{fig:semilinearMII} show the predicted and exact solutions at $t=0,0.5,1$ together with the absolute error profiles. The numerical of both sampling strategies have relative $L^2$ errors and maximum absolute errors below $2\%$. This agreement confirms that sampling strategies Method~($I$) and Method~($II$) both guarantees the high accuracy of iSMART.

\begin{figure}[!h]
  \centering
  \begin{minipage}{0.95\textwidth}
    \centering
    \includegraphics[width=1\textwidth]{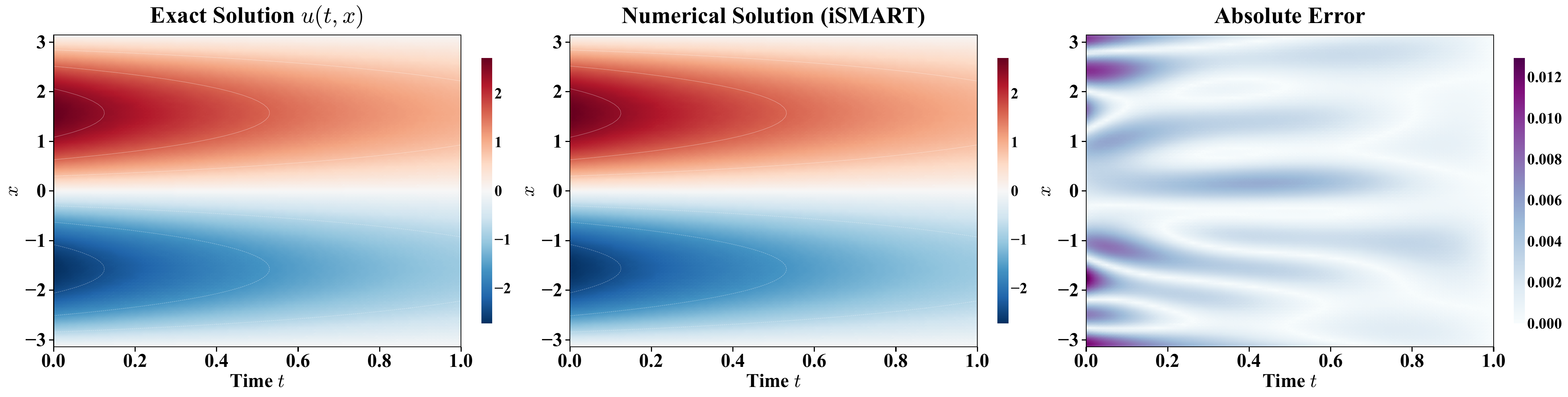}
  \end{minipage}
  \begin{minipage}{0.95\textwidth}
    \centering
    \includegraphics[width=1\textwidth]{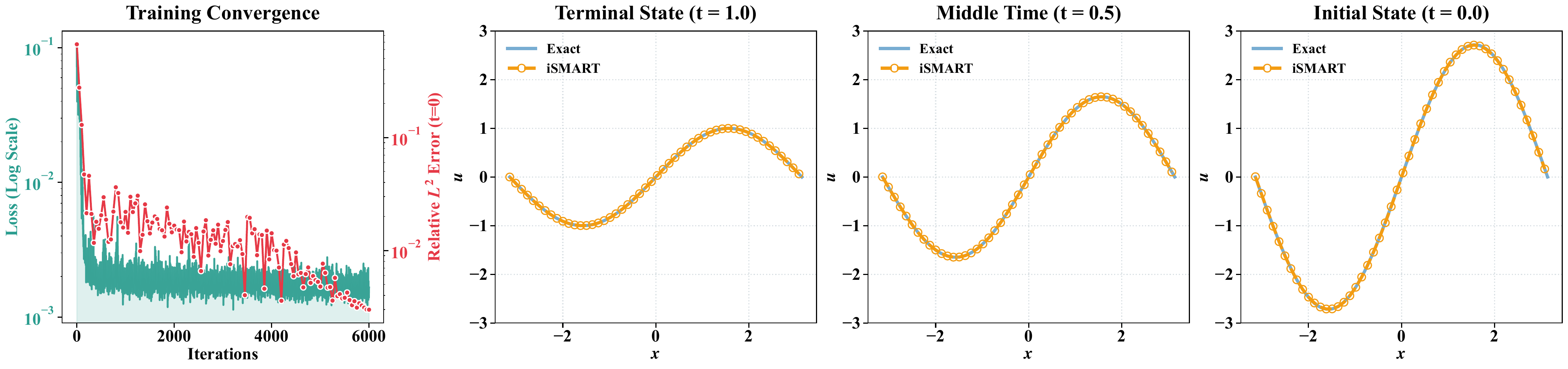}
  \end{minipage}
 \caption{Validation and quantitative assessment of the iSMART algorithm for the one-dimensional Problem \eqref{examp:Semi-linear1D} with sampling performed using Method~($I$). The numerical solution is compared with the exact solution over time at representative snapshots, including the initial ($t=0$), intermediate ($t=0.5$), and final ($t=1$) states, demonstrating excellent agreement. The associated absolute error profiles, loss histories, and $L^2$ relative errors defined in \eqref{eq:L2referrexact} further confirm the accuracy, stability, fast convergence, and high fidelity of the proposed algorithm.}
  \label{fig:semilinearMI}
\end{figure}

We remark that the wavy pattern in the absolute error visualizations of Figures~\ref{fig:semilinearMII} and~\ref{fig:semilinearCompare} is primarily due to coverage of the sampled trajectories. Since we only samples a few points $(t_i, x_i)$ in temporal-spatial domain and generate trajectories starting from them, these trajectories cannot cover the entire temporal-spatial domain equably. The local error can be relative large when there are only a small number of points lie in the local region. This phenomenon is also discussed in \cite{Cai26b}.

\begin{figure}[!h]
  \centering
  \begin{minipage}{0.95\textwidth}
    \centering
    \includegraphics[width=\textwidth]{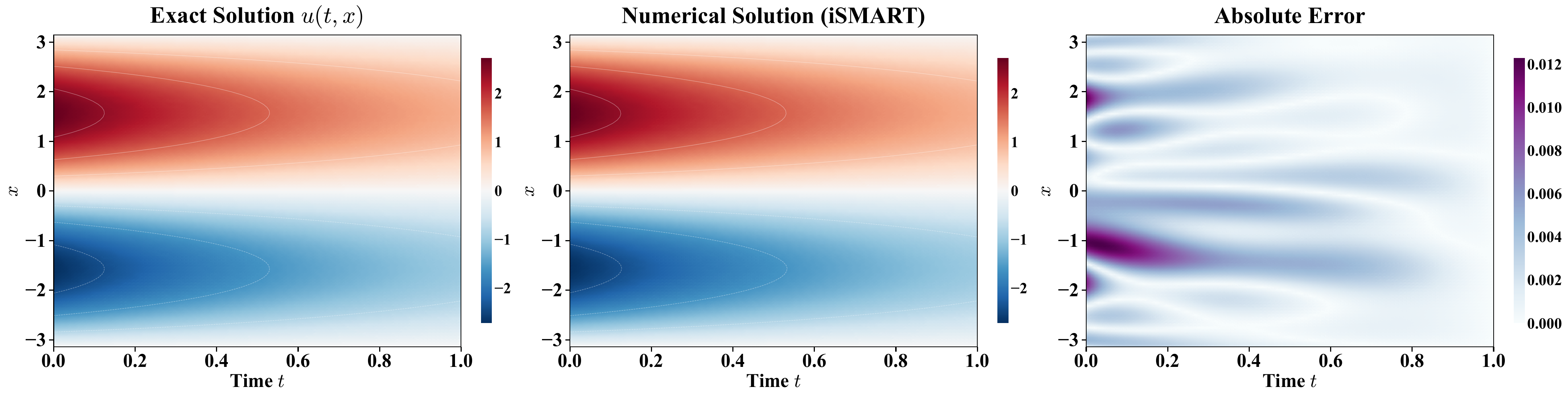}
  \end{minipage}
  \begin{minipage}{0.95\textwidth}
    \centering
    \includegraphics[width=1\textwidth]{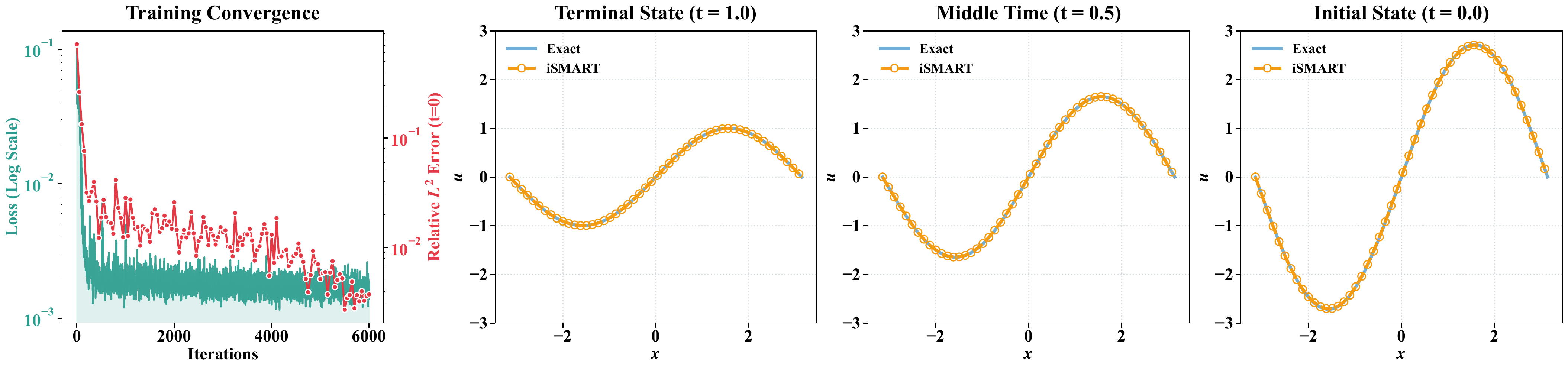}
  \end{minipage}
  \caption{Validation and quantitative performance evaluation of the iSMART algorithm for the 1D Problem \eqref{examp:Semi-linear1D} using sampling via Method~($II$). The temporal evolution of the numerical solution is benchmarked against the exact solution at $t=0$, $0.5$, and $1$, showing excellent agreement. The corresponding absolute error profiles, loss trajectories, and $L^2$ relative errors defined by \eqref{eq:L2referrexact} further demonstrate the algorithm's accuracy, stability, rapid convergence, and high fidelity.}
  \label{fig:semilinearMII}
\end{figure}

To further compare path generating strategies Method $(I)$, $(II)$ and assess the efficiency of iSMART, we compare it against the DeepMartNet algorithm on the same 1D semi-linear problem. Both methods are trained under identical configurations as described before. For each starting point, $M=100\,000$ full paths are generated to ensure a fair comparison. A direct comparison with DRDM\footnote{The code is available at \url{https://github.com/sx-fang/DRDM}.} is not included in this study, as its original implementation requires GPU memory and storage resources that are beyond our current hardware capacity.

Figure~\ref{fig:semilinearCompare} presents a side-by-side comparison of the numerical solutions of DeepMartNet and iSMART at time snapshots $t=0, 0.5, 1$. We find that both sampling strategies can help DeepMartNet and iSMART. DeepMartNet is clearly effective for this semi-linear equation, yet iSMART consistently yields more accurate solutions.

The GPU time and error statistics are shown in Table~\ref{tab:gputimeSemiLinear}. The statistics shows that both path generation methods works for iSMART and DeepMartNet. Note that the drift and source term do not depend on $\nabla_x u$, Method ($II$) is derivative free. Thus it is significantly more efficient than Method ($I$), as we discussed in subsection 3.2.

The statistics shows that Method ($I$) helps both DeepMartNet and iSMART to achieve a smaller error level, but cost more computing time, while Method ($II$) can be more efficient for generating trajectories and leads to similar performance for iSMART. However, it pulls down the accuracy of DeepMartNet.

On the aspect of accracy, the iSMART method achieves similar relative $L^2$ errors on the order of $10^{-3}$ at $t=0, 0.5$ for both sampling strategies, while DeepMartNet's errors are at least one order of magnitude larger. Moreover, iSMART requires substantially less GPU time: $310.80$\,s (Method~$I$) and $149.41$\,s (Method~$II$) versus $535.54$\,s and $386.37$\,s for DeepMartNet. We remark that since the neural network structure for iSMART is set as \eqref{eq:NNA} to ensure the terminal condition, the error at $t=1.0$ is always 0. On the other hand, DeepMartNet ensures terminal condition by adding a penalty term in loss function \cite{Cai2026a}, and hence exhibits a small but nonzero discrepancy. This comparison highlights the advantage of iSMART: it avoids the nested expectation evaluation of DeepMartNet, thereby making more efficient use of the simulation data. The consistent gains in both accuracy and computational speed confirm that iSMART is a highly competitive solver for semi-linear PDEs.

\begin{figure}[!h]
  \centering
  \begin{minipage}{0.95\textwidth}
    \centering
    \includegraphics[width=\textwidth]{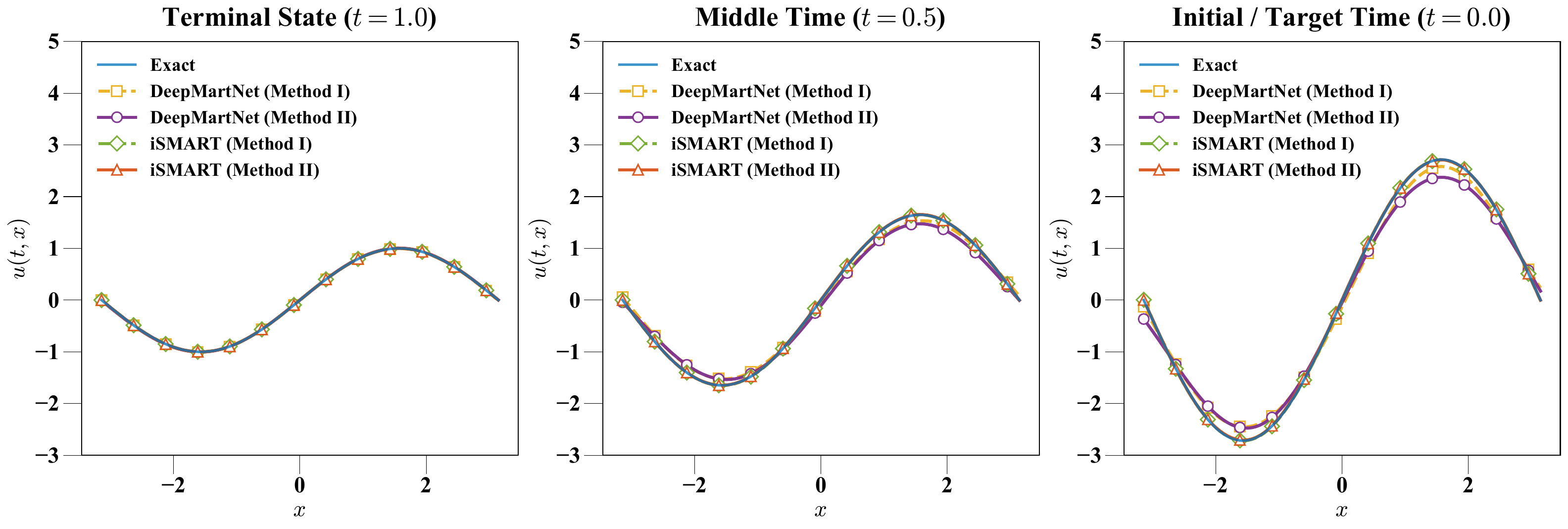}
    %\caption*{(b)}
  \end{minipage}
  \begin{minipage}{0.95\textwidth}
    \centering
    \includegraphics[width=\textwidth]{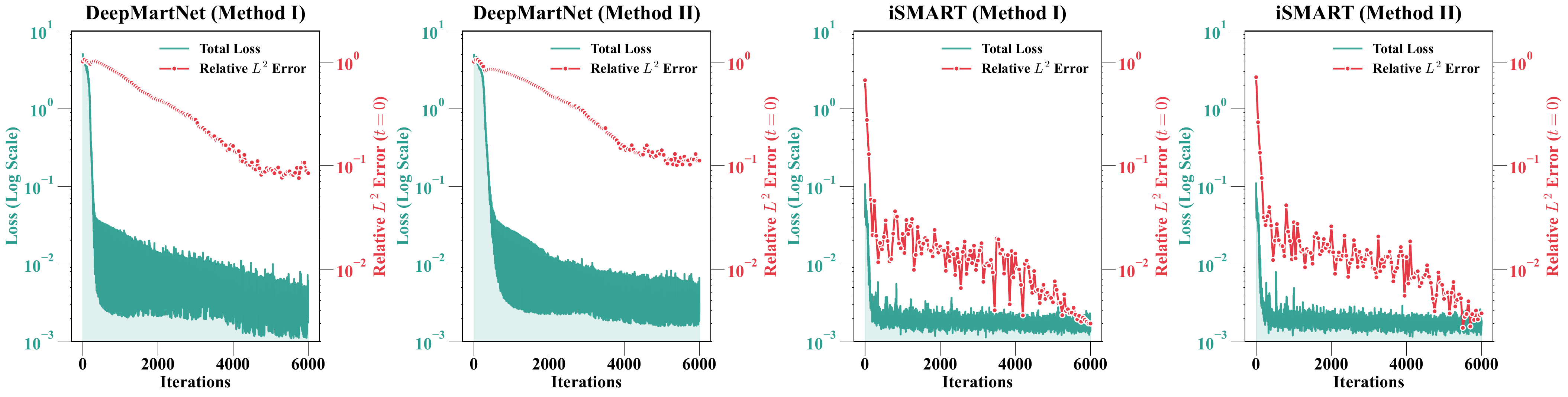}
    %\caption*{(a)}
  \end{minipage}
  \caption{Numerical performance comparison between DeepMartNet and iSMART for the 1D Problem \eqref{examp:Semi-linear1D}. The top row illustrates the agreement between the numerical solutions and the exact solution, where the two deep learning algorithms employ two sampling strategies Method~($I$) and Method~($II$). The lower row reports the corresponding loss histories (in green line) and relative $L^2$ errors (in red dots).}
  \label{fig:semilinearCompare}
\end{figure}

\begin{table}[!h]
\centering
\caption{Comparison of GPU time costs and relative $L_2$ error for DeepMartNet and iSMART with different path sampling strategies in solving 1D semi-linear problem \eqref{examp:Semi-linear1D}.}
\vspace{0.15cm}
\label{tab:gputimeSemiLinear}
\begin{tabular}{l c c c c}
\toprule
 & \multicolumn{2}{c}{DeepMartNet} & \multicolumn{2}{c}{iSMART} \\
\cmidrule(lr){2-3} \cmidrule(lr){4-5}
 & Method ($I$) & Method ($II$) & Method ($I$) & Method ($II$) \\
\midrule
GPU time costs (in seconds)         & 535.54\,s   & 386.37\,s   & 310.80\,s   & 149.41\,s   \\
\midrule
Relative $L^2$ error ($t=1.0$)      & 3.6077e-03  & 6.6317e-03  & 0.0000e+00  & 0.0000e+00  \\
Relative $L^2$ error ($t=0.5$)      & 8.4358e-02  & 1.0501e-01  & 1.7450e-03  & 3.3268e-03   \\
Relative $L^2$ error ($t=0.0$)      & 8.3291e-02  & 1.0926e-01  & 2.9814e-03  & 3.7465e-03   \\
\bottomrule
\end{tabular}
\end{table}

%+++++++++++++++++++++++++++++++++++++++++++++++++++++++++++++++++
%        Example 2.2
%+++++++++++++++++++++++++++++++++++++++++++++++++++++++++++++++++
\subsubsection{High-dimensitional case}
We next examine the Burgers-type semi-linear equation \eqref{eq:200DsemiLPDE} for $d=200$ dimensions. Because the exact solution depends only on the spatial mean $\bar{x}$, the problem serves as an ideal benchmark for verifying whether a high-dimensional solver can automatically discover this low-dimensional structure.

The structure of DNN is set as 4 hidden layers with 256 neurons in each layer, $\mathrm{SiLU}$ activation. Other setting are the same as previous experiment. We use iSMART to solve this high dimensional problem with path generation Method ($I$) and ($II$). The result is visualized along $\mathbf{1}_d$. Figure~\ref{fig:200DsemilinearMII} summarizes the performance of iSMART. The two rows correspond to Method~$(I)$ and Method~$(II)$. In all cases, the numerical and exact solutions are in close agreement, and the relative $L^{2}$ errors remain order $10^{-2}$.
%% 是对角线还是along \bar{x}？

\begin{figure}[!h]
  \centering
  \begin{minipage}{0.95\textwidth}
    \centering
    \includegraphics[width=\textwidth]{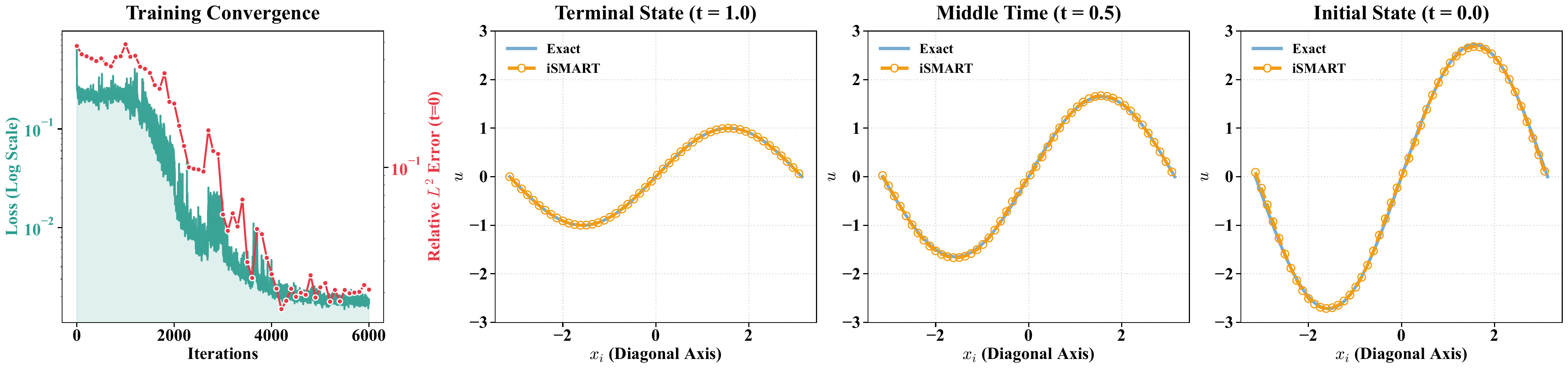}
    %\caption*{(b)}
  \end{minipage}
  \begin{minipage}{0.95\textwidth}
    \centering
    \includegraphics[width=\textwidth]{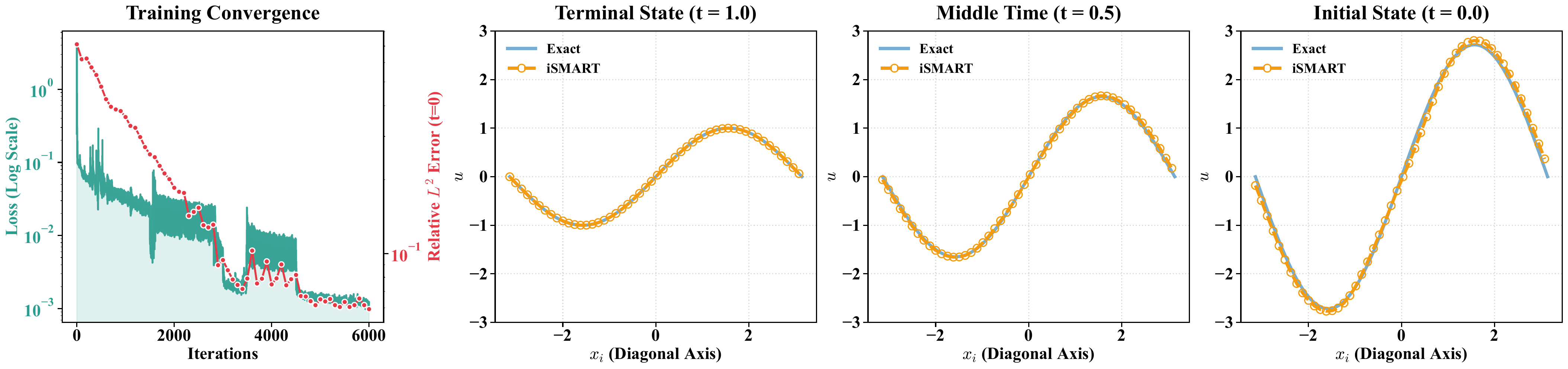}
    %\caption*{(a)}
  \end{minipage}
  \caption{Performance of iSMART on the 200-dimensional semi-linear problem \eqref{eq:200DsemiLPDE}.
           Top row: Sampling strategy Method~$(I)$.  Bottom row: Sampling strategy Method~$(II)$.
           From left to right: training loss (in green line) and relative $L^{2}$ error history (in red dots), solution profiles at $t=0,0.5,1$ along the main diagonal of the hypercube $[-1,1]^{200}$.}
  \label{fig:200DsemilinearMII}
\end{figure}

We also compare the performance of iSMART and DeepMartNet. Both algorithms share the same network architecture and training settings as described before. The results are provided in Figure~\ref{fig:200DsemilinearCompare} and Table~\ref{tab:200DgputimeSemiLinear}.

The upper panels of Figure~\ref{fig:200DsemilinearCompare} illustrates that iSMART approximates the exact solution more faithfully than DeepMartNet, regardless of the sampling strategy. The lower panels show that the training of iSMART is more stable and converge faster than DeepMartNet. The detailed quantitative results in Table~\ref{tab:200DgputimeSemiLinear} shows the accuracy of iSMART: it achieves relative $L^{2}$ errors of order $10^{-2}$–$10^{-3}$ at $t=0$ and $0.5$ for both sampling strategies, while DeepMartNet errors are roughly one order of magnitude larger. The zero error of iSMART at time $t=1$ dues to the construction \eqref{eq:NNA}, while DeepMartNet shows a small but non-negligible discrepancy as explained before.

\begin{figure}[!h]
  \centering
  \begin{minipage}{0.95\textwidth}
    \centering
    \includegraphics[width=\textwidth]{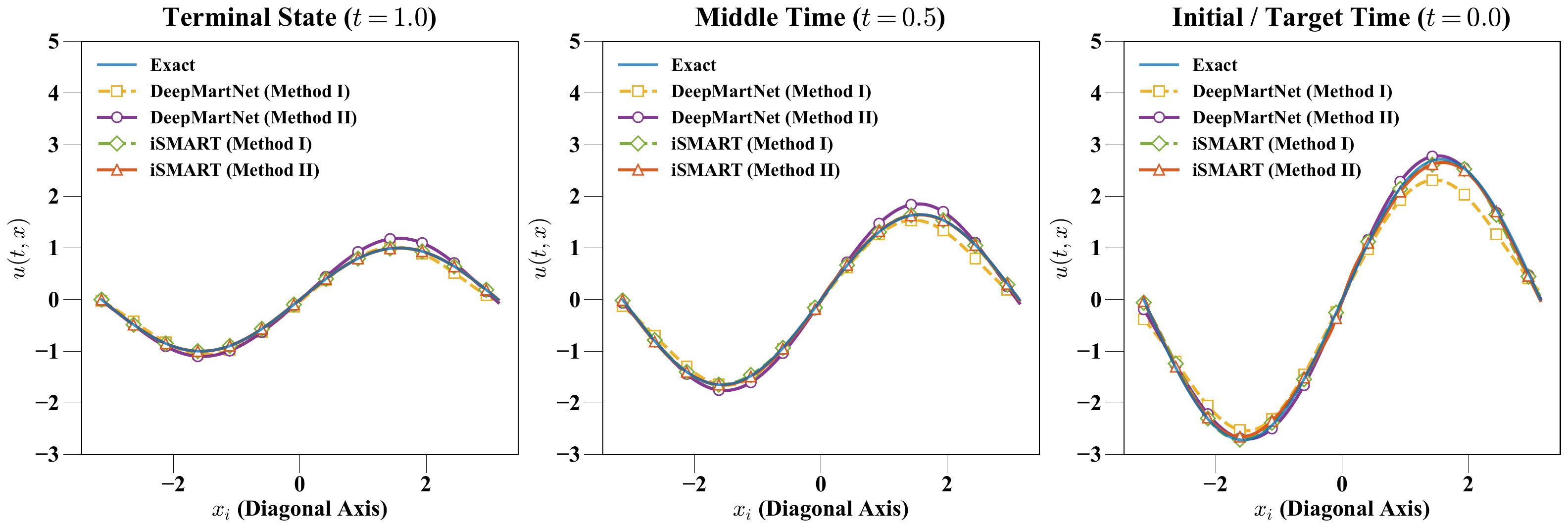}
    %\caption*{(b)}
  \end{minipage}
  \begin{minipage}{0.95\textwidth}
    \centering
    \includegraphics[width=\textwidth]{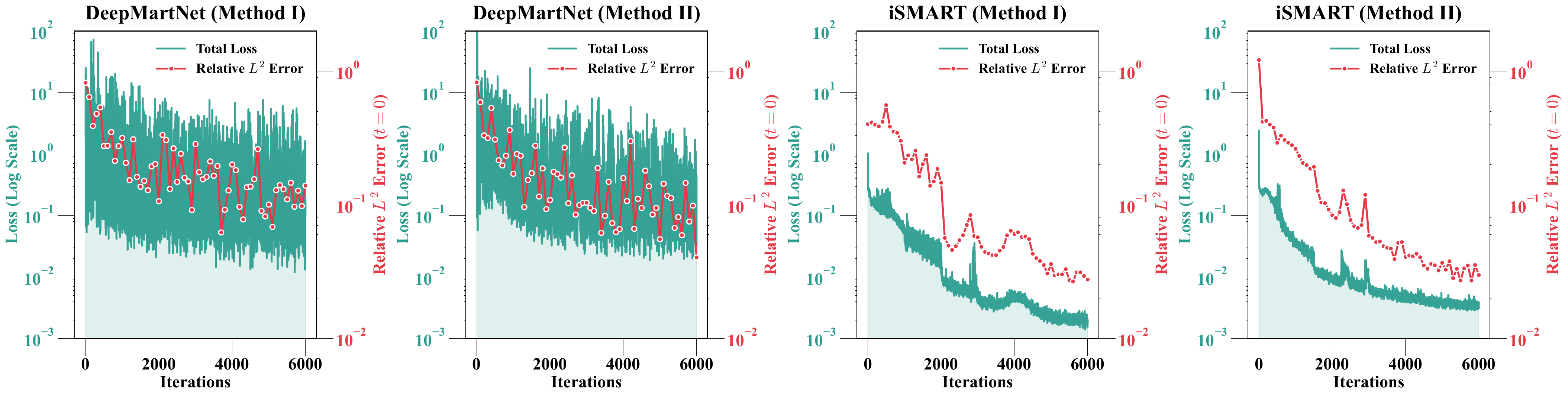}
    %\caption*{(a)}
  \end{minipage}
  \caption{Numerical performance comparison between DeepMartNet and iSMART for the 200-dimensional semi-linear problem \eqref{eq:200DsemiLPDE}. The top row illustrates the agreement between the numerical solutions and the exact solution, where the two deep learning algorithms employ two sampling strategies Method~($I$) and Method~($II$). The lower row reports the corresponding loss histories (in green line) and relative $L^2$ errors (in red dots).}
  \label{fig:200DsemilinearCompare}
\end{figure}

Comparing Table \ref{tab:gputimeSemiLinear} and \ref{tab:200DgputimeSemiLinear}, we find when the dimension increase from $d=1$ to $d=200$, the GPU time for both method increases only about $50\%$, indicating that they are both powerful for high dimensional PDE. Moreover, iSMART accomplishes higher accuracy with substantially lower GPU time. It cost only about $2/3$ GPU time of that of DeepMartNet. This advantage becomes particularly pronounced in high dimensions, where the nested expectation evaluation in DeepMartNet requires many more trajectories. iSMART, which does not require nested expectation, can be a scalable and efficient solver for high dimensional PDEs.

\begin{table}[!h]
\centering
\caption{GPU time and relative $L^{2}$ error for the 200-dimensional semi-linear problem \eqref{eq:200DsemiLPDE}. }
\vspace{0.15cm}
\label{tab:200DgputimeSemiLinear}
\begin{tabular}{l c c c c}
\toprule
 & \multicolumn{2}{c}{DeepMartNet} & \multicolumn{2}{c}{iSMART} \\
\cmidrule(lr){2-3} \cmidrule(lr){4-5}
 & Method ($I$) & Method ($II$) & Method ($I$) & Method ($II$) \\
\midrule
GPU Time Costs (in seconds)         & 699.21\,s   &487.13\,s    & 446.08\,s   &291.95\,s      \\
\midrule
Relative $L^2$ error ($t=1.0$)      & 8.8256e-02  & 1.3392e-01  & 0.0000e+00  & 0.0000e+00   \\
Relative $L^2$ error ($t=0.5$)      & 9.9271e-02  & 9.3032e-02  & 9.5502e-03  & 9.8838e-03   \\
Relative $L^2$ error ($t=0.0$)      & 1.3921e-01  & 4.0476e-02  & 2.7693e-02  & 2.9516e-02   \\
\bottomrule
\end{tabular}
\end{table}

%+++++++++++++++++++++++++++++++++++++++++++++++++++++++++++++++++
%        Example 3
%+++++++++++++++++++++++++++++++++++++++++++++++++++++++++++++++++
\subsection{Example 3: Nonlinear optimal control problem --HJB equations}
\label{Exp:HJB}
We solve a Hamilton–Jacobi–Bellman (HJB) equations raised in stochastic control to show the capability of iSMART for dealing with highly non-linearity. Consider a stochastic optimal control problem where the value function $u(t,x)$ satisfies the Hamilton-Jacobi-Bellman (HJB) equation
\begin{equation}\label{HJB-P}
\left\{\begin{aligned}
&\partial_t u(t,x)+\inf_{\kappa\in U}\Big\{\mathcal{L}^{\kappa}u(t,x)+c(t,x,\kappa)\Big\}=0,
                                                   &&\qquad (t,x)\in[0,T)\times \mathbb{R}^{d},\\
&u(T,x)=\Phi(x),                                   &&\qquad x\in\mathbb{R}^{d},
\end{aligned}\right.
\end{equation}
where the controlled infinitesimal generator \(\mathcal{L}^{\kappa}\) is $\mathcal{L}^{\kappa}:=\big(b+2\kappa\big)^{\top}\nabla_x
+\frac{1}{2}\mathrm{Tr}\,\Big\{\sigma\sigma^{\top}\,\nabla_x^{2}\Big\}$, with $\sigma=\sqrt{2 \delta}I$ and the running cost $c(t,x,\kappa)=\delta^{-2}|\kappa|^2$. Given $\delta>0$, constant vector field $b\in\mathbb{R}^{d}$, $T=1$, and control set  $U=\mathbb{R}^{d}$, the HJB equation is simplified to
\begin{equation}
\partial_t u+\inf_{\kappa\in \mathbb{R}^d}\left\{\big(b+2\kappa\big)^{\top}\nabla_xu
+\delta\Delta_x u+\delta^{-2}\big|\kappa\big|^2\right\}=0.
\end{equation}
The Hamiltonian is
$H(t,x,\nabla_x u,\Delta_x v)=\inf_{\kappa\in\mathbb{R}^d}\{(b+2\kappa)^\top\nabla_x u+\delta\Delta_x u+\delta^{-2}|\kappa|^2\}$. The optimal $\kappa^*$ can be found to be $\kappa^*(t,x)= -\delta^2\nabla_x u(t,x)$, and hence the HJB equation reduces to the following non-linear PDE:
\begin{equation}\label{eq:HJB_Simp}
\partial_t u+b^\top\nabla_x u-\delta^2\big|\nabla_x u\big|^2+\delta\Delta_x u=0,\qquad u(T,x)=\Phi(x).
\end{equation}
Equation \eqref{eq:HJB_Simp} contains the strongly nonlinear term $-\delta^2|\nabla_x u|^2$, which poses a fundamental challenge for PDE solvers.

\begin{figure}[!h]
  \centering
  \begin{minipage}{0.95\textwidth}
    \centering
    \includegraphics[width=\textwidth]{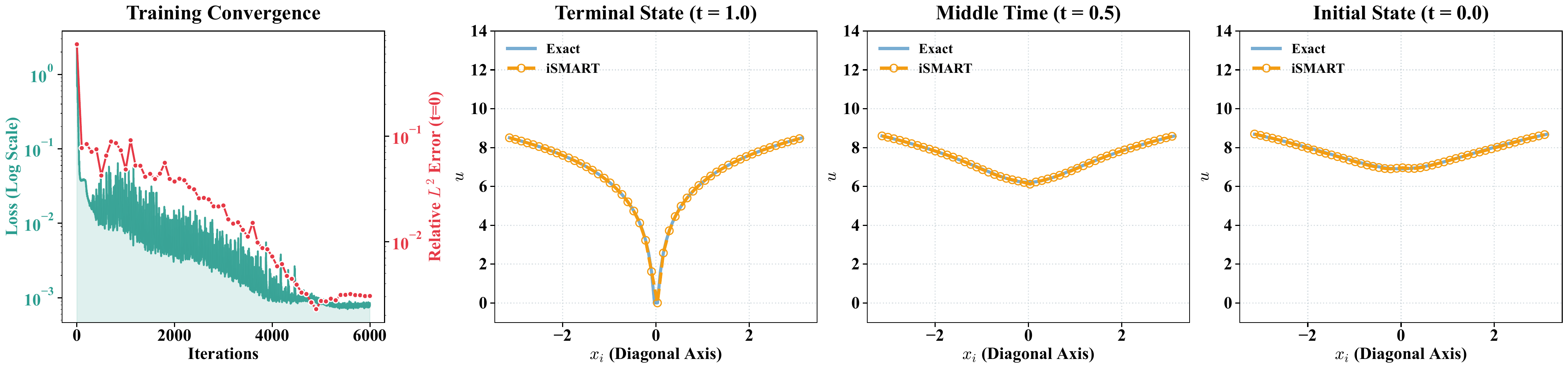}
    %\caption*{(b)}
  \end{minipage}
  \begin{minipage}{0.95\textwidth}
    \centering
    \includegraphics[width=\textwidth]{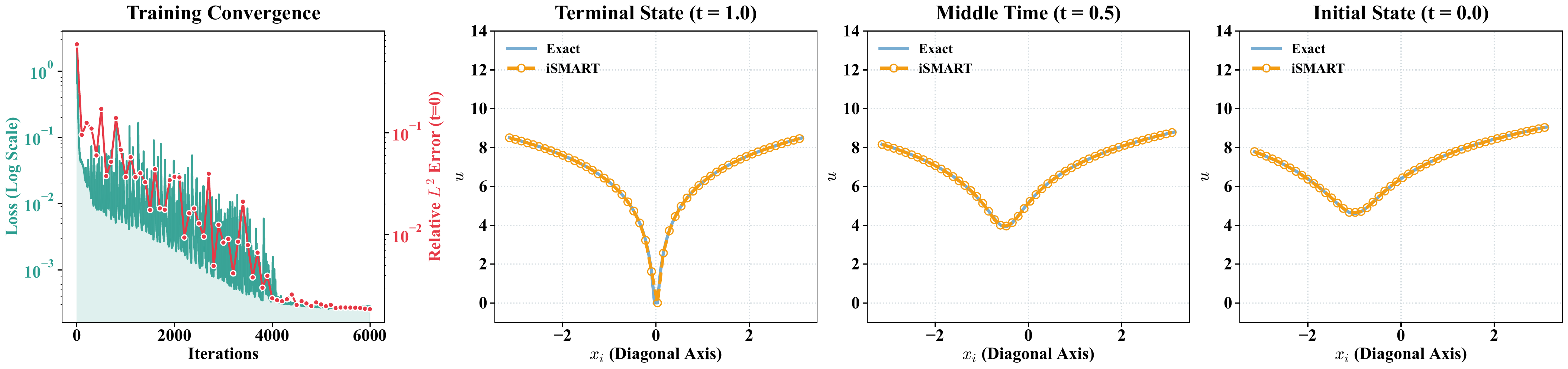}
    %\caption*{(a)}
  \end{minipage}
  \begin{minipage}{0.95\textwidth}
    \centering
    \includegraphics[width=\textwidth]{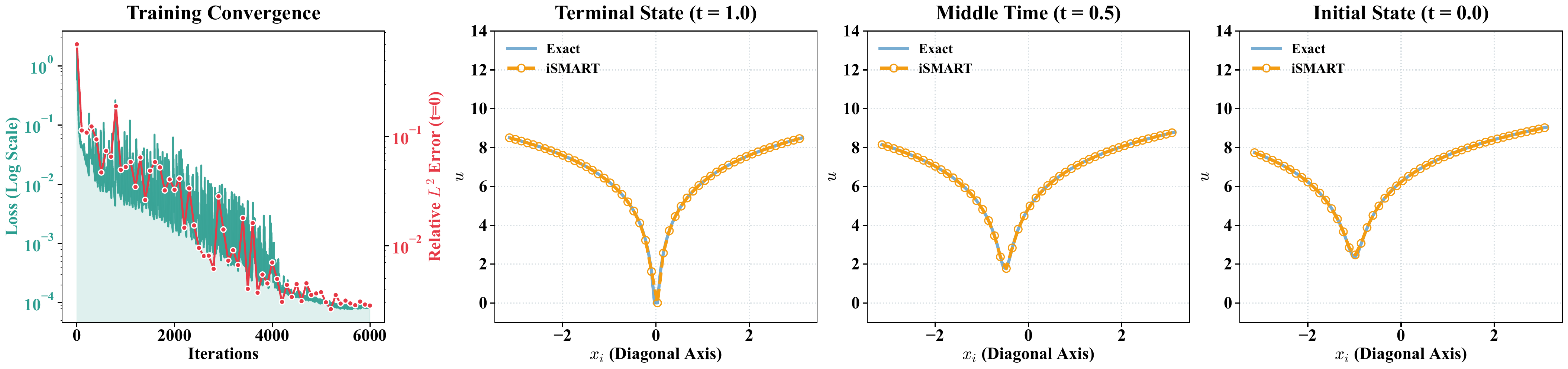}
    %\caption*{(a)}
  \end{minipage}
  \caption{Comparison between the iSMART numerical solutions and the exact solutions for the HJB equation \eqref{HJB-P} in the $1000$ dimensions. The first, second, and third rows correspond to \textbf{HJB-1}, \textbf{HJB-2-2a}, and \textbf{HJB-2-2b}, respectively, with paths generated by Method~($III$). In each row, the panels from left to right show the history of the training loss (in green line), the relative $L^2$ error, and the solution profiles at temporal snapshots $t=0, 0.5, 1$. The solution profiles $u(t, s\mathbf{1}_d), s\in [-3,3]$ are visualized.}
  \label{fig:up-down6}
\end{figure}

Applying the Cole-Hopf transformation and the Feynman-Kac formula, the solution to \eqref{eq:HJB_Simp} is given by
\begin{equation}\label{eq:exactHJB}
u(t,x)=-\delta^{-1}\ln\Big(\mathbb{E}\left[\exp\big(-\delta\,\Phi\left(\mathbf{X}_T\right)\big)\right]\Big),
\end{equation}
where process $\mathbf{X}_T$ is defined as $\mathbf{X}_T=x+b(T-t)+\sqrt{2\delta}\,\sqrt{T-t}\,\xi$ with $\xi\sim \mathcal{N}\left(0,I_d\right)$. We construct the reference solution of equation \eqref{eq:HJB_Simp} by form \eqref{eq:exactHJB} and $10^6$ Monte Carlo samples.

We solve this equation in $d=1000$ dimension. To assess the proposed algorithm comprehensively, we consider the following three different setups, which include smooth and oscillatory terminal costs, as well as a wide range of diffusion and the nonlinear effects.
\begin{description}
  \item[HJB-1:] A baseline setup with $b=0$, $\delta=1$, and terminal cost $\Phi(x)=\ln\big(\frac{1}{2}\big(1+|x|^2\big)\big)$;
  \item[HJB-2:] Drift $b=\mathbf{1}_d$ and the same $\Phi(x)$ as in \textbf{HJB-1}, with two small diffusion strengths that make the problem convection-dominated
               \begin{itemize}
                 \item \textbf{HJB-2a}: $\delta = 0.1$;\quad
                       \textbf{HJB-2b}: $\delta = 0.01$.
               \end{itemize}
  \item[HJB-3:] Drift $b=\mathbf{1}_d$, $\delta=1$, and a oscillatory terminal cost $\Phi(x)=\bar{g}(x-\mathbf{1}_d)$, where $\bar{g}(x):=\frac{1}{d}\sum_{i=1}^d\{\sin(x_i-\frac{\pi}{2})
      +\sin((0.1\pi+x_{i}^{2})^{-1})\}$. Two sub-cases with small diffusion coefficient:
                \begin{itemize}
                  \item \textbf{HJB-3a}: $\delta = 0.1$;\quad
                        \textbf{HJB-3b}: $\delta = 0.01$.
                \end{itemize}
\end{description}
We design the DNN a feedforward architecture consisting of $4$ hidden layers with a uniform width of $256$ neurons, equipped with Layer Normalization and $\mathrm{SiLU}$ activation functions. To handle the non-linearity efficiently, we use Method ($III$), namely,
\begin{displaymath}
\mathrm{d}\mathbf{X}_\tau
=\Big(\mu(\tau,\mathbf{X}_\tau,u_n(\tau,\mathbf{X}_\tau),\nabla_x u_n(\tau,\mathbf{X}_\tau))
-\alpha_n \delta^2\nabla_x u_n(\tau,\mathbf{X}_\tau)\Big) \,\mathrm{d}\tau
+\sigma(\tau,\mathbf{X}_\tau)\,\mathrm{d}\mathbf{B}_\tau
\end{displaymath}
to generate paths. The numerical results are visualized in Figures \ref{fig:up-down6} and \ref{fig:up-down7}.

\begin{table}[!h]
\centering
\caption{Comparison of training GPU time costs and relative $L^2$ error in solving 1000D \textbf{HJB-1} and \textbf{HJB-2}.}
\vspace{0.15cm}
\label{tab:gputimeHJB12}
\small
\begin{tabular*}{\textwidth}{@{\extracolsep{\fill}} l c c c c c c c c @{}}
\toprule
 & \multicolumn{4}{c}{DeepMartNet} & \multicolumn{4}{c}{iSMART (Method ($III$))} \\
\cmidrule(lr){2-5} \cmidrule(lr){6-9}
 & $t=1.0$ & $t=0.5$ & $t=0.0$ & GPU time & $t=1.0$ & $t=0.5$ & $t=0.0$ & GPU time\\
\midrule
\textbf{HJB-1}   &2.4142e-02  &2.0258e-01  &2.4090e-01  & 575.31\,s &0.0000e+00  &3.3191e-03  &3.6725e-03  & 326.04\,s  \\

\textbf{HJB-2a} &4.9018e-02  &6.3294e-02  &1.3098e-01  & 550.20\,s &0.0000e+00  &1.5017e-03  &1.6151e-03  & 388.95\,s  \\

\textbf{HJB-2b} &5.0897e-02  &5.6847e-02  &1.0810e-01  & 581.36\,s &0.0000e+00  &1.2008e-03  &1.7271e-03  & 389.68\,s  \\

\bottomrule
\end{tabular*}
\end{table}

Figure~\ref{fig:up-down6} shows the performance of iSMART for \textbf{HJB-1} and \textbf{HJB-2}. iSMART delivers excellent agreement with the reference solutions across all the tests. This experiment shows that iSMART guarantees a steady performance over a wide range of diffusion strength, from diffusion dominant regime to convection dominant regime. We also compare iSMART with the DeepMartNet method described in \cite{Cai2026a} for this problem. The statistics in \eqref{tab:gputimeHJB12} shows iSMART achieves superior accuracy and efficiency. iSMART maintains a steady relative $L^2$ errors within $5\times 10^{-3}$ at different time layers. With the higher accuracy, iSMART costs only about $60\%$ computing time than that of DeepMartNet. The loss history in Figure~\ref{fig:up-down6} also indicates that the freezing-and-compensating strategy effectively stabilises the iteration and benefits the performance of iSMART.

\begin{figure}[!h]
  \centering
  \begin{minipage}{0.95\textwidth}
    \centering
    \includegraphics[width=\textwidth]{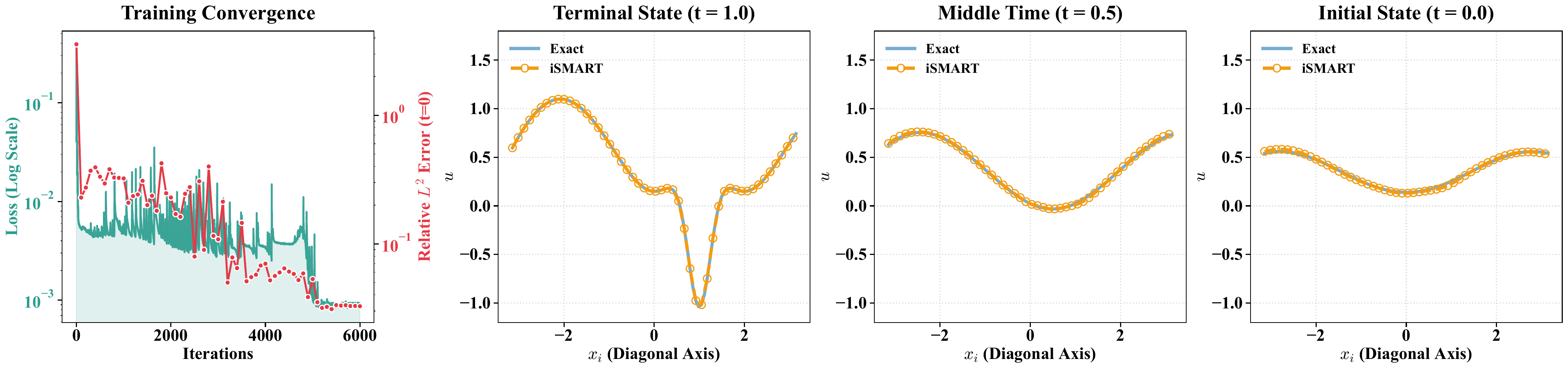}
    %\caption*{(b)}
  \end{minipage}
  \begin{minipage}{0.95\textwidth}
    \centering
    \includegraphics[width=\textwidth]{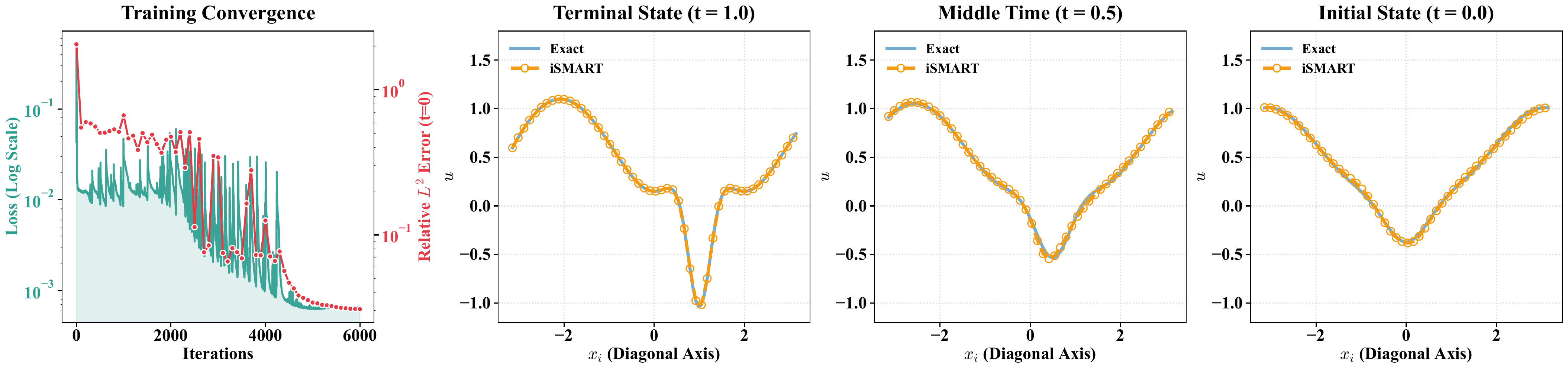}
    %\caption*{(a)}
  \end{minipage}
  \begin{minipage}{0.95\textwidth}
    \centering
    \includegraphics[width=\textwidth]{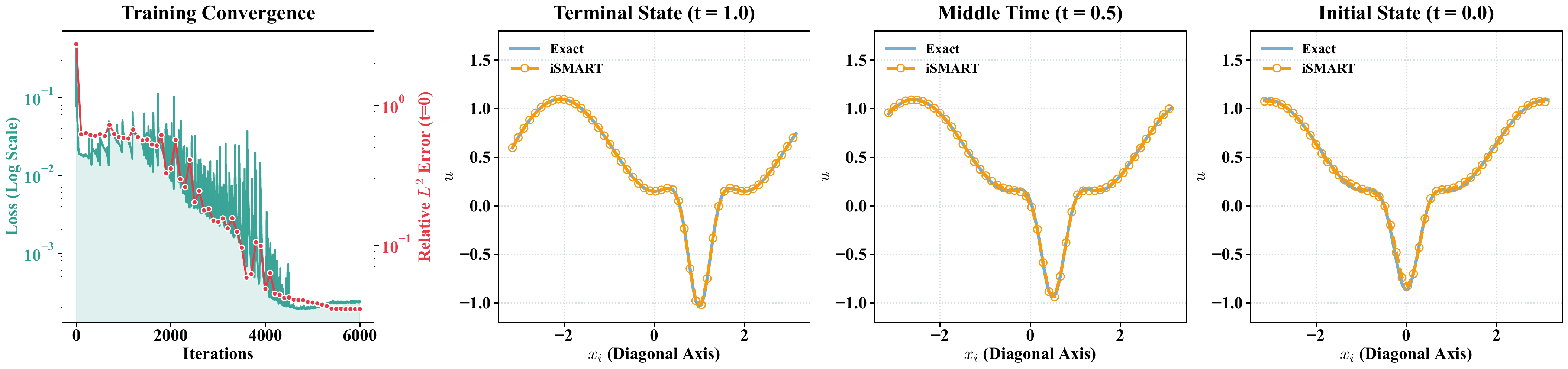}
    %\caption*{(a)}
  \end{minipage}
  \caption{Comparison between the iSMART numerical solutions and the exact solutions for the HJB equation \eqref{HJB-P} in the $1000$ dimensions. The first, second, and third rows correspond to \textbf{HJB-3}, \textbf{HJB-3-a}, and \textbf{HJB-3-b}, respectively, with paths generated by Method~($III$). In each row, the panels from left to right show the history of the training loss (in green line), the relative $L^2$ error, and the solution profiles at temporal snapshots $t=0, 0.5, 1$. The solution profiles $u(t, s\mathbf{1}_d), s\in [-3,3]$ are visualized.}
  \label{fig:up-down7}
\end{figure}

\begin{table}[!h]
\centering
\caption{Comparison of training GPU time costs and relative $L^2$ error in solving 1000D \textbf{HJB-3}.}
\vspace{0.15cm}
\label{tab:gputimeHJB3}
% 可选：缩小字号使表格更紧凑（若觉得列距太宽可开启）
\small
\begin{tabular*}{\textwidth}{@{\extracolsep{\fill}} l c c c c c c c c @{}}
\toprule
 & \multicolumn{4}{c}{DeepMartNet} & \multicolumn{4}{c}{iSMART (Method ($III$))} \\
\cmidrule(lr){2-5} \cmidrule(lr){6-9}
 & $t=1.0$ & $t=0.5$ & $t=0.0$ & GPU time & $t=1.0$ & $t=0.5$ & $t=0.0$ & GPU time\\
\midrule
\textbf{HJB-3}  &3.3255e-03  &4.6315e-01  &7.3706e-01 &4010.72\,s  &0.0000e+00   &2.0464e-02 &3.2436e-02  &1487.15\,s \\

\textbf{HJB-3a} &3.1578e-03  &1.6076e-01  &2.0934e-01 &4287.21\,s  &0.0000e+00   &3.5040e-02 &2.8771e-02  &1491.61\,s \\

\textbf{HJB-3b} &2.1585e-03  &2.5208e-01  &3.4905e-01 &4096.44\,s  &0.0000e+00   &1.7790e-02 &3.3120e-02  &1489.33\,s \\

\bottomrule
\end{tabular*}
\end{table}

For the most challenging case \textbf{HJB-3}, Figures~\ref{fig:up-down7} and Table ~\ref{tab:gputimeHJB3} illustrate that iSMART remains robust for different diffusion coefficients and gives a relative $L^2$ errors less than $5\%$ while the time requirement is only about $40\%$ of DeepMartNet. Overall, the consistently close match between the predicted and reference solutions across diverse suite of HJB equations highlights the accuracy, robustness, and wide applicability of the proposed iSMART method in nonlinear settings.

\section{Conclusion}
\label{Sec:Conclusion}

In this work, we present iSMART, a highly efficient iterative sampling-and-regression approach designed for solving high-dimensional PDEs admitting a martingale representation. The core novelty of our approach lies in reformulating the martingale representation as a sequence of least-squares regression problems, made possible by leveraging the $L^2$-projection property of conditional expectations. This perspective allows us to bypass nested Monte Carlo estimation entirely, replacing it with standard gradient-based optimization on discretized SDE trajectories.

iSMART applies uniformly to linear, semilinear, and fully nonlinear equations, and we provide a convergence analysis of the iterative scheme in a weighted Sobolev space. We discuss three distinct path-generation strategies within our framework. More significantly, we contribute a freezing-and-compensating technique specifically designed for fully nonlinear problems, which incorporates part of the gradient nonlinearity into the drift and substantially enhances the stability of the iteration. The framework also accommodates various spatio-temporal sampling strategies, further broadening its applicability.

The effectiveness of iSMART is demonstrated through extensive numerical experiments in up to $1000$ dimensions on a single NVIDIA RTX 4090 GPU. Across a variety of challenging regimes, our method maintains good accuracy and consistently outperforms DeepMartNet in terms of both relative error and computational cost. Importantly, the proposed framework is not confined to the specific class of equations considered here; it extends naturally to other problems whose solutions admit martingale representations, including partial integro-differential equations \cite{Andersson25,Georgoulis26},  elliptic problems, among others. The freezing-and-compensating technique developed in this work is likewise transferable to such settings. We believe that the simplicity, efficiency, and generality of iSMART make it a promising tool for high-dimensional PDEs arising in applications.

%+++++++++++++++++++++++++++++++++++++++++++++++++++++++++++++++++++++++++++++++++++++++++++++++++++
\section*{Code and Data Availability}
The code and data for the iSMART is available upon request and might be made publicly available upon publication.

%The source code for the iSMART is available on GitHub at \url{https://github.com/Fugui-Ma/iSMART}.

\section*{Acknowledges}
The authors acknowledge the support from National Key R\&D Program of China under grant 2021YFA1003301, the National Science Foundation of China under grant 12288101. F. Ma is partially supported by the Peking University Boya Postdoctoral Fellowship. We also thank the High-performance Computing Platform of Peking University for providing the computational resources for this research.

\section*{CRediT authorship contribution statement}
\textbf{Tiejun Li}: Writing – review \& editing, Methodology, Conceptualization; \textbf{Xiaoguang Li}: Writing – review \& editing, Methodology, Conceptualization. \textbf{Fugui Ma}: Writing-original draft, Software, Methodology.

\section*{Declaration of competing interest}
The authors declare that they have no known competing financial interests or personal relationships that could have appeared to influence the work reported in this paper.

%% Labels are used to cross-reference an item using \ref command.
%% Use \subsubsection, \paragraph, \subparagraph commands to
%% start 3rd, 4th and 5th level sections.
%% Refer following link for more details.
%% https://en.wikibooks.org/wiki/LaTeX/Document_Structure#Sectioning_commands


\appendix
\renewcommand{\thelemma}{\Alph{section}.\arabic{lemma}}   % 重新定义，此时 \Alph{section} 输出 A, B...
\setcounter{lemma}{0}
\section{Auxiliary Lemmas and Their Proofs}\label{Lems:ALTF}

The proof of the main theorem \ref{thm:converge} rests on four successive estimates. Lemma~\ref{lem:boundofgu} ensures uniform regularity of the iterates via parabolic Gaussian bounds: $|\nabla_x u_n|\le M$, $|\nabla_x^2 u_n|\le \frac{M}{\sqrt{T-t}}$, which provides the Lipschitz constants needed for subsequent stochastic estimates. Lemma~\ref{lem:moment} gives the moment growth $\mathbb{E}[\sup|\mathbf{X}^u|]\le C(1+|x|)$ for the process driven by $u$, enabling expectation bounds for nonlinear terms. Lemmas~\ref{lem:tail} and \ref{lem:head} together yield Lemma~\ref{lem:X}, the key probabilistic sensitivity estimate $\mathbb{E}\left[\big|\mathbf{X}_s^u-\mathbf{X}_s^v\big|\right]
\le \big(C/\sqrt{\beta}\big)\,\|u-v\|_\beta e^{\beta(T-t)}(1+|x|)$.
Finally, Lemma~\ref{lem:gradientE} derives the analytic error contraction $|\nabla_x e_{n+1}|_\beta \le \frac{C}{\sqrt{\beta}}\,\|e_n\|_\beta$. Combining this with a companion estimate for the full function error (which uses Lemma~\ref{lem:X}) gives $\|e_{n+1}\|_\beta\le\frac{C}{\sqrt{\beta}}\,\|e_n\|_\beta$. Choosing $\beta$ sufficiently large makes the contraction factor $C/\sqrt{\beta}<1$, ensuring geometric convergence and completing the main proof. The detailed proofs of the auxiliary lemmas are given below.

\begin{lemma}\label{lem:boundofgu}
Assuming the initial function $u_0(t,x)\in\mathbb{X}_{\beta}$ satisfies
\begin{displaymath}
\lvert\nabla_{x} u_0(t,x)\rvert\le A_0 e^{\beta_0(T-t)},\qquad
\lvert\nabla^2_{x} u_0(t,x)\rvert \le A_0 e^{\beta_0(T-t)}/\sqrt{T-t}
\end{displaymath}
for some $A_0$,  $\beta_0>0$. The terminal condition $u_n(T,x)=g(x)$ holds for all $n=0,1,2\cdots$. Then $u_n\in\mathbb{X}_\beta$ and there is a constant $M$ such that
\begin{displaymath}
  \big\lvert{\nabla_{x}} u_n(t,x)\big\rvert \le M,\qquad
  \big\lvert\nabla^2_{x} u_n(t,x)\big\rvert \le M/\sqrt{T-t},  \quad n\geq0,~t\in[0,T).
\end{displaymath}
\end{lemma}

\begin{proof}
For a given $u_{n-1}(t,x)$, define the coefficients $\mu_{n-1}(t,x)=\mu\bigl(t,x,u_{n-1}(t,x),\nabla_x u_{n-1}(t,x)\bigr)$, $\sigma_{n-1}(t,x)=\sigma\bigl(t,x,u_{n-1}(t,x)\bigr)$, and $f_{n-1}(t,x)=f\bigl(t,x,u_{n-1}(t,x),\nabla_x u_{n-1}(t,x)\bigr)$. Then $u_n(t,x)$ satisfies the linear parabolic equation
\begin{equation}\label{eq:linparaPDEn}
\partial_t u_n + \mu_{n-1}(t,x)\cdot\nabla_x u_n
+ \frac{1}{2}\operatorname{Tr}\Bigl(\sigma_{n-1}\sigma_{n-1}^{\top}(t,x)\nabla^2_x u_n\Bigr)
= f_{n-1}(t,x),
\end{equation}
with terminal condition $u_n(T,x)=g(x)$.

Let $\Gamma_n(t,x,s,y)$ be the fundamental solution of \eqref{eq:linparaPDEn}.  The solution admits the representation
\begin{equation}\label{eq:Green}
u_n(t,x) = \int_{\mathbb{R}^d}\Gamma_n(t,x,T,y)\,g(y)\,\mathrm{d}y
         + \int_{t}^{T}\mathrm{d}\tau \int_{\mathbb{R}^d}\Gamma_n(t,x,\tau,y)\,f_{n-1}(\tau,y)\,\mathrm{d}y=:I+II.
\end{equation}
For every $n\ge0$, $\Gamma_n$ satisfies
\begin{equation}\label{eq:uniform}
\int_{\mathbb{R}^d}\Gamma_n(t,x,s,y)\,\mathrm{d}y = 1
\end{equation}
and the Gaussian estimate \cite{Friedman64}
\begin{equation}\label{eq:GE0}
\big\lvert\Gamma_n(t,x,s,y)\big\rvert
\le \frac{C}{(s-t)^{d/2}}\exp\left\{-\frac{c\,\lvert x-y\rvert^{2}}{s-t}\right\},
\end{equation}
with positive constants $C$ and $c$ depending only on $\lambda$. Combining the Gaussian estimate \eqref{eq:GE0} with the linear growth of $g$ and $f_{n-1}$ yields
\begin{displaymath}
\begin{aligned}
|\,I\,|=\left\lvert\int_{\mathbb{R}^d}\Gamma_n(t,x,T,y)g(y)\,\mathrm{d}y\right\rvert
\le C\int_{\mathbb{R}^d}\frac{\mathrm{d}y}{(T-t)^{d/2}}
\exp\left\{-\frac{c\lvert x-y\rvert^{2}}{T-t}\right\}
\Bigl(1+\lvert x\rvert+\lvert x-y\rvert\Bigr)
\le CZ\Bigl(1+\lvert x\rvert\Bigr)+2CZ\sqrt{T}
\end{aligned}
\end{displaymath}
and
\begin{displaymath}
\begin{aligned}
|\,II\,|=\left\lvert\int_{t}^{T}\mathrm{d}\tau\int_{\mathbb{R}^d}\Gamma_n(t,x,\tau,y)f_{n-1}(\tau,y)\,\mathrm{d}y\right\rvert
&\le C\int_{t}^{T}\!\mathrm{d}\tau\int_{\mathbb{R}^d}\frac{\mathrm{d}y}{(\tau-t)^{d/2}}
\exp\left\{-\frac{c\lvert x-y\rvert^{2}}{\tau-t}\right\}
\Bigl(1+\lvert x\rvert+\lvert x-y\rvert\Bigr)\\
&\le CZT\Bigl(1+\lvert x\rvert\Bigr) + 2CZ\sqrt{T},
\end{aligned}
\end{displaymath}
with $Z:= \int_{\mathbb{R}^d}\lvert z\rvert\exp\{-c\lvert z\rvert^{2}\}\,\mathrm{d}z$ being a finite constant depending only on $\lambda$.  Consequently, $\lvert u_n(t,x)\rvert \le C(1+\lvert x\rvert)
\le C(1+\lvert x\rvert)e^{\beta(T-t)}$, which implies $\lvert u_n\rvert_{\beta}\le C$. Furthermore, the second-order differentiability of $u_n$ follows from classical parabolic theory \cite{Friedman64}. The gradient and Hessian of the fundamental solution admit the Gaussian estimates
\begin{displaymath}
\begin{aligned}
\big\lvert\nabla_{x}\Gamma_n(t,x,s,y)\big\rvert
\le\frac{C}{(s-t)^{(d+1)/2}}\exp\left\{-\frac{c\lvert x-y\rvert^{2}}{s-t}\right\},\quad
\big\lvert\nabla^{2}_{x}\Gamma_n(t,x,s,y)\big\rvert
\le\frac{C}{(s-t)^{(d+2)/2}}\exp\left\{-\frac{c\lvert x-y\rvert^{2}}{s-t}\right\}.
\end{aligned}
\end{displaymath}

Based on the above estimates, on one hand, by differentiating under the integral sign and leveraging  \eqref{eq:uniform}, we obtain
$\int_{\mathbb{R}^d}\nabla_{x}\Gamma_n(t,x,s,y)\,\mathrm{d}y=0$. By exploiting the Lipschitz continuity of $g$, we deduce
\begin{displaymath}
\begin{aligned}
\left\lvert\int_{\mathbb{R}^d}\nabla_{x}\Gamma_n(t,x,T,y)g(y)\,\mathrm{d}y\right\rvert
&= \left\lvert\int_{\mathbb{R}^d}\nabla_{x}\Gamma_n(t,x,T,y)\bigl(g(y)-g(x)\bigr)\,\mathrm{d}y\right\rvert
\le L\int_{\mathbb{R}^d}\big\lvert\nabla_{x}\Gamma_n(t,x,T,y)\big\rvert\,\big\lvert x-y\big\rvert\,\mathrm{d}y \\
&\le LC\int_{\mathbb{R}^d}\frac{\lvert x-y\rvert}{(T-t)^{(d+1)/2}}
\exp\left\{-\frac{c\lvert x-y\rvert^{2}}{T-t}\right\}\,\mathrm{d}y .
\end{aligned}
\end{displaymath}
Through a change of variables, we then arrive at
\begin{equation}\label{eq:homo}
\left\lvert\int_{\mathbb{R}^d}\nabla_{x}\Gamma_n(t,x,T,y)g(y)\,\mathrm{d}y\right\rvert \le LCZ.
\end{equation}
On the other hand, we obtain
\begin{equation}\label{eq:rhs}
\begin{aligned}
\Bigg\lvert&\int_{t}^{T}\!\mathrm{d}\tau\int_{\mathbb{R}^d}\nabla_{x}
  \Gamma_n(t,x,\tau,y)f_n(\tau,y)\,\mathrm{d}y\Bigg\rvert \\
&\le \int_{t}^{T}\!\mathrm{d}\tau\int_{\mathbb{R}^d}
\big\lvert\nabla_{x}\Gamma_n(t,x,\tau,y)\big\rvert\,
\bigl\lvert f\bigl(\tau,y,u_{n-1}(\tau,y),\nabla_x u_{n-1}(\tau,y)\bigr)
- f\bigl(\tau,x,u_{n-1}(\tau,x),\nabla_x u_{n-1}(\tau,x)\bigr)\bigr\rvert\,\mathrm{d}y \\
&\le LC\int_{t}^{T}\!\mathrm{d}\tau\int_{\mathbb{R}^d}
\frac{\mathrm{d}y}{(\tau-t)^{(d+1)/2}}
\left(\big\lvert x-y\big\rvert + \big\lvert u_{n-1}(\tau,x)-u_{n-1}(\tau,y)\big\rvert+ \big\lvert\nabla_x u_{n-1}(\tau,x)-\nabla_x u_{n-1}(\tau,y)\big\rvert\right)
\exp\left\{-\frac{c\lvert x-y\rvert^{2}}{\tau-t}\right\}.
\end{aligned}
\end{equation}
Applying the mean-value theorem, we have
\begin{displaymath}
\begin{aligned}
\big\lvert u_{n-1}(\tau,x)-u_{n-1}(\tau,y)\big\rvert
\le \sup_{x\in\mathbb{R}^d}\big\lvert\nabla_x u_{n-1}(\tau,x)\big\rvert\,\big\lvert x-y\big\rvert,\quad
\big\lvert\nabla_x u_{n-1}(\tau,x)-\nabla_x u_{n-1}(\tau,y)\big\rvert
\le \sup_{x\in\mathbb{R}^d}\big\lvert\nabla^{2}_x u_{n-1}(\tau,x)\big\rvert\,\big\lvert x-y\big\rvert.
\end{aligned}
\end{displaymath}
We define $G_n(t):=\sup_{x\in\mathbb{R}^d}\big\lvert\nabla_{x} u_n(t,x)\big\rvert$ and $H_n(t):= \sup_{x\in\mathbb{R}^d}\big\lvert\nabla^{2}_{x} u_n(t,x)\big\rvert$. Combining the representation \eqref{eq:Green} with the estimates \eqref{eq:homo} and \eqref{eq:rhs}, we then obtain
\begin{equation}\label{eq:Gn}
G_n(t) \le C_1 + C_1\int_{t}^{T}\bigl(G_{n-1}(\tau) + H_{n-1}(\tau)\bigr)\,\mathrm{d}\tau,
\end{equation}
where $C_1$ is a generic constant independent of $n$.

A similar argument yields the estimate for $\nabla^{2}_{x}u_n$.  For the homogeneous part $I$, we deduce
\begin{equation}\label{eq:homo2}
\begin{aligned}
\left\lvert\int_{\mathbb{R}^d}\nabla^{2}_{x}\Gamma_n(t,x,T,y)g(y)\,\mathrm{d}y\right\rvert
&= \left\lvert\int_{\mathbb{R}^d}\nabla^{2}_{x}\Gamma_n(t,x,T,y)\bigl(g(y)-g(x)\bigr)\,\mathrm{d}y\right\rvert
\le L\int_{\mathbb{R}^d}\big\lvert\nabla^{2}_{x}\Gamma_n(t,x,T,y)\big\rvert\,\lvert x-y\rvert\,\mathrm{d}y \\
&\le LC\int_{\mathbb{R}^d}\frac{\lvert x-y\rvert}{(T-t)^{(d+2)/2}}
\exp\left\{-\frac{c\lvert x-y\rvert^{2}}{T-t}\right\}\,\mathrm{d}y
\le \frac{LCZ}{\sqrt{T-t}} .
\end{aligned}
\end{equation}
For the inhomogeneous part $II$, we also get
\begin{equation}\label{eq:inhomo2}
\begin{aligned}
\Bigl\lvert&\int_{t}^{T}\!\mathrm{d}\tau\int_{\mathbb{R}^d}\nabla^{2}_{x}\Gamma_n(t,x,\tau,y)f_n(\tau,y)\,\mathrm{d}y\Bigr\rvert \\
&\le \int_{t}^{T}\mathrm{d}\tau\int_{\mathbb{R}^d}
\big\lvert\nabla^{2}_{x}\Gamma_n(t,x,\tau,y)\big\rvert\,
\bigl\lvert f\bigl(\tau,y,u_{n-1}(\tau,y),\nabla_x u_{n-1}(\tau,y)\bigr)
- f\bigl(\tau,x,u_{n-1}(\tau,x),\nabla_x u_{n-1}(\tau,x)\bigr)\bigr\rvert\,\mathrm{d}y \\
&\le LC\int_{t}^{T}\mathrm{d}\tau\int_{\mathbb{R}^d}
\frac{\mathrm{d}y}{(\tau-t)^{(d+2)/2}}
\Bigl(\big\lvert x-y\big\rvert+\big\lvert u_{n-1}(\tau,x)-u_{n-1}(\tau,y)\big\rvert +\big\lvert\nabla_x u_{n-1}(\tau,x)-\nabla_x u_{n-1}(\tau,y)\big\rvert\Bigr)
\exp\left\{-\frac{c\lvert x-y\rvert^{2}}{\tau-t}\right\} \\
&\le 2LCZ\sqrt{T} + LCZ\int_{t}^{T}\frac{1}{\sqrt{\tau-t}}
\bigl(G_{n-1}(\tau) + H_{n-1}(\tau)\bigr)\,\mathrm{d}\tau .
\end{aligned}
\end{equation}
Consequently, there exists a constant $C_2>0$ such that
\begin{equation}\label{eq:Hn}
H_n(t) \le \frac{C_2}{\sqrt{T-t}} + C_2\int_{t}^{T}\frac{1}{\sqrt{\tau-t}}
\Bigl(G_{n-1}(\tau) + H_{n-1}(\tau)\Bigr)\,\mathrm{d}\tau .
\end{equation}

Now assume that for some $A$, $\beta>0$, $G_{n-1}(t)+H_{n-1}(t)\le\frac{A e^{\beta(T-t)}}{\sqrt{T-t}}$. We shall prove that $G_n+H_n$ satisfies the same bound. Adding \eqref{eq:Gn} and \eqref{eq:Hn} and inserting the induction hypothesis, we obtain
\begin{displaymath}
G_n(t)+H_n(t) \le \frac{C}{\sqrt{T-t}} + C A e^{\beta(T-t)}
\int_{t}^{T}\frac{e^{-\beta(\tau-t)}\,\mathrm{d}\tau}{\sqrt{(\tau-t)(T-\tau)}} .
\end{displaymath}
By setting $\frac{\tau-t}{T-t}=\sin^{2}\theta$, a direct calculation yields
\begin{displaymath}
\begin{aligned}
\int_{t}^{T}\frac{e^{-\beta(\tau-t)}\,\mathrm{d}\tau}{\sqrt{(\tau-t)(T-\tau)}}
= \int_{0}^{\frac{\pi}{2}} 2 e^{-\beta(T-t)\sin^{2}\theta}\,\mathrm{d}\theta
\le \int_{0}^{\frac{\pi}{2}} 2 e^{-4\beta(T-t)\theta^{2}/\pi^{2}}\,\mathrm{d}\theta
\le \frac{\pi}{\sqrt{\beta(T-t)}}\int_{0}^{+\infty}e^{-\phi^{2}}\,\mathrm{d}\phi
\le \frac{C}{\sqrt{\beta(T-t)}},
\end{aligned}
\end{displaymath}
where the elementary inequality $\sin\theta \ge \frac{2}{\pi}\theta$ for $\theta\in[0,\pi/2]$ has been utilized, and the last inequality follows from the substitution $\phi = \frac{2\sqrt{\beta(T-t)}}{\pi}\theta$ together with the obvious bound $\int_{0}^{\sqrt{\beta(T-t)}}e^{-\phi^{2}}\,\mathrm{d}\phi \le \int_{0}^{\infty}e^{-\phi^{2}}\,\mathrm{d}\phi$. Thus, we have
\begin{equation}\label{eq:G+H}
G_n(t)+H_n(t) \le \frac{A e^{\beta(T-t)}}{\sqrt{T-t}}
\left(\frac{C}{A} + \frac{C}{\sqrt{\beta}}\right).
\end{equation}
We may now select $A>\max\{2C,A_0\}$ and $\beta >\max\{4C^{2},\beta_0\}$ such that
$\frac{C}{A}+\frac{C}{\sqrt{\beta}}<1$, which implies $G_n(t) + H_n(t) \le \frac{A e^{\beta(T-t)}}{\sqrt{T-t}}$.

Since the estimate holds for $n=0$ by assumption, induction ensures the bound is valid for all $n\ge0$. Finally, by inserting the uniform bound on $G_n+H_n$ provided in \eqref{eq:G+H} back into \eqref{eq:Gn} and \eqref{eq:Hn}, we obtain the desired pointwise estimates
\begin{displaymath}
\big\lvert\nabla_{x} u_n(t,x)\big\rvert
\le G_n(t) \le C_1 + C_1 A \int_{t}^{T}\frac{e^{\beta(T-\tau)}}{\sqrt{T-\tau}}\,\mathrm{d}\tau
\le C_1 + 2C_1\sqrt{T}A e^{\beta T}
\end{displaymath}
and
\begin{displaymath}
\big\lvert\nabla^{2}_{x} u_n(t,x)\big\rvert
\le H_n(t) \le \frac{C_2}{\sqrt{T-t}} + C_2 A \int_{t}^{T}\frac{e^{\beta(T-\tau)}}{\sqrt{(\tau-t)(T-\tau)}}\,\mathrm{d}\tau
\le \frac{C_2}{\sqrt{T-t}} + \pi C_2 A e^{\beta T} .
\end{displaymath}
The fact that $u_n\in\mathbb{X}_{\beta}$ follows from the bound on $\lvert u_n\rvert_{\beta}$ and the uniform boundedness of $\lvert\nabla u_n\rvert$.
\end{proof}

For any $u(\mathbf{x})\in\mathbb{X}_\beta$, let us define a process $\mathbf{X}_t^u$ by
\begin{equation}\label{eq:dXtu}
  \mathrm{d}\mathbf{X}_t^u = \mu\big(t,\mathbf{X}_t^u, u(t,\mathbf{X}_t^u), \nabla_x u(t,\mathbf{X}_t^u)\big)\mathrm{d}t + \sigma\big(t,\mathbf{X}_t^u,u(t,\mathbf{X}_t^u)\big)\mathrm{d}\mathbf{B}_t.
\end{equation}
Assumption \textbf{A1}, \textbf{A2} and the global Lipschitz condition of $u$ ensures the existence and uniqueness of a strong solution to \eqref{eq:dXtu} for any initial value $\textbf{X}_0=x$. Moreover, we can establish a moment estimation of $\textbf{X}^u_t$.

\begin{lemma}\label{lem:moment}
For any $u\in\mathbb{X}_\beta$ with $|\nabla_{x}u|< M$, there exists a constant $C>0$ such that for any $(t,x)\in\mathcal{D}$ and $s\in[t,T]$,
\begin{equation}\label{eq:moment}
\mathbb{E}^{t,x}\left[\big|\mathbf{X}_s^u\big|\right] \leq \mathbb{E}^{t,x}\Big[\sup_{s\in [t,T]}\big|\mathbf{X}_s^u\big|\Big]
\leq C\Big(1+|x|\Big).
\end{equation}
Consequently, for any $b(t,x)\in\mathbb{X}_\beta$, $\mathbb{E}^{t,x}\big[|b(s,\mathbf{X}_s^u)|\big]\leq C|b|_\beta\, e^{\beta(T-s)}(1+|x|)$.
\end{lemma}

\begin{proof}
For $s\in[t,T]$, it follows from the definition of $\mathbf{X}_s^u$ in \eqref{eq:dXtu} that
\begin{displaymath}
    \mathbf{X}_s^u = x
    + \int_{t}^{s}\mu\bigl(\tau,\mathbf{X}_\tau^u, u(\tau,\mathbf{X}_\tau^u), \nabla_{x} u(\tau,\mathbf{X}_\tau^u)\bigr)\mathrm{d}\tau
    + \int_{t}^{s}\sigma\bigl(\tau,\mathbf{X}_\tau^u,u(\tau,\mathbf{X}_\tau^u)\bigr)\mathrm{d}\mathbf{B}_\tau.
\end{displaymath}
Since $\mu$ is bounded (Assumption \textbf{A.1}), we have $\left|\int_{t}^{s}\mu\bigl(\tau,\mathbf{X}_\tau^u, u(\tau,\mathbf{X}_\tau^u), \nabla_{x} u(\tau,\mathbf{X}_\tau^u)\bigr)\mathrm{d}\tau\right|\le K$ for some constant $K$. Hence
\begin{displaymath}
\big|\mathbf{X}_s^u\big|\le |x|+K+ \left|\int_{t}^{s}\sigma\bigl(\tau,\mathbf{X}_\tau^u,u(\tau,\mathbf{X}_\tau^u)\bigr)\mathrm{d}\mathbf{B}_\tau\right|.
\end{displaymath}
Taking squares and expectations, using $(a+b)^2\le 2a^2+2b^2$ and It\^o's isometry, we obtain
\begin{equation}\label{eq:xmean}
\begin{aligned}
\mathbb{E}^{t,x}\left[\big|\mathbf{X}_s^u\big|^2\right]
\le C\big(1+|x|\big)^2 + 2\mathbb{E}^{t,x}\left[\left|\int_{t}^{s}\sigma\bigl(\tau,\mathbf{X}_\tau^u,u(\tau,\mathbf{X}_\tau^u)\bigr)\mathrm{d}\mathbf{B}_\tau\right|^2\right]
&= C\big(1+|x|\big)^2 + 2\int_{t}^{s}\mathbb{E}^{t,x}\left[\Big|\sigma\bigl(\tau,\mathbf{X}_\tau^u,u(\tau,\mathbf{X}_\tau^u)\bigr)\Big|^2\mathrm{d}\tau\right].
\end{aligned}
\end{equation}
By the Lipschitz continuity of $\sigma$ in the last two arguments (Assumption \textbf{A.2}) and the bound $|\nabla u|<M$ from Lemma \ref{lem:boundofgu},
\begin{displaymath}
\big|\sigma\big(\tau,\mathbf{x},u(\tau,\mathbf{x})\big)\big|
\le \big|\sigma(\tau,0,u(\tau,0)\big)\big| + L\Bigl(|x| + \big|u(\tau,\mathbf{x})-u(\tau,0)\big|\Bigr)
\le C\big(1+|x|\big).
\end{displaymath}
Substituting this linear growth estimate into \eqref{eq:xmean} yields
\begin{displaymath}
\mathbb{E}^{t,x}\left[|\mathbf{X}_s^u|^2\right] \le C\big(1+|x|\big)^2 + C\int_{t}^{s}\mathbb{E}^{t,x}\left[|\mathbf{X}_\tau^u|^2\right]\mathrm{d}\tau.
\end{displaymath}
By Gr\"onwall's inequality, $\mathbb{E}^{t,x}\big[|\mathbf{X}_s^u|^2\big]\le C(1+|x|)^2$, which implies
$\mathbb{E}^{t,x}\left[|\mathbf{X}_s^u|\right]\le C(1+|x|)$. The estimate for the supremum follows by a standard application of Doob's maximal inequality (see \cite{Karatzas91}) (or the Burkholder-Davis-Gundy inequality, see, e.g., \cite{Burkholder70}); we omit the routine details.

For any $b\in\mathbb{X}_\beta$, the norm $|b|_\beta$ satisfies $|b(s,\mathbf{x})|\le |b|_\beta(1+|\mathbf{x}|)e^{\beta(T-s)}$. Taking expectations and using the previous moment bound,
\begin{displaymath}
\mathbb{E}^{t,x}\left[\big|b(s,\mathbf{X}_s^u)\big|\right]
\le \big|b\big|_\beta e^{\beta(T-s)}\mathbb{E}^{t,x}\Bigl[1+\big|\mathbf{X}_s^u\big|\Bigr]
\le C\big|b\big|_\beta\left(1+|x|\right)e^{\beta(T-s)}.
\end{displaymath}
\end{proof}

For any $u, v\in\mathbb{X}_\beta$ and $0\leq t<s\leq T$, we define
\begin{displaymath}
\begin{aligned}
\mathbf{X}_s^u &:= x + \int_{t}^{s}\mu\bigl(\tau,\mathbf{X}_\tau^u, u(\tau,\mathbf{X}_\tau^u), \nabla_{x} u(\tau,\mathbf{X}_\tau^u)\bigr)\mathrm{d}\tau
                + \int_{t}^{s}\sigma\bigl(\tau,\mathbf{X}_\tau^u,u(\tau,\mathbf{X}_\tau^u)\bigr)\mathrm{d}\mathbf{B}_\tau,\\
\mathbf{X}_s^v &:= x + \int_{t}^{s}\mu\bigl(\tau,\mathbf{X}_\tau^v, v(\tau,\mathbf{X}_\tau^v), \nabla_{x} v(\tau,\mathbf{X}_\tau^v)\bigr)\mathrm{d}\tau
                + \int_{t}^{s}\sigma\bigl(\tau,\mathbf{X}_\tau^v,v(\tau,\mathbf{X}_\tau^v)\bigr)\mathrm{d}\mathbf{B}_\tau.
\end{aligned}
\end{displaymath}
Taking the difference, we obtain
\begin{equation}\label{eq:differencex}
\begin{aligned}
\big|\mathbf{X}_s^u - \mathbf{X}_s^v\big|
&\leq \int_{t}^{s}\bigl|\mu\big(\tau,\mathbf{X}_\tau^u, u(\tau,\mathbf{X}_\tau^u), \nabla_{x} u(\tau,\mathbf{X}_\tau^u)\big)
        - \mu\big(\tau,\mathbf{X}_\tau^v, v(\tau,\mathbf{X}_\tau^v), \nabla_{x} v(\tau,\mathbf{X}_\tau^v)\big)\bigr|\mathrm{d}\tau\\
&\quad + \biggl|\int_{t}^{s}\Bigl[\sigma(\tau,\mathbf{X}_\tau^u,u(\tau,\mathbf{X}_\tau^u))
        - \sigma(\tau,\mathbf{X}_\tau^v,v(\tau,\mathbf{X}_\tau^v))\Bigr]\mathrm{d}\mathbf{B}_\tau\biggr|
\triangleq H_1(s,\omega)+H_2(s,\omega).
\end{aligned}
\end{equation}
To estimate this difference we introduce the following lemma.

\begin{lemma}\label{lem:tail}
  Let $u,v\in\mathbb{X}_\beta$ satisfy $|\nabla_{x} u(t,\mathbf{x})|\leq M$ and $|\nabla^2_{x} u(t,\mathbf{x})|\leq \frac{M}{\sqrt{T-t}}$.
  Then there exist constants $C>0$ and $\delta>0$ independent of $u$ and $v$ such that for all $T-\delta<t<s\leq T$,
\begin{displaymath}
  \mathbb{E}^{t,x}\left[\big|\mathbf{X}_s^u - \mathbf{X}_s^v\big|\right]
  \leq \frac{C}{\sqrt{\beta}}\,\big\|u-v\big\|_{\beta}\, e^{\beta(T-t)}\big(1+|x|\big).
\end{displaymath}
\end{lemma}

\begin{proof}
From the definition of $H_1$ and $H_2$ in \eqref{eq:differencex}, we obtain $|\mathbf{X}_s^u-\mathbf{X}_s^v|\le H_1(s,\omega)+H_2(s,\omega)$. Squaring and taking expectations, then applying the Cauchy--Schwarz inequality gives
\begin{displaymath}
\mathbb{E}^{t,x}\left[\big|\mathbf{X}_s^u - \mathbf{X}_s^v\big|\right]
\leq \left(\mathbb{E}^{t,x}\left[|\mathbf{X}_s^u - \mathbf{X}_s^v|^2\right]\right)^{1/2}
\leq \left(2\mathbb{E}^{t,x}\left[H_1^2\right]+ 2\mathbb{E}^{t,\mathbf{x}}\left[H_2^2\right]\right)^{1/2}.
\end{displaymath}
We estimate the $H_1$ and $H_2$ terms separately.

To estimate $H_1$, we leverage the Lipschitz property of $\mu$ and the bounds established in Lemma~\ref{lem:boundofgu}. This yields the following inequality for $H_1$,
\begin{equation}\label{eq:H1}
\begin{aligned}
H_1 &\le L\int_{t}^{s}\Bigl(\big|\mathbf{X}_\tau^u - \mathbf{X}_\tau^v\big|
      + \big|u(\tau,\mathbf{X}_\tau^u)-u(\tau,\mathbf{X}_\tau^v)\big|
      + \big|u(\tau,\mathbf{X}_\tau^v)-v(\tau,\mathbf{X}_\tau^v)\big| \\
     &\qquad\qquad + \big|\nabla_{x} u(\tau,\mathbf{X}_\tau^u)-\nabla_{x} u(\tau,\mathbf{X}_\tau^v)\big|
      + \big|\nabla_{x} u(\tau,\mathbf{X}_\tau^v)-\nabla_{x} v(\tau,\mathbf{X}_\tau^v)\big|\Bigr)\mathrm{d}\tau \\
&\le L\int_{t}^{T}\left(1+M+\frac{M}{\sqrt{T-\tau}}\right)
       \big|\mathbf{X}_\tau^u-\mathbf{X}_\tau^v\big|\,\mathrm{d}\tau
     + L\int_{t}^{T}\Big(\big|u(\tau,\mathbf{X}_\tau^v)-v(\tau,\mathbf{X}_\tau^v)\big|
      + \big|\nabla_{x} u(\tau,\mathbf{X}_\tau^v)-\nabla_{x} v(\tau,\mathbf{X}_\tau^v)\big|\Big)\mathrm{d}\tau \\
&\le C\sqrt{T-t}\sup_{\tau\in[t,T]}\big|\mathbf{X}_\tau^u-\mathbf{X}_\tau^v\big|
     + L\int_{t}^{T}\Bigl(\big|u(\tau,\mathbf{X}_\tau^v)-v(\tau,\mathbf{X}_\tau^v)\big|
      + \big|\nabla_{x}u(\tau,\mathbf{X}_\tau^v)-\nabla_{x}v(\tau,\mathbf{X}_\tau^v)|\Bigr)\mathrm{d}\tau.
\end{aligned}
\end{equation}
Using Lemma~\ref{lem:moment} we can bound the expectation of $H_1^2$ (and also its supremum) by
\begin{equation}\label{eq:H12}
\begin{aligned}
\mathbb{E}^{t,x}\left[\sup_{s\in[t,T]} H_1^2\right]
&\le C(T-t)\,\mathbb{E}^{t,x}\left[\sup_{\tau\in[t,T]}|\mathbf{X}_\tau^u - \mathbf{X}_\tau^v|^2\right]
+ C\int_{t}^{T}\mathbb{E}^{t,x}\left[\big|u(\tau,\mathbf{X}_\tau^v)-v(\tau,\mathbf{X}_\tau^v)\right|^2
      +\big|\nabla_x u(\tau,\mathbf{X}_\tau^v)-\nabla_x v(\tau,\mathbf{X}_\tau^v)\big|^2\Bigr]\mathrm{d}\tau \\
&\le C(T-t)\,\mathbb{E}^{t,x}\left[\sup_{\tau\in[t,T]}\big|\mathbf{X}_\tau^u-\mathbf{X}_\tau^v\big|^2\right]
   + C\big\|u-v\big\|_{\beta}^2\int_{t}^{T}e^{2\beta(T-\tau)}\big(1+|x|\big)^2\mathrm{d}\tau \\
&\le C(T-t)\,\mathbb{E}^{t,x}\left[\sup_{\tau\in[t,T]}|\mathbf{X}_\tau^u-\mathbf{X}_\tau^v|^2\right]
   + \frac{C}{\beta}\,\big\|u-v\big\|_{\beta}^2\,e^{2\beta(T-t)}\big(1+|x|\big)^2 .
\end{aligned}
\end{equation}

To estimate $H_2$, we first apply the Burkholder-Davis-Gundy inequality \cite[Theorem~3.3.28]{Karatzas91} to convert the supremum of the stochastic integral into an $L^2$-estimate of the integrand.  Using the Lipschitz continuity of $\sigma$ together with $|\nabla u|<M$ and expanding the square yields
\begin{equation}\label{eq:H2}
\begin{aligned}
\mathbb{E}^{t,\mathbf{x}}\left[\sup_{s\in[t,T]} H_2^2\right]
&\le C\int_{t}^{T}\mathbb{E}^{t,x}\left[\bigl|\sigma\big(\tau,\mathbf{X}_\tau^u, u(\tau,\mathbf{X}_\tau^u)\big)
   - \sigma\big(\tau,\mathbf{X}_\tau^v, v(\tau,\mathbf{X}_\tau^v)\big)\bigr|^2\right]\mathrm{d}\tau \\[4pt]
&\le C\int_{t}^{T}\mathbb{E}^{t,x}\Bigl[(1+M)\big|\mathbf{X}_\tau^u-\mathbf{X}_\tau^v\big|
   + \big|u(\tau,\mathbf{X}_\tau^u)-v(\tau,\mathbf{X}_\tau^v)\big|\Bigr]^2\mathrm{d}\tau \\
&\le C\int_{t}^{T}\mathbb{E}^{t,x}\left[|\mathbf{X}_\tau^u-\mathbf{X}_\tau^v|^2\right]\mathrm{d}\tau
   + C\|u-v\|_{\beta}^2\int_{t}^{T}e^{2\beta(T-\tau)}\big(1+|x|\big)^2\mathrm{d}\tau \\
&\le C(T-t)\,\mathbb{E}^{t,x}\left[\sup_{\tau\in[t,T]}|\mathbf{X}_\tau^u-\mathbf{X}_\tau^v|^2\right]
   + \frac{C}{\beta}\,\|u-v\|_{\beta}^2\,e^{2\beta(T-t)}\big(1+|x|\big)^2 .
\end{aligned}
\end{equation}
The last two lines follow from separating the mixed terms by the inequality $(a+b)^2\le 2a^2+2b^2$, then estimating $|u(\tau,\mathbf{X}_\tau^v)-v(\tau,\mathbf{X}_\tau^v)|$ via the norm $\|\cdot\|_\beta$ and using the moment bound in Lemma~\ref{lem:moment}.

By combining \eqref{eq:H12} and \eqref{eq:H2}, we derive the following result
\begin{equation}\label{eq:H1H2}
\begin{aligned}
\mathbb{E}^{t,x}\left[|\mathbf{X}_s^u-\mathbf{X}_s^v|^2\right]
&\le \mathbb{E}^{t,x}\left[\sup_{\tau\in[t,T]}|\mathbf{X}_\tau^u-\mathbf{X}_\tau^v|^2\right]
\le 2\,\mathbb{E}^{t,x}\left[\sup_{s}H_1^2 + 2\,\mathbb{E}^{t,\mathbf{x}}\sup_{s}H_2^2\right] \\
&\le C(T-t)\,\mathbb{E}^{t,x}\left[\sup_{\tau\in[t,T]}\big|\mathbf{X}_\tau^u-\mathbf{X}_\tau^v\big|^2\right]
   + \frac{C}{\beta}\,\big\|u-v\big\|_{\beta}^2\,e^{2\beta(T-t)}\big(1+|x|\big)^2.
\end{aligned}
\end{equation}
Next, let us select a small positive $\delta>0$ such that $T-t<\delta$, we have $C(T-t)<\frac{1}{2}$.  For such values of $t$, this leads to
\begin{displaymath}
\frac{1}{2}\,\mathbb{E}^{t,x}\left[\sup_{\tau\in[t,T]}|\mathbf{X}_\tau^u-\mathbf{X}_\tau^v|^2\right]
\le \frac{C}{\beta}\,\|u-v\|_{\beta}^2 e^{2\beta(T-t)}\big(1+|x|\big)^2 .
\end{displaymath}
As a consequence, we obtain
\begin{displaymath}
\left(\mathbb{E}^{t,x}\left[\left|\mathbf{X}_s^u-\mathbf{X}_s^v\right|\right]\right)^2
\le \mathbb{E}^{t,x}\left[\big|\mathbf{X}_s^u-\mathbf{X}_s^v\big|^2\right]
\le \frac{2C}{\beta}\,\big\|u-v\big\|_{\beta}^2\,e^{2\beta(T-t)}\big(1+|x|\big)^2 .
\end{displaymath}
Taking square roots gives exactly the claimed estimate.
\end{proof}

\begin{lemma}\label{lem:head}
Let $u$, $v\in\mathbb{X}_\beta$ satisfy $|\nabla_{x}u|<M$ and $|\nabla^2_{x}u|<\frac{M}{\sqrt{T-t}}$. For any $\delta>0$ there exists a constant $C=C(\delta)$ such that for all $0<t<s<T-\delta$,
\begin{displaymath}
\mathbb{E}^{t,x}\left[\big|\mathbf{X}_{s}^u-\mathbf{X}_{s}^v\big|\right]
\leq \frac{C}{\sqrt{\beta}}\,\big\|u-v\big\|_\beta\, e^{\beta(T-t)}\big(1+|x|\big).
\end{displaymath}
\end{lemma}

\begin{proof}
Fix $t\in[0,T-\delta]$ and $s\in(t,T-\delta]$. Recall that $|\mathbf{X}_s^u-\mathbf{X}_s^v|\le H_1+H_2$, where $H_1$ and $H_2$ are defined in \eqref{eq:differencex}. We estimate their second moments separately.

To estimate $H_1$, we leverage the Lipschitz property of $\mu$ along with the bounds $|\nabla_{x}u|<M$ and $|\nabla^2_{x}u|\le M/\sqrt{T-\tau}$, which allows us to derive
\begin{displaymath}
H_1 \le L\int_{t}^{s}\left(1+M+\frac{M}{\sqrt{T-\tau}}\right)
\big|\mathbf{X}_\tau^u-\mathbf{X}_\tau^v\big|\,\mathrm{d}\tau
+ L\int_{t}^{s}\left(\big|u(\tau,\mathbf{X}_\tau^v)-v(\tau,\mathbf{X}_\tau^v)\big|
+\big|\nabla_{x}u(\tau,\mathbf{X}_\tau^v)-\nabla_{x}v(\tau,\mathbf{X}_\tau^v)\big|\right)\mathrm{d}\tau.
\end{displaymath}
Since $T-\tau\ge\delta$ on the integration interval, we have $(1+M+\frac{M}{\sqrt{T-\tau}})^2\le C/\delta$ (with a constant $C$ depending only on $M$). Applying the Cauchy-Schwarz inequality to the time integrals and taking expectations,
\begin{equation}\label{eq:H1H}
\begin{aligned}
\mathbb{E}^{t,x}\left[H_1^2\right]
&\le 2L^2\,\mathbb{E}^{t,x}\Bigg[\left(\int_{t}^{s}\left(1+M+\frac{M}{\sqrt{T-\tau}}\right)
\big|\mathbf{X}_\tau^u-\mathbf{X}_\tau^v\big|\,\mathrm{d}\tau\right)^{2}\Bigg] \\
&\quad+2L^2\,\mathbb{E}^{t,x}\Bigg[\left(\int_{t}^{s}\left(\big|u(\tau,\mathbf{X}_\tau^v)-v(\tau,\mathbf{X}_\tau^v)\big|
+\big|\nabla_{x} u(\tau,\mathbf{X}_\tau^v)-\nabla_{x} v(\tau,\mathbf{X}_\tau^v)\big|\right)\mathrm{d}\tau\right)^{2}\Bigg] \\
&\le\frac{CT}{\delta}\int_{t}^{s}\mathbb{E}^{t,x}\left[\big|\mathbf{X}_\tau^u-\mathbf{X}_\tau^v\big|^{2}\right]\,\mathrm{d}\tau
+ C\big\|u-v\big\|_{\beta}^{2}\int_{t}^{s}e^{2\beta(T-\tau)}\big(1+|x|\big)^{2}\,\mathrm{d}\tau \\
&\le C\int_{t}^{s}\mathbb{E}^{t,x}\left[\big|\mathbf{X}_\tau^u-\mathbf{X}_\tau^v\big|^{2}\right]\,\mathrm{d}\tau
+ \frac{C}{\beta}\,\big\|u-v\big\|_{\beta}^{2}\,e^{2\beta(T-t)}\big(1+|x|\big)^{2}.
\end{aligned}
\end{equation}
In the second inequality we used $\bigl(\int_t^s f\,\mathrm{d}\tau\bigr)^2\le T\int_t^s f^2\,\mathrm{d}\tau$ together with the bound on the coefficient and the moment estimate of Lemma~\ref{lem:moment} for the terms involving $u$ and $v$.

To estimate $H_2$, we leverage It\^{o}'s isometry, the Lipschitz continuity of $\sigma$, and the decomposition $|u(\tau,\mathbf{X}_\tau^u)-v(\tau,\mathbf{X}_\tau^v)|\le M|\mathbf{X}_\tau^u-\mathbf{X}_\tau^v|+|u(\tau,\mathbf{X}_\tau^v)-v(\tau,\mathbf{X}_\tau^v)|$, thereby deriving
\begin{equation}\label{eq:H2H}
\begin{aligned}
\mathbb{E}^{t,\mathbf{x}}\left[H_2^2\right]
&= \int_{t}^{s}\mathbb{E}^{t,x}\left[\bigl|
     \sigma\big(\tau,\mathbf{X}_\tau^u,u(\tau,\mathbf{X}_\tau^u)\big)
     -\sigma\big(\tau,\mathbf{X}_\tau^v,v(\tau,\mathbf{X}_\tau^v)\big)\bigr|^{2}\right]\,\mathrm{d}\tau \\
&\le L^{2}\int_{t}^{s}\mathbb{E}^{t,x}\left[\left(
     (1+M)\big|\mathbf{X}_\tau^u-\mathbf{X}_\tau^v\big|
     +\big|u(\tau,\mathbf{X}_\tau^v)-v(\tau,\mathbf{X}_\tau^v)\big|\right)^{2}\right]\,\mathrm{d}\tau \\
&\le C\int_{t}^{s}\mathbb{E}^{t,\mathbf{x}}
      \left[\big|\mathbf{X}_\tau^u-\mathbf{X}_\tau^v\big|^{2}\right]\,\mathrm{d}\tau
     + C\big\|u-v\big\|_{\beta}^{2}\int_{t}^{s}e^{2\beta(T-\tau)}\big(1+|x|\big)^{2}\,\mathrm{d}\tau \\
&\le C\int_{t}^{s}\mathbb{E}^{t,\mathbf{x}}
      \left[\big|\mathbf{X}_\tau^u-\mathbf{X}_\tau^v|^{2}\right]\,\mathrm{d}\tau
     + \frac{C}{\beta}\,\big\|u-v\big\|_{\beta}^{2}\,e^{2\beta(T-t)}\big(1+|x|\big)^{2}.
\end{aligned}
\end{equation}

By adding \eqref{eq:H1H} and \eqref{eq:H2H} and applying $|\mathbf{X}_s^u-\mathbf{X}_s^v|^{2}\le 2H_1^{2}+2H_2^{2}$, we can derive
\begin{displaymath}
\mathbb{E}^{t,x}\left[\big|\mathbf{X}_s^u-\mathbf{X}_s^v\big|^{2}\right]
\le C\int_{t}^{s}\mathbb{E}^{t,x}\left[|\mathbf{X}_\tau^u-\mathbf{X}_\tau^v|^{2}\right]\,\mathrm{d}\tau
     + \frac{C}{\beta}\,\big\|u-v\big\|_{\beta}^{2}\,e^{2\beta(T-t)}\big(1+|x|\big)^{2}.
\end{displaymath}
Gr\"onwall's inequality then leads us to
\begin{displaymath}
\mathbb{E}^{t,x}\left[\big|\mathbf{X}_s^u-\mathbf{X}_s^v\big|^{2}\right]
\le \frac{C}{\beta}\,\big\|u-v\big\|_{\beta}^{2}\,e^{2\beta(T-t)}\big(1+|x|\big)^{2}.
\end{displaymath}
Taking square roots and using Jensen's inequality provides the desired estimate.
\end{proof}

Lemma~\ref{lem:tail} and Lemma~\ref{lem:head} collectively establish the following estimate.

\begin{lemma}\label{lem:X}
Let $u,v\in\mathbb{X}_\beta$ satisfy $|\nabla_{x}u|<M$ and $|\nabla^2_{x}u|<\frac{M}{\sqrt{T-t}}$. Then there exists a constant $C>0$ such that for all $t<s\leq T$,
\begin{displaymath}
\mathbb{E}^{t,x}\left[\big|\mathbf{X}_{s}^u-\mathbf{X}_{s}^v\big|\right]
\leq \frac{C}{\sqrt{\beta}}\,\big\|u-v\big\|_{\beta}\,e^{\beta(T-t)}\big(1+|x|\big).
\end{displaymath}
\end{lemma}

To establish the main theorem, we also require the following gradient estimate.
\begin{lemma}\label{lem:gradientE}
Set $e_n(t,x)=u_n(t,x)-u_{n-1}(t,x)$. Then $|\nabla_{x} e_{n+1}|_{\beta}\leq\frac{C}{\sqrt{\beta}}\,\|e_n\|_{\beta}$.
\end{lemma}

\begin{proof}
By the construction of the iteration, $u_{n+1}$ and $u_n$ satisfy the linear equations
\begin{equation}\label{eq:un+1}
\begin{cases}
\partial_t u_{n+1}+\mu_n^\top\,\nabla_{x}u_{n+1}+\frac12\mathrm{Tr}\Big(\sigma_n\sigma_n^\top\nabla^2_{x} u_{n+1}\Big)=f_n(t,x),\\
u_{n+1}(T,x)=g(x),
\end{cases}
\end{equation}
and
\begin{equation}\label{eq:un}
\begin{cases}
\partial_t u_n+\mu_{n-1}^\top\nabla_{x}u_n+\frac{1}{2}\mathrm{Tr}\Big(\sigma_{n-1}\sigma_{n-1}^\top\nabla^2_{x} u_n\Big)=f_{n-1}(t,x),\\
u_n(T,x)=g(x),
\end{cases}
\end{equation}
where $\mu_k(t,\mathbf{x})=\mu(t,\mathbf{x}, u_k(t,\mathbf{x}), \nabla_{x}u_k(t,\mathbf{x}))$, $\sigma_k(t,\mathbf{x})=\sigma(t,\mathbf{x}, u_k(t,\mathbf{x}))$, and $f_k(t,\mathbf{x})=f(t,\mathbf{x}, u_k(t,\mathbf{x}), \nabla_{x}u_k(t,\mathbf{x}))$ and all functions $u_k$ satisfy the uniform bounds of Lemma~\ref{lem:boundofgu} by induction.

Subtracting \eqref{eq:un} from \eqref{eq:un+1} and rearranging gives
\begin{equation}\label{eq:en}
\begin{cases}
\partial_t e_{n+1}+\mu_{n-1}^\top\nabla_{x} e_{n+1}+\frac{1}{2}\mathrm{Tr}\Big(\sigma_{n-1}\sigma_{n-1}^\top\nabla^2_{x}e_{n+1}\Big)
= R_{n+1}(t,x),\\
e_{n+1}(T,x) = 0,
\end{cases}
\end{equation}
where the right-hand side is given by $R_{n+1}=f_n-f_{n-1}-(\mu_n-\mu_{n-1})^\top\nabla_{x}u_{n+1}
-\frac{1}{2}\mathrm{Tr}\left[(\sigma_n\sigma_n^\top-\sigma_{n-1}\sigma_{n-1}^\top)\nabla^2_{x} u_{n+1}\right]$. Notice that only $u_n$ and $u_{n-1}$ appear on the right, which are already known to satisfy the bounds $|\nabla_{x}u_n|\le M$ and $|\nabla^2_{x} u_n|\le M/\sqrt{T-t}$ by induction hypothesis.

Let $\Gamma_n$ be the fundamental solution of the linear operator in \eqref{eq:en}. Then
\begin{displaymath}
e_{n+1}(t,x)=\int_{t}^{T}\int_{\mathbb{R}^d} \Gamma_n(t,x,\tau,y)\,R_{n+1}(\tau,y)\,\mathrm{d}y\,\mathrm{d}\tau.
\end{displaymath}
Using the Lipschitz assumptions on $f,\mu,\sigma$ and the bounds in Lemma~\ref{lem:boundofgu},
\begin{equation}\label{eq:Rn}
\begin{aligned}
\big|R_{n+1}\big|
  \le L\left(\big|e_n\big|+\left(\big|e_n\big|+\big|\nabla_x e_n\big|\right)\big|\nabla_{x}u_{n+1}\big|
  +\left(\big|\sigma_n\big|+\big|\sigma_{n-1}\big|\right)\big|e_n\big|\big|\nabla_x^2 u_{n+1}\big|\right)
  \le L\left(1+M+\frac{2\lambda M}{\sqrt{T-t}}\right)\big|e_n\big|+LM\,\big|\nabla_{x}e_n\big|,
\end{aligned}
\end{equation}
where $\lambda$  is a uniform bound for $|\sigma|$ (see Assumption~\textbf{A.3}). The gradient of the fundamental solution satisfies the Gaussian estimate
\begin{displaymath}
\big|\nabla_{x}\Gamma_n(t,x,\tau,y)\big|
\le \frac{C}{(\tau-t)^{(d+1)/2}}\exp\left(-\frac{c|x-y|^{2}}{\tau-t}\right).
\end{displaymath}
Inserting these bounds and \eqref{eq:Rn} into the differentiated representation of $e_{n+1}$ yields
\begin{equation}\label{eq:gradest}
\begin{aligned}
\big|\nabla_{x}e_{n+1}(t,x)\big|
&\le C\int_{t}^{T}\int_{\mathbb{R}^d}\frac{1}{(\tau-t)^{(d+1)/2}}
\left[\left(1+M+\frac{2\lambda M}{\sqrt{T-\tau}}\right)\big|e_n(\tau,y)\big|
+ M\big|\nabla_x e_n(\tau,y)\big|\right]
\exp\left(-\frac{c|x-y|^{2}}{\tau-t}\right)\,\mathrm{d}y\,\mathrm{d}\tau\\
&\le C\int_{t}^{T}\frac{e^{\beta(T-\tau)}}{\sqrt{\tau-t}}\,
\left[\Bigl(1+M+\frac{2\lambda M}{\sqrt{T-\tau}}\Bigr)|e_n|_{\beta}
+ M\big|\nabla_{x}e_n\big|_{\beta}\right]\int_{\mathbb{R}^d}\big(1+|y|\big)\,
\exp\left(-\frac{c|x-y|^{2}}{\tau-t}\right)\,\mathrm{d}y\,\mathrm{d}\tau.
\end{aligned}
\end{equation}
The inner integral is bounded by $C(\tau-t)^{d/2}(1+|x|)$.  Hence
\begin{displaymath}
\frac{|\nabla_{x}e_{n+1}(t,x)|}{1+|x|} e^{-\beta(T-t)}
\le C\int_{t}^{T}\frac{e^{-\beta(\tau-t)}}{\sqrt{\tau-t}}\,
\left(1+M+\frac{2\lambda M}{\sqrt{T-\tau}}\right)
\Big(\big|e_n\big|_{\beta}+\big|\nabla_{x}e_n\big|_{\beta}\Big)\,\mathrm{d}\tau.
\end{displaymath}
Because $1/\sqrt{T-\tau}$ dominates the parenthesis for small $T-\tau$, we can bound the whole factor by $C/\sqrt{T-\tau}$. Thus
\begin{displaymath}
\frac{|\nabla_{x}e_{n+1}(t,x)|}{1+|x|}e^{-\beta(T-t)}
\le C\int_{t}^{T}\frac{e^{-\beta(\tau-t)}}{\sqrt{(\tau-t)(T-\tau)}}\,\Big(\big|e_n\big|_{\beta}+\big|\nabla_{x} e_n\big|_{\beta}\Big)\,\mathrm{d}\tau.
\end{displaymath}
The remaining integral is bounded by $C/\sqrt{\beta}$ (using the substitution $\tau-t = (T-t)\sin^2\theta$). Hence,
\begin{displaymath}
\frac{|\nabla_{x}e_{n+1}(t,x)|}{1+|x|} e^{-\beta(T-t)}
\le \frac{C}{\sqrt{\beta}}\,\left(\big|e_n\big|_{\beta}+\big|\nabla_{x}e_n\big|_{\beta}\right)
\le \frac{C}{\sqrt{\beta}}\,\big|e_n\big|_{\beta}.
\end{displaymath}
Taking the supremum over $(t,x)$ completes the proof.
\end{proof}

\bibliographystyle{elsarticle-num}
\bibliography{References}



%\begin{thebibliography}{00}
%
%% For numbered reference style
%\bibitem{lamport94}
%  \textit{\LaTeX: a document preparation system},
%  Addison Wesley, Massachusetts,
%\end{thebibliography}

\end{document}